\documentclass{article}
\usepackage{makeidx}
\usepackage{amssymb}
\usepackage{amsfonts}
\usepackage{graphicx}
\usepackage{amsmath}
\usepackage{algpseudocode}
\usepackage{algorithmicx}
\usepackage{algorithm}
\usepackage{authblk}
\usepackage{url}

\usepackage{epstopdf} 

\usepackage{gensymb}
\usepackage{mathptmx}      

\newtheorem{theorem}{Theorem}

\newtheorem{proposition}[theorem]{Proposition}
\newtheorem{remark}[theorem]{Remark}

\newenvironment{proof}[1][Proof]{\textbf{#1.} }{\ \rule{0.5em}{0.5em}}
\newcommand{\Amp}{A_m^\prime}
\newcommand{\bex}{b^{\text{ex}}}
\newcommand{\bH}{H}

\newcommand{\dg}{\mathrm{diag}}
\newcommand{\eps}{\varepsilon}
\newcommand{\R}{\mathbb{R}}

\newcommand{\Smf}{\Sigma_m^f}
\newcommand{\K}{\mathcal{K}}

\newcommand{\hW}{\widehat{W}}
\newcommand{\mmax}{m_{\max}}

\newcommand{\sGM}{s^{\text{GMR}}}
\newcommand{\tql}{\textquotedblleft}
\newcommand{\tqr}{\textquotedblright}

\newcommand{\Umf}{U_m^f}
\newcommand{\VHgm}{V^{H_5}}
\newcommand{\VHhy}{V^{H_{40}}}
\newcommand{\VT}{V^{T_{5}}}
\newcommand{\WGM}{W^{\text{GMR}}}
\newcommand{\WCG}{W^{\text{CG}}}

\newcommand{\weps}{\widehat{\varepsilon}}

\newcommand{\Xex}{X^{\text{ex}}}
\newcommand{\xex}{x^{\text{ex}}}
\newcommand{\xGM}{x^{\text{GMR}}}

\newcommand{\xmk}{x_{m,k}}

\newcommand{\ymk}{y_{m,k}}%
\begin{document}

\title{Some transpose-free CG-like solvers\\for nonsymmetric ill-posed problems}


\author{Silvia Gazzola         and
        Paolo Novati 
}


\date{}

\maketitle

\begin{abstract}
This paper introduces and analyzes an original class of Krylov
subspace methods that provide 
an efficient alternative to many well-known
conjugate-gradient-like (CG-like) Krylov solvers for square nonsymmetric
linear systems arising from discretizations of inverse ill-posed problems.
The main idea underlying the new methods is to consider some rank-deficient
approximations of the transpose of the system matrix, obtained by running
the (transpose-free) Arnoldi algorithm, and then apply some Krylov
solvers to a formally right-preconditioned system of equations.
Theoretical insight is given, and many numerical tests show that the new
solvers outperform classical Arnoldi-based or CG-like methods in a variety
of situations.
\end{abstract}

\section{Introduction}\label{sect:intro}

Let us consider a linear system of the form
\begin{equation}
Ax=b\,,\quad \mbox{where}\quad A\in \mathbb{R}^{N\times N},  \label{sys}
\end{equation}%
coming from a suitable discretization of an inverse ill-posed problem. In
this setting, the matrix $A$ typically has ill-determined rank, i.e., when
considering the singular value decomposition (SVD) of $A$, given by
$A = U\Sigma V^T$, with $\Sigma=\dg(\sigma_1,\dots,\sigma_N)$, the singular values $\sigma_{i}\geq\sigma_{i+1}>0$, $i=1,\dots,N-1$, 
quickly decay and cluster at zero with no evident gap between two consecutive ones to indicate numerical rank. In particular, $A$ is
ill-conditioned. Moreover, the right-hand side vector in (\ref{sys}) is
typically affected by some unknown noise $e$, i.e., $b=b^{\text{ex}}+e$, where $\bex$ is the unknown exact version of $b$. Our
goal is to compute a meaningful approximation of the solution $x^{\text{ex}}$
of the unknown noise-free linear system $Ax^{\text{ex}}=b^{\text{ex}}$ and,
because of the ill-conditioning of $A$ and the presence of the noise $e$,
some kind of regularization should be applied to the available system (\ref{sys}) (see \cite{PCH98} for an
overview). Truncated SVD (TSVD) is a well-established regularization method, which consists in replacing (\ref{sys}) by the least square problem
$\min_{x\in\R^N}\|A_mx-b\|$, where
\begin{equation}\label{TSVD}
A_m=U_m^{A}\Sigma_m^{A}(V_m^{A})^T,\quad U_m^{A}\in\R^{N\times m},\;\Sigma_m^{A}\in\R^{m\times m},\;V_m^{A}\in\R^{N\times m},
\end{equation}
is the best rank-$m$ approximation of $A$ in the matrix 2-norm. Here $U_m^{A}$ and $V_m^{A}$ are obtained by taking the first $m$ left and right singular vectors of $A$, respectively (i.e., the first $m$ columns of $U$ and $V$, respectively), and $\Sigma_m^{A}$ is the diagonal matrix of the first $m$ singular values of $A$.
Since the SVD is typically needed to define (\ref{TSVD}), and because of its high computational cost, TSVD is not suitable to regularize large-scale and unstructured problems. Therefore, in this paper, we are particularly interested in iterative
regularization methods, which compute an approximation of $x^{\text{ex}}$ by
leveraging the so-called \textquotedblleft
semi-convergence\textquotedblright\ phenomenon, so that regularization is achieved by an early termination
of the iterations. Iterative regularization methods typically require one matrix-vector product with $A$ and/or $A^T$ at each iteration, and therefore they can also be employed when the coefficient matrix $A$ and/or $A^T$ is not explicitly available.
Many Krylov subspace methods have been proven to be efficient
iterative regularization methods, as they typically show good accuracy
together with a fast initial convergence (see \cite{Survey} and the
references therein).

Given a matrix $C\in \mathbb{R}^{N\times N}$ and a vector $d\in
\mathbb{R}^{N}$, the Krylov subspace $\mathcal{K}_{m}(C,d)$ is defined as
\begin{equation*}
\mathcal{K}_{m}(C,d)=\mathrm{span}\{d,\,Cd,\dots ,C^{m-1}d\}\subseteq
\mathbb{R}^{N}.
\end{equation*}%
Here and in the following we assume that the dimension of $\mathcal{K}%
_{m}(C,d)$ is $m$. A Krylov subspace method is a projection method onto Krylov
subspaces, i.e., at the $m$th iteration, an approximate solution $x_{m}$ for
(\ref{sys}) is computed by imposing the conditions
\begin{equation}
x_{m}\in \mathcal{K}_{m}(C_{1},d_{1})\quad \mbox{and}\quad
r_{m}=b-Ax_{m}\perp \mathcal{K}_{m}(C_{2},d_{2}),  \label{projKry}
\end{equation}%
where, without loss of generality, the initial guess $x_{0}=0$ is assumed. Different Krylov subspace methods are obtained by varying the choice of $%
\mathcal{K}_{m}(C_{1},d_{1})$ and $\mathcal{K}_{m}(C_{2},d_{2})$ in (\ref%
{projKry}) (see \cite[Chapter 6]{Saad}), which also depends on the properties of the system at hand. A common computational approach to Krylov subspace methods consists in
generating an orthonormal basis $\{w_{1},\dots ,w_{m}\}$ for $\mathcal{K}%
_{m}(C_{1},d_{1})$ (i.e., a new vector \linebreak[4]$w_{i}\in \mathbb{R}^{N}$ is added at
the $i$th iteration, $i=1,\dots ,m$), and this can be typically achieved by
applying the Arnoldi algorithm (see \cite[\S 6.3]{Saad}). In particular,
when $C_{1}$ is symmetric, the Arnoldi algorithm greatly simplifies and
coincides with the so-called symmetric Lanczos algorithm, which is
associated to three-term recurrence formulas: this is the case for popular methods like CG, CGLS (or the mathematically equivalent LSQR), CGNE, and MINRES.

While the theoretical regularizing properties and the performance of some
CG-like solvers, such as CG, CGLS, and CGNE, are well-understood, and these
methods are well-established regularization methods (see, for instance, \cite{Han95,LancReg,PCH98,JH07}),
the same is not true for the methods based on the
Arnoldi algorithm. The authors of \cite{CLR02-2} prove that, under some assumptions on $\bex$, GMRES equipped with a stopping rule based on the discrepancy principle is a regularization method in the classical sense, meaning that $x_m$ tends to $\xex$ as the noise $e$ tends to $0$. However, recent analysis (see \cite{JH07} and the references therein) suggests that GMRES
(or even its range-restricted variant \cite{RRGMRESfirst}) might fail if the problem is highly
non-normal, as in this case the SVD components of the matrix $A$ are severely mixed in the GMRES approximate solutions;
moreover, the approximation subspace generated by GMRES may fail to reproduce the relevant components of the solution $\xex$. These situations arise, for instance, when considering image deblurring problems characterized by a highly non-symmetric blur. 

We should stress that, among all the Krylov methods mentioned so far, GMRES is the only one that can handle a nonsymmetric linear system (\ref{sys}) when $A^T$ is unavailable, or when matrix-vector multiplications with $A^T$ are impractical to compute. For this reason, it is important to investigate ways of overcoming the shortcomings of GMRES, e.g., by defining a more appropriate approximation subspace for GMRES. A common way of achieving this is to incorporate some sort of \tql preconditioning\tqr. For instance, the authors of \cite{SN} propose to incorporate into GMRES a \tql smoothing-norm preconditioner\tqr, which can enforce some additional regularity into the solution (achieving an effect similar to Tikhonov regularization in general form, as opposed to standard form). The authors of \cite{DMR15} propose to incorporate into GMRES a \tql reblurring preconditioner\tqr\ $A^\prime$, which approximates $A^T$ and is tailored for particular image deblurring problems: by doing so, the original system (\ref{sys}) is replaced by an equivalent one, whose coefficient matrix $AA^\prime$ (or $A^\prime A$) well approximates $AA^T$ (or $A^TA$, respectively), so that the problem is somewhat symmetrized. We emphasize that, here and in the following, the term \tql preconditioner\tqr\ is not used in a classical sense: indeed, these \textquotedblleft preconditioners\textquotedblright\ do not accelerate the \tql convergence\tqr\ of GMRES, but rather enforce some desirable properties into the solution subspace (and, in doing so, sometimes they may actually require more iterations than standard GMRES). Even the concept of \tql convergence\tqr\ is not well defined in this setting, as early termination must be considered in order to regularize (\ref{sys}).

This paper proposes an efficient and reliable strategy to symmetrize the coefficient matrix of system (\ref{sys}).
More precisely, after the Krylov subspace $\mathcal{K}_{m}(A,b)$ is
generated by performing $m$ iterations of the Arnoldi algorithm
applied to (\ref{sys}), which only involves $m$ matrix-vector products with the
matrix $A$, a rank-$m$ matrix $A_{m}^{\prime }\in \mathbb{R}^{N\times N}$
is formed by exploiting the quantities computed by the Arnoldi algorithm, in such a way that $AA_{m}^{\prime }$ is symmetric semi-positive definite. The
original system (\ref{sys}) is then replaced by the following \tql preconditioned\tqr, symmetric, rank-deficient problem, to be solved in the least squares sense
\begin{equation}\label{symmsys}
y_m = \arg\min_{y\in\R^N}\|AA_{m}^{\prime }y-b\|\,,\quad\mbox{with}\quad x_{m}=A_{m}^{\prime }y_m\,.
\end{equation}
Here and in the following, $\Vert \cdot \Vert $ denotes the vectorial 2-norm or the induced matrix or operator 2-norm. Since in many situations $A_{m}^{\prime }$ is a good approximation of $A_{m}^{T}$, where $A_{m}$ is defined as in (\ref{TSVD}), one can regard the system (\ref{symmsys}) as a rank-deficient
symmetric version of (\ref{sys}), which is also a good approximation of the normal equations $AA^{T}y=b$ associated to (\ref{sys}), with $x=A^{T}y$. Furthermore, one can easily see that taking $\Amp=A_m^{T}$ and solving problem (\ref{symmsys}) is equivalent to compute a TSVD solution. Therefore, problem (\ref{symmsys}) can also be regarded to as a regularized version of (\ref{sys}). The least squares problem (\ref{symmsys}) can then be solved directly (by truncated SVD) or iteratively (by a variety of Krylov subspace methods).
Depending on the choice of the solver, transpose-free CGLS-like and transpose-free CGNE-like methods can be defined, whose accuracy will depend
on the choice of $m$.
We stress that the strategy for defining $A_{m}^{\prime }$ presented in this paper is independent of the
application at hand.
We also remark that, as we shall see, once the Arnoldi
algorithm is run to compute $A_{m}^{\prime }$, the computational cost for
solving (\ref{symmsys}) is negligible, as the computations can be arranged
in such a way that only matrices of order $m$ are involved: therefore, the
overall cost of the new methods is essentially the cost of performing $m$ iterations of the Arnoldi algorithm. Moreover, provided that $m$ is sufficiently small, the memory requirements of the new methods are not demanding.

This paper is organized as follows: Section \ref{sect:trfreeArn} surveys
some known properties of the Arnoldi algorithm, introduces the matrix $%
A_{m}^{\prime }$ appearing in (\ref{symmsys}), and derives some insightful theoretical results. Section %
\ref{sect:algo} describes different algorithmic approaches for the solution
of (\ref{symmsys}): some computational details are unfolded, and connections
with CGLS and CGNE are explored. Section \ref{sect:numexp} displays the
results of many numerical experiments, which compare the performance of the
new class of solvers for (\ref{symmsys}) with traditional Krylov methods for
(\ref{sys}). Finally, Section \ref{sect:final} draws some concluding remarks.

\section{A transpose-free \textquotedblleft symmetrization\textquotedblright\ of the Arnoldi algorithm}
\label{sect:trfreeArn}

The Arnoldi algorithm is a process for building an orthonormal basis of the
Krylov subspace $\mathcal{K}_{m}(A,b)$: $m$ steps of the Arnoldi algorithm
lead to the following matrix decomposition
\begin{equation}\label{ArnoldiDec}
AW_{m} =W_{m+1}H_{m}\,,
\end{equation}%
where $W_{m}=[w_{1},\dots ,w_{m}]\in \mathbb{R}^{N\times m}$ has orthonormal
columns that span $\mathcal{K}_{m}(A,b)$, and \linebreak[4]$H_{m}\in \mathbb{R}^{(m+1)\times
m}$ is an upper Hessenberg matrix. Moreover,
\begin{equation}
w_{1}=\frac{b}{\Vert b\Vert }\,,\quad \mbox{and}\quad W_{m+1}=[W_{m}\;w_{m+1}]\in \mathbb{R}%
^{N\times (m+1)}\,. \label{projrhs}
\end{equation}%
As said in the Introduction, in this paper we assume $m$
to be sufficiently small, so that $\mathcal{K}_{m}(A,b)$ is of dimension $m$
and decomposition (\ref{ArnoldiDec}) exists (i.e., the sub-diagonal elements $h_{j+1,j}$, $j=1,\dots,m-1$ of $H_m$ do not vanish, so that no breakdown occurs in the Arnoldi algorithm). If $A$ is symmetric, the Arnoldi decomposition (\ref{ArnoldiDec}%
) reduces to the symmetric Lanczos decomposition, where the matrix $H_{m}$
is tridiagonal.

GMRES is arguably the most popular Krylov method based on the Arnoldi
algorithm. At the $m$th step of GMRES (see \cite[\S 6.5]{Saad}), one updates the decomposition (\ref%
{ArnoldiDec}), and an approximation $\xGM_{m}$ of the solution of the original
linear system is obtained by taking
\begin{equation}
\xGM_{m}=W_{m}\sGM_{m}\,,\quad \mbox{where}\quad \sGM_{m}=\arg \min_{s\in
\mathbb{R}^{m}}\left\Vert H_{m}s-\Vert b\Vert e_{1}\right\Vert \,,
\label{GMRES}
\end{equation}%
where $e_1$ is the first canonical basis vector of $\R^{m+1}$. Thanks to (\ref{ArnoldiDec}), $\xGM_m$ enjoys the optimality property
\begin{equation}\label{optGM}
\xGM_{m}=\arg \min_{x_{m}\in \mathcal{K}_{m}(A,b)}\left\Vert
Ax_{m}-b\right\Vert .
\end{equation}

Let us now assume that $m$ steps of the Arnoldi algorithm are performed, so that
the quantities in (\ref{ArnoldiDec}) are computed, and let us consider the
right-preconditioned system (\ref{symmsys}), where $A_{m}^{\prime }$ is
defined as
\begin{equation}\label{Aprime}
A_{m}^{\prime }=W_{m}H_{m}^{T}W_{m+1}^{T}\in \mathbb{R}^{N\times N}\,.
\end{equation}%
The rank of $A_{m}^{\prime }$ is $m$, as it can be easily seen from (\ref%
{ArnoldiDec}). Exploiting once again relation (\ref{ArnoldiDec}), one realizes that
\begin{equation}
AA_{m}^{\prime }=AW_{m}H_{m}^{T}W_{m+1}^{T}=W_{m+1}H_{m}H%
_{m}^{T}W_{m+1}^{T}=C_{m}C_{m}^{T}\,,  \label{SSPD}
\end{equation}%
where $C_{m}=W_{m+1}H_{m}\in \mathbb{R}^{N\times m}$.
Therefore, the least square problem (\ref{symmsys}) can be reformulated as
\begin{equation}\label{newProjPb}
y_{m}=\arg \min_{y\in \mathbb{R}^{N}}\left\Vert C_{m}C_{m}^{T}y-b\right\Vert
\,,\quad\mbox{with}\quad x_m=\Amp y_m.
\end{equation}%
Directly from definition (\ref%
{Aprime}), and recalling that $\mathrm{range}(W_{m})=\mathcal{K}_{m}(A,b)$, one can immediately see that $x_m\in \mathcal{K}_{m}(A,b)$, as computed in (\ref{newProjPb}). Therefore, by (\ref{optGM}), one has
\begin{equation}\label{optGMres}
\left\Vert A\xGM_{m}-b\right\Vert \leq \left\Vert
Ax_{m}-b\right\Vert .
\end{equation}
The following proposition sheds light on the links between the solutions of problems (\ref{sys}) and (\ref{newProjPb}).

\begin{proposition}\label{prop:sols}
Let $y_m\in\R^N$ be a solution of $C_{m}C_{m}^{T}y=b$.
Then $x_m=W_ms_m\in\R^N$ solves $Ax=b$, where $s_m=\bH_m^TW_{m+1}^Ty_m\in\R^m$. Conversely, let $x_m=W_{m}s_m$ be the solution of (\ref{sys}), where $s_m\in\R^m$. Then the system
\begin{equation}\label{defz}
\bH_{m}\bH_{m}^{T}t=\|b\|e_1\,
\end{equation}
has a solution $t_m\in\R^{m+1}$, and $y_m=W_{m+1}t_m\in\R^N$ is the minimal norm solution of $C_mC_m^Ty=b$.
\end{proposition}

\begin{proof}
The first part obviously follows from (\ref{SSPD}), as
\[
b = C_mC_m^Ty_m = A\Amp y_m = AW_m\bH_m^T W_{m+1}^Ty_m = AW_ms_m = Ax_m\,.
\]
To prove the second part, one should first consider the Arnoldi decomposition (\ref{ArnoldiDec}), so that
\[
b = Ax_m = AW_ms_m = W_{m+1}\bH_ms_m\,,
\]
and, thanks to the first equality in (\ref{projrhs}),
\begin{equation}\label{ArnSolution}
\bH_ms_m = \|b\|e_1\,.
\end{equation}
Now, consider the economy-size SVD of $\bH_{m}$, given by
\begin{equation}\label{svdH}
H_m=U_m\Sigma_mV_m^T\,,\quad\mbox{where}\quad U_m\in\R^{(m+1)\times m},\;\Sigma_m\in\R^{m\times m},\; V_m\in\R^{m\times m},
\end{equation}
and the associated full-size SVD, given by
\[
H_m=\Umf\Smf V_m^T\,,\quad\mbox{where}\quad \Umf=[U_m\; u_{m+1}]\in\R^{(m+1)\times (m+1)},\;\Smf=\left[\begin{array}{c}
\Sigma_m\\
0
\end{array}\right]\in\R^{(m+1)\times m}.
\]
Note that (\ref{ArnSolution}) holds if and only if
\[
\Umf\Smf V_m^T s_m = \Umf(\Umf)^T(\|b\|e_1)\,,
\]
which is equivalent to asking the last component of $(\Umf)^Te_1$ (i.e., $u_{m+1}^Te_1$) to be zero. Then there exists a solution $t_m\in\R^{m+1}$ of (\ref{defz}), as
\[
\Umf\Smf(\Smf)^T\underbrace{(\Umf)^Tt}_{\widehat{t}\,\in\,\R^{m+1}}=\|b\|e_1\quad\mbox{implies}\quad
\left[\begin{array}{cc}
\Sigma_m^2 & \\
& 0\end{array}\right]
\widehat{t}=
\left[\begin{array}{c}
U_m^T(\|b\|e_1)\\
0\end{array}\right]\,.
\]
At this point, each $y$ such that $W_{m+1}^{T}y=t_m$ satisfies
\[
\bH_{m}\bH_{m}^{T}W_{m+1}^{T}y =\|b\|e_1\,.
\]
By multiplying both terms by $W_{m+1}$ from the left, and by exploiting the first equality in (\ref{projrhs}), one obtains $W_{m+1}\bH_{m}\bH_{m}^{T}W_{m+1}^{T}y =b$, which, thanks to (\ref{SSPD}), can be rewritten as
\begin{equation}\label{minnorm}
C_mC_m^Ty=b\,.
\end{equation}
Therefore, the minimum norm solution $y_m$ of (\ref{minnorm}) satisfies $W_{m+1}^Ty=t_m$, and is obtained by computing the solution $\tilde{t}_m\in\R^{m+1}$ of $W_{m+1}^TW_{m+1}\tilde{t}=t_m$ and taking $y_m = W_{m+1}\tilde{t}_m$. Since $W_{m+1}^TW_{m+1}=I$, $\tilde{t}_m = t_m$, and $y_m=W_{m+1}t_m$.
\end{proof}

Proposition \ref{prop:sols} essentially states that solving (\ref{sys}) by an Arnoldi-based method is equivalent to solve (\ref{newProjPb}). More specifically,  whenever the solution of (\ref{sys}) can be computed by performing $m$ steps of a solver based on the Arnoldi algorithm (such as GMRES), a minimal
norm solution of (\ref{newProjPb}) can be recovered by solving the projected
symmetric semi-positive definite system (\ref{defz}). However, as explained
in Section \ref{sect:intro}, when dealing with ill-posed systems one is not interested in fully solving (\ref{sys}) and (\ref{newProjPb}),
and an iterative solver should be stopped reasonably early. Because of this,
in the next section we will derive a variety of approaches for regularizing
problem (\ref{newProjPb}).
We also remark that, as emphasized in \cite{JH07}, the performance of GMRES as a regularization method can sometimes be unsatisfactory, due to the unwanted mixing of the SVD components of $A$ in the GMRES approximation subspaces and, therefore, in the GMRES approximate solutions. Since the approximation subspaces for GMRES and for any method applied to (\ref{newProjPb}) coincide, one may suspect the approximate solutions of (\ref{newProjPb}) to be affected by the same issue. As we shall see in the next section, the SVD mixing is somewhat damped in (\ref{newProjPb}), depending on the chosen solver.
In the remaining part of this section we provide some
motivations underlying the choice of (\ref{Aprime}), which are connected to
the regularizing properties of the Arnoldi algorithm.

Define
\begin{equation}
\widehat{U}_{m}=W_{m+1}U_{m}=[\widehat{u}_{1},\dots ,\widehat{u}%
_{m}]\in \mathbb{R}^{N\times m}\quad \mbox{and}\quad \widehat{V}_{m}=W_{m}%
V_{m}=[\widehat{v}_{1},\dots ,\widehat{v}_{m}]\in \mathbb{R}^{N\times
m}\,,  \label{svdhfull}
\end{equation}%
where $U_m$ and $V_m$ are the matrices of the left and right singular vectors of $H_m$ (\ref{svdH}), respectively, and define
\begin{equation}
\widehat{A}_{m}=W_{m+1}H_{m}W_{m}^{T}=\widehat{U}_{m}\Sigma_{m}%
\widehat{V}_{m}^{T}\quad \text{(note that $\widehat{A}_{m}=(A_{m}^{\prime
})^{T}$).}  \label{Akry}
\end{equation}%
One can easily show that the (T)SVD of $\widehat{A}_{m}$ is given by $%
\widehat{U}_{m}\Sigma_{m}\widehat{V}_{m}^{T}$, and that the
Moore-Penrose pseudo-inverse $\widehat{A}_{m}^{\dagger }$ of $\widehat{A}%
_{m} $ is the regularized inverse (as defined in \cite[\S 4.4]{PCH98})
associated to the $m$th iteration of GMRES. Indeed, by exploiting relation (\ref{GMRES}), the first equality in (\ref{projrhs}), and the TSVD (\ref{Akry}), one can write
\begin{equation*}
\xGM_{m}=W_{m}\sGM_{m}=W_{m}H_{m}^{\dagger }(\Vert b\Vert
e_{1})=W_{m}H_{m}^{\dagger }W_{m+1}^{T}b=\widehat{V}_{m}\Sigma_{m}^{-1}\widehat{U}_{m}^{T}b=\widehat{A}_{m}^{\dagger }b\,.
\end{equation*}
In order for a (generic) regularization method to be successful, the
regularized matrix should contain information about the dominant singular
values of the original matrix $A$, and filter out the influence of the small
ones. If $\widehat{A}_{m}$ is a good regularized approximation of $A$, then
using $\widehat{A}_{m}^{T}=\Amp$ to approximate $A^{T}$ is a meaningful choice.

If $A$ is severely ill-conditioned, the authors of \cite{Survey,NR14}
numerically show that $\widehat{A}_{m}$ quickly inherits the spectral
properties of $A$. In particular, the following relations hold for $%
k=1,\dots ,m$%
\begin{eqnarray}
A\widehat{v}_{k}-{\sigma}_{k}^{(m)}\widehat{u}_{k} &=&0\,,  \label{s1} \\
W_{m}^{T}(A^{T}\widehat{u}_{k}-{\sigma}_{k}^{(m)}\widehat{v}_{k}) &=&0\,,
\notag
\end{eqnarray}%
where $\sigma_k^{(m)}$, $k\leq m$, is the $k$th singular value of $H_m$. Moreover, working in a continuous setting and under the hypothesis that $A$ is a Hilbert-Schmidt operator of
infinite rank (see \cite[Chapter 2]{Ri} for a background), whose singular values form a $\ell _{2}$ sequence, in \cite{N} it has been shown
that
\begin{equation}
\left\Vert A^{T}\widehat{u}_{k}-{\sigma}_{k}^{(m)}\widehat{v}%
_{k}\right\Vert \rightarrow 0\quad \text{as}\quad m\rightarrow \infty ,  \label{s2}
\end{equation}%
where the decay rate is closely connected to the decay rate of the singular
values of $A$. This property is inherited by the discrete case
whenever $A$ is a suitable discretization of a Hilbert-Schmidt operator.
Note that this class of operators includes Fredholm intergral operators of the first
kind with $L_{2}$ kernels. As a consequence of (\ref{s1}) and (\ref{s2}), in many relevant situations the dominant singular values of $A$
are well approximated by the singular values of $H_{m}$ (see \cite%
{Survey} for many numerical examples). Therefore, $\widehat{A}_{m}$ as defined in (\ref{Akry}) may represent a good
regularized approximation of $A$ for a variety of problems.

\section{Solving the \tql preconditioned\tqr\ problems}\label{sect:algo}

This section proposes two different iterative techniques to solve the rank-deficient symmetric least squares problem (\ref{newProjPb}), and therefore to compute a regularized solution of (\ref{sys}). Thanks to the definition of $C_{m}$, decomposition (\ref{ArnoldiDec}),
and Proposition \ref{prop:sols}, one can rewrite (\ref{newProjPb}) as
\begin{equation}
y_{m}=\arg \min_{y\in \mathbb{R}^{N}}\left\Vert C_{m}C_{m}^{T}y-b\right\Vert
=W_{m+1}\arg \min_{t\in \mathbb{R}^{m+1}}\left\Vert H_{m}H%
_{m}^{T}t-\Vert b\Vert e_{1}\right\Vert =W_{m+1}t_{m}.  \label{newProjPbRW}
\end{equation}%
By using the above reformulation, it is clear that solving system (\ref{newProjPb}) does not require a significant computational overload with
respect to solving system (\ref{sys}) by any standard Arnoldi-based method
(such as GMRES). Indeed, once $m$ iterations of the Arnoldi algorithm have been performed, with $m\ll N$, all the additional computations for getting (\ref{newProjPbRW}) are executed in dimension $m$, so that the
computational cost of any algorithm for (\ref{newProjPbRW}) is dominated by the cost of the Arnoldi algorithm. Even the storage requirement of any algorithm for the solution of (\ref{newProjPb}) (or (\ref{newProjPbRW})) is dominated by the storage requirement of the Arnoldi algorithm: namely, the cost of storing the matrix $W_{m+1}\in\R^{N\times (m+1)}$. Indeed, the rank-$m$ preconditioner $\Amp$ (\ref{Aprime}) can be stored in factored form, in order to recover $x_m$ (\ref{newProjPb}). Moreover, the residual associated to (\ref{sys}) can be conveniently
monitored in reduced dimension, as
\begin{equation*}
\Vert b-Ax_{m}\Vert =\Vert b-AA_{m}^{\prime }y_{m}\Vert =\Vert
b-C_{m}C_{m}^{T}y_{m}\Vert =\Vert \Vert b\Vert e_{1}-H_{m}H%
_{m}^{T}t_{m}\Vert .
\end{equation*}%
%
%
%
Since the starting vector $b$ of Krylov subspaces generated by the Arnoldi
algorithm (\ref{ArnoldiDec}), (\ref{projrhs}) is affected by some noise, noisy components are
retained in $H_{m}$ and $W_{m}$, so that the vector $t_{m}$ in (\ref%
{newProjPbRW}) should be computed by applying some regularization to the
(noisy) projected problem
\begin{equation}
\min_{t\in \mathbb{R}^{m+1}}\left\Vert H_{m}H_{m}^{T}t-\Vert
b\Vert e_{1}\right\Vert .  \label{sp}
\end{equation}%
Of course the noise propagation may be somehow damped by working with a
range-restricted approach that consists in using $Ab$ instead of $b$ as
starting vector for the Arnoldi process \cite{RRGMRESfirst}. We remark that the theory developed
in the present paper can be easily rearranged to work in this setting.

Direct methods such as Tikhonov regularization or TSVD can be easily applied
to (\ref{sp}), the latter being particularly meaningful because $%
H_{m}H_{m}^{T}$ is rank-deficient. However, in this paper, we
are interested in using an iterative approach for solving (\ref{newProjPb})
or (\ref{sp}), once the dimension $m$ has been fixed.

\subsection{A transpose-free CGLS-like method}
\label{sect:tfCGLS}

Consider computing an approximation $y_{m,k}$ of $y_{m}$ in (\ref%
{newProjPb}) by applying $k$ iterations of the MINRES method. This is equivalent to requiring
\begin{equation}
y_{m,k}\in \mathcal{K}_{k}(C_{m}C_{m}^{T},b)\,,\quad \mbox{and}\quad
b-C_{m}C_{m}^{T}y_{m,k}\,\perp\, (C_{m}C_{m}^{T})\mathcal{K}%
_{k}(C_{m}C_{m}^{T},b)\,,\quad k\leq m\,.  \label{condGMRES}
\end{equation}%
By definition, and by exploiting the Arnoldi algorithm (\ref{ArnoldiDec}),
we rewrite
\begin{equation}
C_{m}C_{m}^{T}=W_{m+1}H_{m}H_{m}^{T}W_{m+1}^{T}=A\underbrace{%
W_{m}W_{m}^{T}}_{=P_{m}}A^{T},  \label{projMatr}
\end{equation}%
where $P_{m}$ is the orthogonal projection onto $\mathcal{K}_{m}(A,b)$.
The first condition in (\ref{condGMRES}), together with (\ref{ArnoldiDec}) and the above relation, implies
\begin{equation}\label{xTFCGLS}
x_{m,k}=W_mH_m^TW_{m+1}^T\ymk=W_mW_m^TA^T\ymk = P_mA^T\ymk\,,
\end{equation}%
so that
\[
\xmk\in P_{m}A^{T}\mathcal{K}_{k}(AP_{m}A^{T},b)=\mathcal{K}_{k}(P_{m}A^{T}A,P_{m}A^{T}b)\,.
\]
Similarly, the second condition in (\ref{condGMRES}) implies
\begin{equation*}
b-AP_{m}A^{T}y_{m,k}\,\perp \,AP_{m}A^{T}\mathcal{K}_{k}(AP_{m}A^{T},b)\,,
\end{equation*}%
and, thanks to (\ref{xTFCGLS}), it can be equivalently rewritten as
\begin{equation*}
b-Ax_{m,k}\,\perp \,AP_{m}A^{T}\mathcal{K}_{k}(AP_{m}A^{T},b)=A\mathcal{K}%
_{k}(P_{m}A^{T}A,P_{m}A^{T}b)\,.
\end{equation*}%
We can summarize the above arguments in the following

\begin{proposition}
\label{prop:tfCGLS} For any given $m\geq 1$ the sequence $\left\{
x_{m,k}\right\} _{k\leq m}$ obtained by applying $k$ steps of the MINRES
method to problem (\ref{newProjPb}) is the result of a Krylov method defined
by
\begin{equation}
x_{m,k}\in \mathcal{K}_{k}(P_{m}A^{T}A,P_{m}A^{T}b)\quad \mbox{and}\quad
b-Ax_{m,k}\perp A\mathcal{K}_{k}(P_{m}A^{T}A,P_{m}A^{T}b)\,.  \label{tfCGLS}
\end{equation}
\end{proposition}

The above proposition has two important consequences. Firstly, thanks to a
well-known characterization of projection methods (see \cite[\S 5.2]{Saad}),
the residual $b-Ax_{m,k}$ in (\ref{tfCGLS}) has minimal norm among all the
residuals $b-A\widehat{x}_{m,k}$, with \linebreak[4]$\widehat{x}_{m,k}\in \mathcal{K}_{k}(P_{m}A^{T}A,P_{m}A^{T}b)$.
Secondly, recall that, through an implicit construction of the Krylov subspaces $\mathcal{K}_{k}(A^{T}A,A^{T}b)$,
CGLS generates a sequence of approximate solutions $\left\{ x_{k}^{\text{CGLS}}\right\} _{k\geq 1}$ of (\ref{sys}) such that
\begin{equation}
x_{k}^{\text{CGLS}}\in \mathcal{K}_{k}(A^{T}A,A^{T}b)\quad \mbox{and}\quad
b-Ax_{k}^{\text{CGLS}}\perp A\mathcal{K}_{k}(A^{T}A,A^{T}b)\,.  \label{CGLS}
\end{equation}%
Instead of the approximation subspace $\mathcal{K}_{m}(A^{T}A,A^{T}b)$ considered in (\ref{CGLS}), method (\ref{tfCGLS}) implicitly builds a Krylov subspace where the action of $A^{T}$ is replaced by its projection $P_mA^T$ onto $\mathcal{K}_{m}(A,b)$. Therefore, if $P_m=I$, conditions (\ref{tfCGLS}) and (\ref{CGLS}) are equivalent. In this sense, the new method (\ref{tfCGLS}) can be regarded as a transpose-free variant of a CGLS-like method, and from now on it will be simply referred
to as TF-CGLS; correspondingly, the vector $x_{m,k}$ in (\ref{tfCGLS}) will be denoted as $x_{m,k}^{\text{LS}}$.
 \begin{remark}\label{rem:AdvDraw}\emph{
 The TF-CGLS method has two clear advantages over CGLS: it does not require knowledge of $A^{T}$, and it
basically needs only one matrix-vector multiplication with $A$ at each step (to initially generate $W_{m+1}$ and $H_m$), as
opposed to one matrix-vector product with $A$ and one matrix-vector product
with $A^{T}$ at each step of CGLS. Indeed, the additional $k$ MINRES iterations required by TF-CGLS to compute the solution of (\ref{newProjPbRW}) can be performed on the projected problem (\ref{sp}) of order $m+1$. Of course each
approximate solution $\{x_{m,k}\}_{k\leq m}$ belongs to $\mathcal{K}%
_{m}(A,b)$ (directly by (\ref{tfCGLS}) and by the definition of $P_{m}$ in (\ref{projMatr})).
However, if $\mathcal{K}_{m}(A,b)$ well captures the features of the
solution that we wish to recover, then multiplication by $P_{m}$ does not spoil the approximation
subspace. Provided that a meaningful regularized solution can be recovered by TSVD (i.e., the columns of $V_m^A$ are a good basis for a regularized solution), this is eventually equivalent to requiring that $\widehat{A}_m$ in (\ref{Akry}) inherits the spectral properties of $A$ (see relations (\ref{s1}) and (\ref{s2})).
}
\end{remark}
%
%
\begin{remark}\label{rem:hybrid}\emph{
Hybrid regularization methods \cite{O'LS} consider additional direct regularization (such as TSVD) within each iteration of a regularizing iterative method. We claim that TF-CGLS can be somewhat regarded as a hybrid regularization method. Indeed, considering the first relation in (\ref{condGMRES}) and exploiting (\ref{projMatr}), one can straightforwardly rewrite
\[
\ymk\in\K_k(C_mC_m^T,b)=W_{m+1}\K_k(H_mH_m^T,\|b\|e_1)\,,
\]
so that
\[
\xmk\in(W_mH_m^TW_{m+1}^T)W_{m+1}\K_k(H_mH_m^T,\|b\|e_1)=W_m\K_k(H_m^TH_m,H_m^T\|b\|e_1)\,,
\]
or, equivalently,
\begin{equation}\label{hybr1}
\xmk = W_mt_k\,,\quad\mbox{where}\quad t_k\in\K_k(H_m^TH_m,H_m^T\|b\|e_1)\,.
\end{equation}
Analogously, considering the second relation in (\ref{condGMRES}) and exploiting (\ref{ArnoldiDec}) and (\ref{projMatr}), one gets
\[
b-C_mC_m^T\ymk\,\perp\,(C_mC_m^T)\K_k(C_mC_m^T,b)=W_{m+1}H_m\K_k(H_m^TH_m,H_m^T\|b\|e_1)\,,
\]
so that
\[
\|b\|e_1-H_mH_m^TW_{m+1}^T\ymk\,\perp\, \K_k(H_mH_m^T,H_mH_m^T\|b\|e_1)\,.
\]
Recalling that $W_m^T\xmk = H_m^TW_{m+1}^T\ymk$ (directly form (\ref{newProjPb})) and the definition of $t_k$ in (\ref{hybr1}), one gets
\[
\|b\|e_1-H_mt_k\,\perp\, \K_k(H_mH_m^T,H_mH_m^T\|b\|e_1)\,.
\]
Therefore, the vector $t_k$ is obtained by applying $k$ steps of the CGLS method to the projected LS problem (\ref{GMRES}) associated to the GMRES method (see the characterization (\ref{CGLS})). In other words, after performing $m$ steps of the Arnoldi algorithm to build $W_m$ (exactly as GMRES does), the CGLS method is employed to solve the projected LS problem in (\ref{GMRES}), with $m$ fixed. Therefore, in a sequential way, one applies another iterative regularization method within a fixed iteration of an iterative regularization method.
}
\end{remark}
\begin{remark}\emph{
As briefly mentioned in Section \ref{sect:trfreeArn} and in Remark \ref{rem:AdvDraw}, regularization methods based on the Arnoldi algorithm
may sometimes be ineffective because the SVD components of $A$ are mixed in the approximation subspace $\K_m(A,b)$. Here we display a numerical example clearly showing that, while severe SVD mixing affects the basis vectors of the GMRES solution, the SVD components are somewhat unmixed in the TF-CGLS basis vectors, whose behavior is comparable to the CGLS ones. The same holds for the hybrid GMRES basis vectors (where the projected problem (\ref{GMRES}) is regularized through TSVD). Analogously to \cite{JH07}, we consider the test problem \texttt{i\_laplace(100)} from \cite{RegT}, and we add Gaussian white noise $e$ to the data vector $b$, in such a way that the noise level $\weps = \|e\|/\|\bex\|$ is $5\cdot 10^{-4}$.
We consider, as an example, approximation subspaces of dimension 5, spanned by the orthonormal columns of the matrix $\hW_5\in\R^{100\times 5}$, associated to the GMRES, TF-CGLS, hybrid GMRES, and CGLS methods. More specifically:
\begin{itemize}
\item for GMRES: we first run 5 steps the Arnoldi algorithm to generate \linebreak[4]$\WGM_5\in\R^{100\times 5}$ and $H_5\in\R^{6\times 5}$ as in (\ref{ArnoldiDec}), and we then compute the SVD of $H_5$ (\ref{svdH}), whose right singular vector matrix is denoted by $\VHgm_5\in\R^{5\times 5}$. We take $\hW_5 = \WGM_5\VHgm_5$.
\item for TF-CGLS and for hybrid GMRES: we first run 40 steps of the Arnoldi algorithm to generate $\WGM_{40}\in\R^{100\times 40}$ and $H_{40}\in\R^{41\times 40}$ as in (\ref{ArnoldiDec}); we then compute the SVD of $H_{40}$ (\ref{svdH}), and we consider truncation after 5 components. We denote the truncated right singular vector matrix by $\VHhy_5\in\R^{40\times 5}$. We take $\hW_5 = \WGM_{40}\VHhy_5$. This corresponds to taking only the first 5 basis vectors in the TF-CGLS approximate solution.
\item for CGLS: we first run 5 steps of the Arnoldi algorithm applied to $A^TA$, with starting vector $A^Tb$ (though, in practice, this procedure is unadvisable, see \cite[\S 8.3]{Saad}). In this way we generate $\WCG_5\in\R^{100\times 5}$ with orthonormal columns, and $T_5\in\R^{6\times 5}$ tridiagonal. We then compute the SVD of $T_{5}$, whose right singular vector matrix is denoted by $\VT_5\in\R^{5\times 5}$. We take $\hW_5 = \WCG_5\VT_5$.
\end{itemize}
Here the Arnoldi algorithm is implemented through Householder transformations, so to guarantee a high accuracy in the orthonormal columns of $W_{m+1}$ (see \cite[\S 6.3]{Saad}). Figure \ref{fig:TSVD} shows the absolute value of the first, third, and fifth column of $\hW_5$ expressed in terms of the right singular values of $A$,
i.e., $V^T\hW_5$, for the GMRES, TF-CGLS, and CGLS methods.
\begin{figure}[tbp]
\centering
\begin{tabular}{cc}
\hspace{-0.7cm}\textbf{{\small {$V^T\hW_5$, 1st column}}} &
\hspace{-0.7cm}\textbf{{\small {$V^T\hW_5$, 3rd column}}} \\
\hspace{-0.7cm}\includegraphics[width=5.8cm]{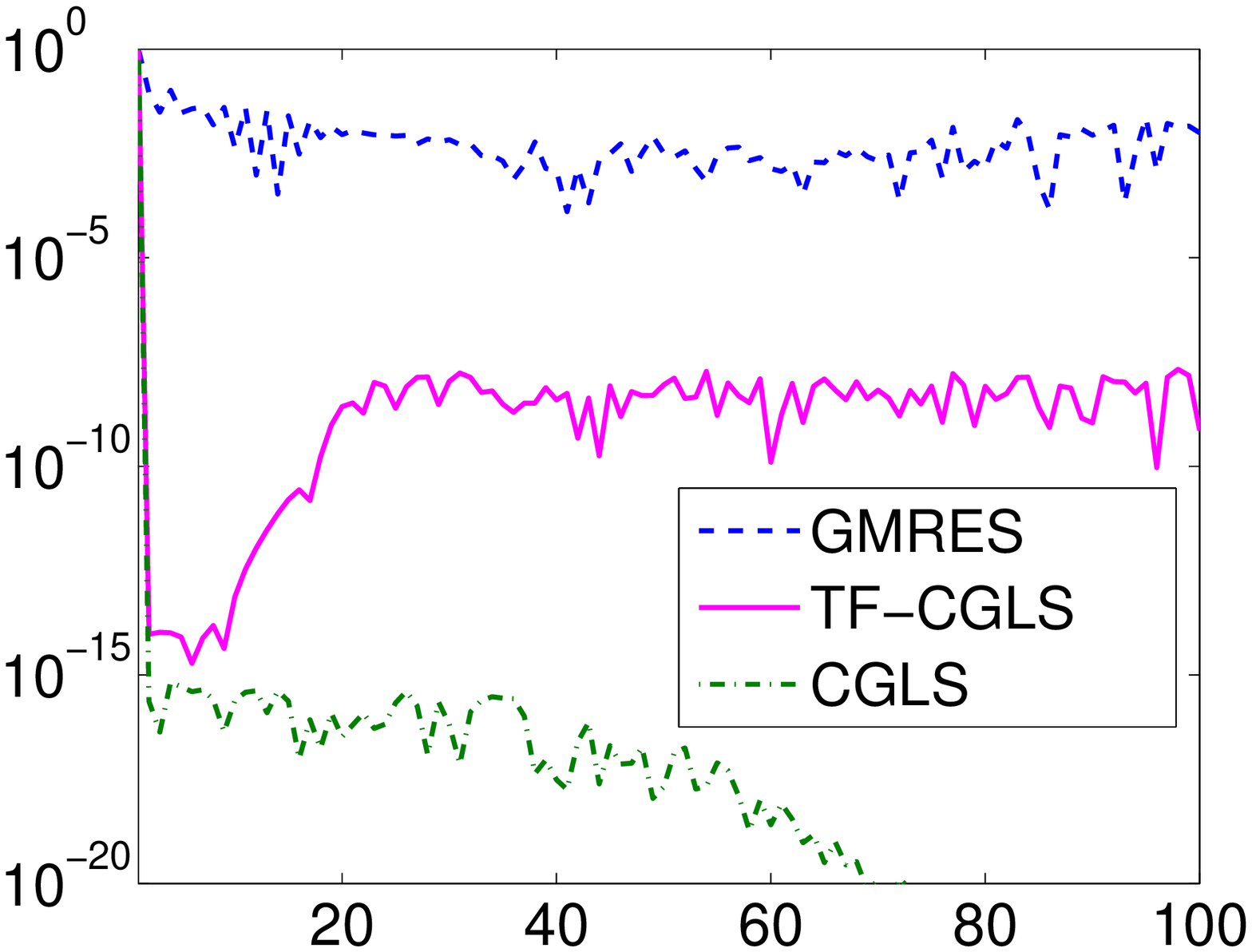} &
\hspace{-0.7cm}\includegraphics[width=5.8cm]{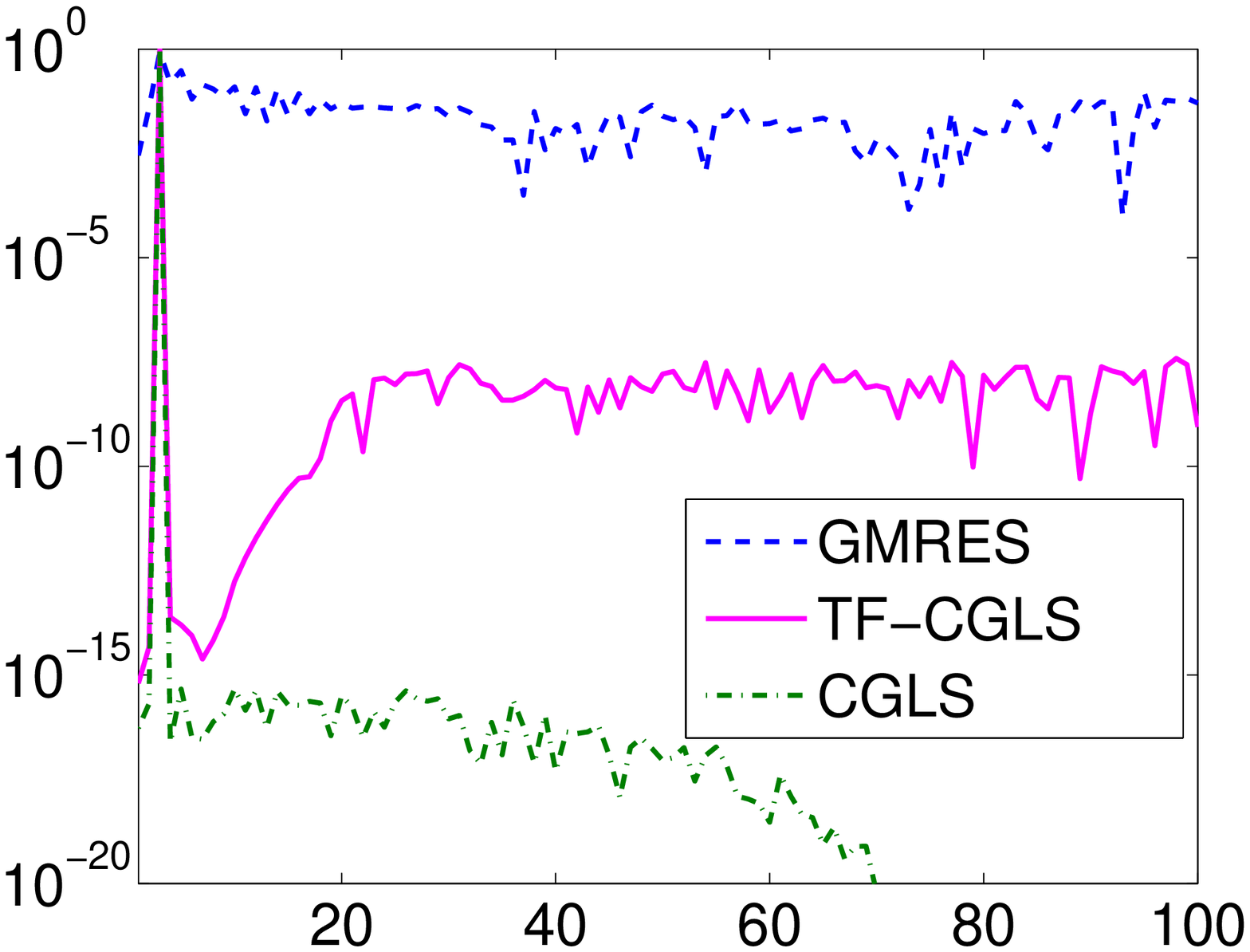}\\
\hspace{-0.7cm}\textbf{{\small {$V^T\hW_5$, 5th column}}} &
\hspace{-0.7cm}\textbf{{\small {filter factors}}} \\
\hspace{-0.7cm}\includegraphics[width=5.8cm]{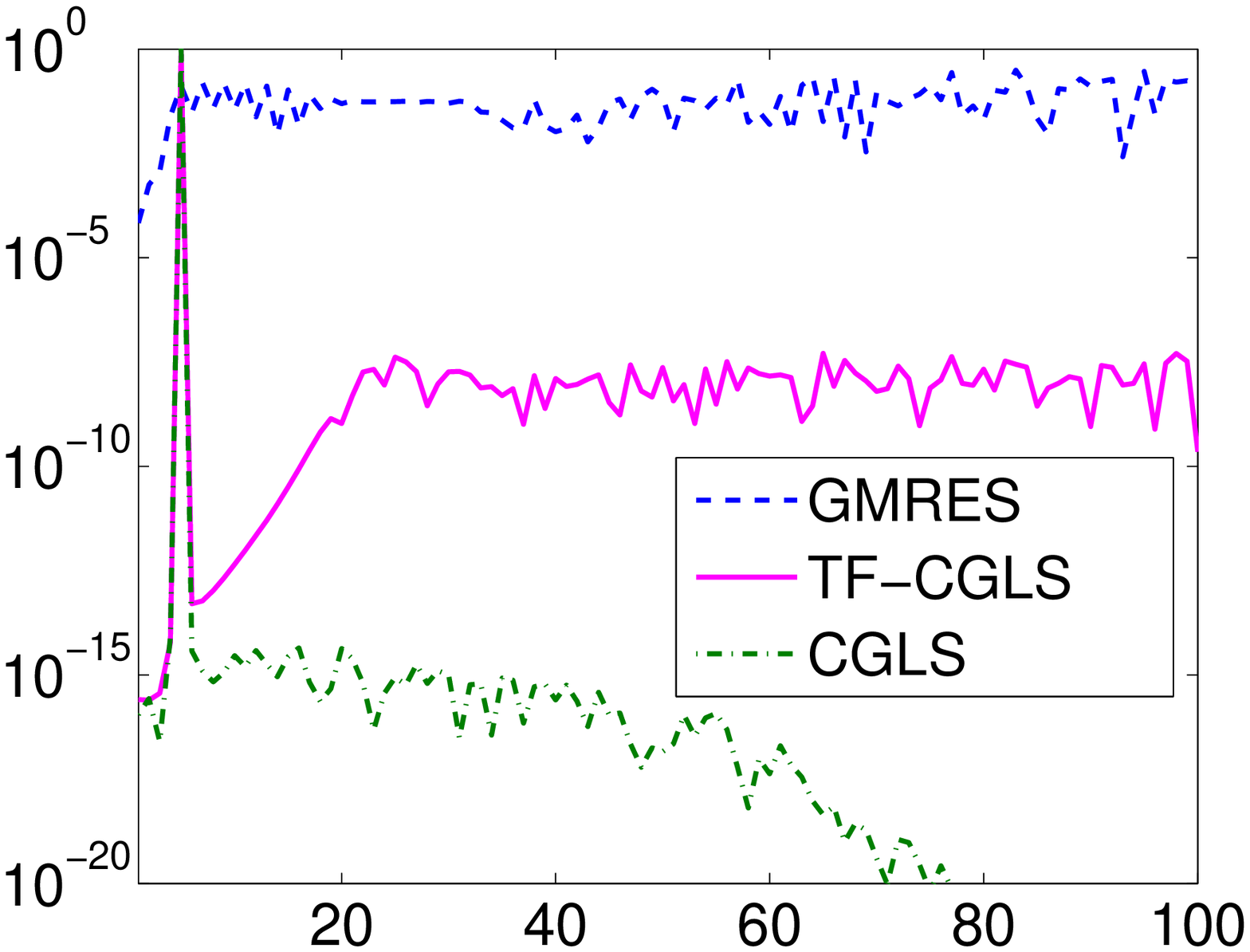} &
\hspace{-0.7cm}\includegraphics[width=5.8cm]{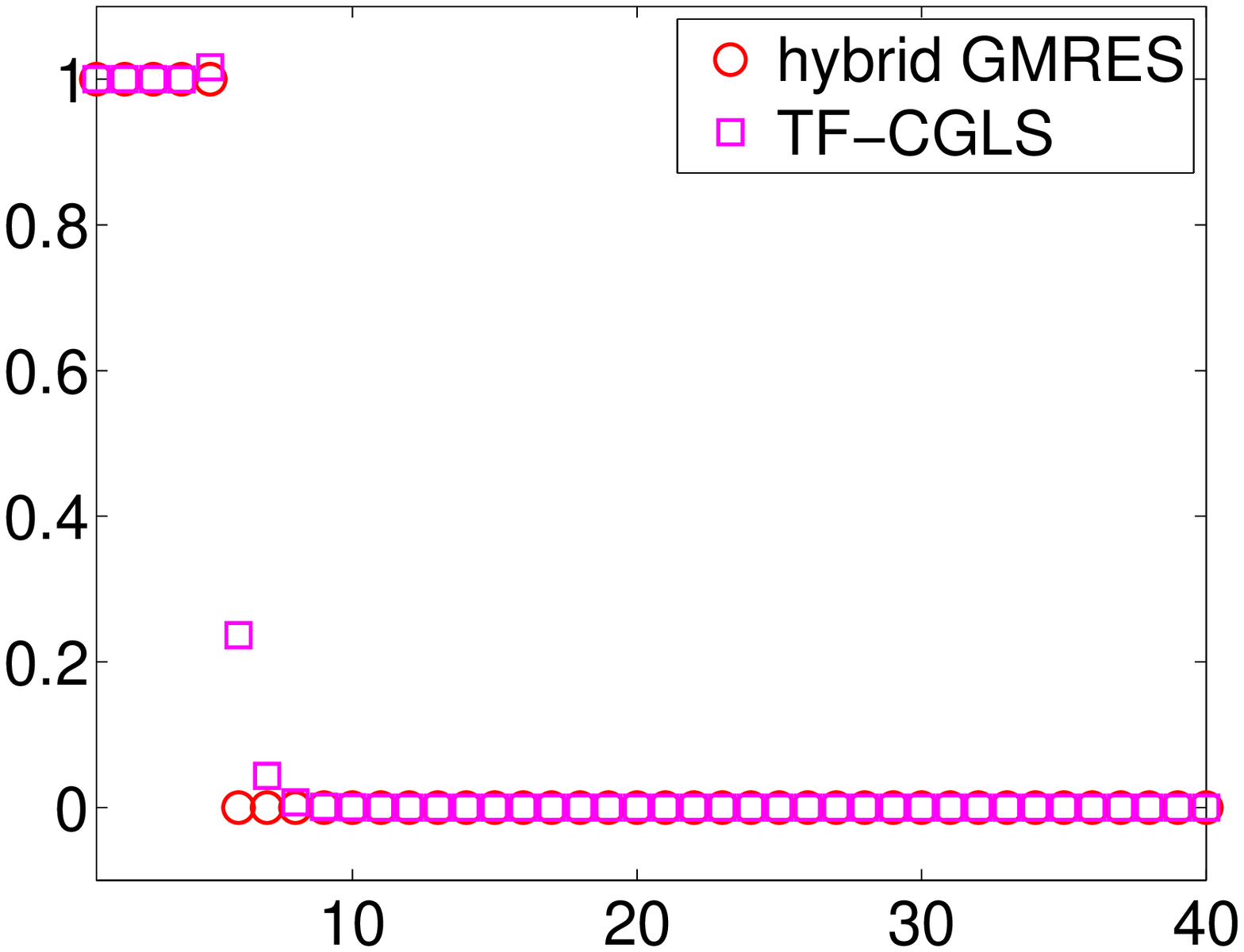}
\end{tabular}
\caption{Components of the first, third, and fifth columns of $\hW_5$ with respect to the right singular vector basis for the GMRES, TF-CGLS, and CGLS methods applied to the \texttt{i\_laplace(100)} test problem. Lower rightmost frame: filter factors for the TF-CGLS (5 CGLS iterations) and hybrid GMRES (5 components) methods.}
\label{fig:TSVD}
\end{figure}
It can be easily seen that, while the GMRES basis vectors have significant components along all the right singular vectors, the same is not true for TF-CGLS. Though the components of the TF-CGLS basis vectors are on average larger than the CGLS ones, the components corresponding to the first singular values of $A$ are clearly dominant (and, in an even more desirable fashion: the $i$th SVD component seems to dominate the $i$th basis vector).
The reason behind this phenomenon lies in the fact that the starting point of TF-CGLS (and hybrid GMRES) is the Krylov subspace $\K_{40}(A,b)$, which is much wider than the Krylov subspace $\K_5(A,b)$ used in standard GMRES and, therefore, it reasonably contains much more spectral information on $A$. As explained in Remark \ref{rem:hybrid}, once $40$ Arnoldi steps have been performed, both TF-CGLS and hybrid GMRES apply additional regularization (or filtering) on the projected least squares problem (\ref{GMRES}), so that
\[
x_{40,5}=\WGM_{40}\VHhy_{40}\Phi^{H_{40}}(\Sigma^{H_{40}}_{40})^{-1}(U^{H_{40}}_{40})^T(\|b\|e_1)\,,
\]
where $U^{H_{40}}_{40}$, $\Sigma^{H_{40}}_{40}$, and $\VHhy_{40}$ are the matrices appearing in the economy-size SVD of $H_{40}$, and $\Phi^{H_{40}}$ is a diagonal filtering matrix, whose elements are:
\[
\Phi^{H_{40}}_{i,i}=p_5(\sigma_i^{(40)})
\mbox{ for TF-CGLS,}\quad
\Phi^{H_{40}}_{i,i}=\begin{cases}
1 \!\!&\quad\text{if $i=1,\dots,5$}\\
0 \!\!&\quad\text{otherwise}
\end{cases}
\mbox{for hybrid GMRES,}\quad
\]
where $p_5$ is the polynomial of degrees at most $4$ associated to 5 CGLS iterations for the projected LS problem in (\ref{GMRES}). These filter factors are displayed in the lower rightmost frame of Figure \ref{fig:TSVD}. Starting from an extended Krylov subspace, and being able to filter out the dominant singular components of the projected quantities in (\ref{GMRES}), both TF-CGLS and hybrid-GMRES build a solution subspace where the original SVD components of $A$ are not as mixed as in the standard GMRES one.
}
\end{remark}

\subsection{A transpose-free CGNE-like method}

\label{sect:tfCGNE}

Now consider computing an approximation $y_{m,k}$ of $y_{m}$ in (\ref{newProjPb}%
) by applying $k$ iterations of the CG method. As well known, this means
that
\begin{equation}
y_{m,k}\in \mathcal{K}_{k}(C_{m}C_{m}^{T},b)\,,\quad \mbox{and}\quad
b-C_{m}C_{m}^{T}y_{m,k}\,\perp \,\mathcal{K}_{k}(C_{m}C_{m}^{T},b)\,,\quad
k\leq m\,.  \label{condCG}
\end{equation}%
As done in (\ref{xTFCGLS}), we can write $\xmk=P_mA^T\ymk$, so that, by using (\ref{projMatr}), the first condition in (\ref{condCG}) can be rewritten as
%
%
%
\[
x_{m,k} \in P_mA^T\mathcal{K}_{k}(AP_{m}A^{T},b)\,. 
\]
Moreover, the second condition in (\ref{condCG}) leads to
\begin{eqnarray*}
Ax - AP_{m}A^{T}y_{m,k}&\perp &\mathcal{K}_{k}(AP_{m}A^{T},b) \\
x-P_{m}A^{T}y_{m,k} &\perp &A^{T}\mathcal{K}_{k}(AP_{m}A^{T},b) \\
x-x_{m,k} &\perp &A^{T}\mathcal{K}_{k}(AP_{m}A^{T},b)\,.
\end{eqnarray*}%
We can summarize the above arguments in the following
\begin{proposition}
For any given $m\geq 1$ the sequence $\left\{ x_{m,k}\right\} _{k\leq m}$
obtained by applying $k$ steps of the CG method to problem (\ref{newProjPb})
is the result of a Krylov method defined by%
\begin{equation}
x_{m,k}\in P_{m}A^{T}\mathcal{K}_{k}(AP_{m}A^{T},b)\quad \mbox{and}\quad
x-x_{m,k}\perp A^{T}\mathcal{K}_{k}(AP_{m}A^{T},b)\,.  \label{tfCGNE}
\end{equation}
\end{proposition}

The above proposition allows us to see the strong relation of this approach with
the well-known CGNE method, whose approximate solutions satisfy
\[
x_{k}^{\text{CGNE}}\in A^{T}\mathcal{K}_{k}(AA^{T},b)\quad \mbox{and}\quad
x-x_{k}^{\text{CGNE}}\perp A^{T}\mathcal{K}_{k}(AA^{T},b)\,,  
\]
and are computed through an implicit construction of the Krylov subspaces
$A^T\mathcal{K}_{k}(AA^{T},b)=\mathcal{K}_{k}(A^TA,A^Tb)$.
Using similar arguments to the ones in
Section \ref{sect:tfCGLS}, the new method (\ref{tfCGNE}) can be regarded as a transpose-free variant of a CGNE-like method, and from now on it will be simply referred to as TF-CGNE; correspondingly, the vector $x_{m,k}$ in (\ref{tfCGNE}) will be denoted as $x_{m,k}^{\text{NE}}$. Statements analogous to the ones explained in Remark \ref{rem:AdvDraw} also hold for the TF-CGNE case.
%

We conclude this section by mentioning that, although CGNE is an iterative regularization method, in practice it may perform very badly.
Indeed, if system (\ref{sys}) is inconsistent, CGNE does not even converge to $A^\dagger b$ (see \cite[Chapter 4]{Han95}). This means that, if the unperturbed system $A\xex=\bex$ is consistent, only small perturbations $e$ of $\bex$ are allowed, in such a way that $b$ still belongs to the range of $A$.
The same behavior is experimentally observed when performing the TF-CGNE method (see the numerical experiments in Section \ref{sect:numexp}). Therefore, even if TF-CGNE potentially represents an alternative to TF-CGLS, the latter is to be preferred when dealing with noisy ill-posed problems.


\subsection{Setting the regularization parameters}\label{sect:RegP}

The transpose-free CG-like methods described in Sections \ref{sect:tfCGLS} and \ref{sect:tfCGNE} (here briefly denoted by TF-CG) are,
indeed, multi-parameter iterative methods, whose success depends on an
accurate tuning of both the scalars $m$ and $k$. It should be also remarked that the parameters $m$ and $k$ act sequentially (this is the main difference between the hybrid and the TF-CGLS methods): once $m$ is
fixed, an appropriate value for $k$ should be set. A natural way to fix $m$
(i.e., the dimension of the Krylov subspace for the
approximate solution $x_{m,k}$) is to monitor the expansion of the Krylov subspace $%
\mathcal{K}_{m}(A,b)$, which can be measured by the sub-diagonal elements of
the Hessenberg matrix $H_{m}$ in (\ref{ArnoldiDec}) (see \cite{Survey,NR14}).
Therefore, we stop the preliminary iterations when
\begin{equation}
h_{m+1,m}<\tau ,  \label{stopO}
\end{equation}%
where $\tau >0$ is a specified threshold.
In terms of regularization, this criterion is partially justified by the bound
\begin{equation*}
\prod\nolimits_{j=1}^{m}h_{j+1,j}\leq \prod\nolimits_{j=1}^{m}\sigma _{j}\,,
\end{equation*}
(see \cite{Mo}), which basically states that, on geometric average, the sequence $\{h_{j+1,j}\}_{j\geq 1}$ decreases quicker than the singular values.

In principle, another natural approach to set $m$ can be devised by
monitoring the values of the quantity
\begin{equation}
\zeta _{m}=\Vert A^{T}A-P_{m}A^{T}A\Vert \,.  \label{stopOvar1}
\end{equation}%
The smaller $\zeta _{m}$, the nearer $A^{T}A$ to $P_{m}A^{T}A$, i.e., the
more accurate the transpose-free approximation of $A^{T}A$. Since the
approximate solutions $x_{m}$ computed by the TF-CG methods belong to the subspace \linebreak[4]
$\mathcal{K}_{m}(P_{m}A^{T}A,P_{m}A^{T}b)$ (see the first relation in (\ref{tfCGLS}) and (\ref{tfCGNE})),
a small $\zeta _{m}$ also implies that the generated approximation
subspaces are close to $\mathcal{K}_{m}(A^{T}A,A^{T}b)$. However, one of the
main motivations behind TF-CG methods being the lack of knowledge of $A^{T}$ for some
large-scale problems, the quantities $\zeta _{m}$ in (\ref{stopOvar1}) cannot be computed in
practice. Therefore, after some simple derivations one can provide the
following upper bound:
\[
\zeta _{m}=\Vert (I-P_{m})A^{T}A\Vert \leq \Vert A\Vert \cdot \Vert
(I-P_{m})A^T\Vert =\sigma _{1}\Vert A(I-P_{m})\Vert \,.  
\]
Though the above bound does not explicitly involve $A^{T}$, $A^{T}$ is
required by algorithms for computing $\sigma _{1}$. Moreover,
when dealing with large-scale problems, both $\sigma _{1}$ and $\Vert
A(I-P_{m})\Vert $ can be expensive to compute. Therefore, one should look
for yet other alternative bounds. One can take ${\sigma}_{1}^{(m)}$, i.e., the largest singular value of $H_m$,
as an approximation of $\sigma_{1}$: indeed, thanks
to the interlacing property of the singular values (see, for instance, \cite%
{ATfirst,iDPC}), one can prove that
\begin{equation*}
\sigma _{1}\geq {\sigma}_{1}^{(\ell +1)}\geq {\sigma}_{1}^{(\ell
)}.
\end{equation*}%
Many numerical experiments available in the literature show that ${\sigma%
}_{1}^{(m)}$ quickly approaches $\sigma _{1}$ (see also \cite{N}),
so that
\begin{equation}
\zeta _{m}\leq \sigma _{1}\Vert A(I-P_{m})\Vert =({\sigma}%
_{1}^{(m)}+\varepsilon _{m})\Vert A-W_{m+1}H_{m}W_{m}^{T}\Vert \,,
\label{stopOvar2}
\end{equation}
where $\eps_m\rightarrow 0$ as $m$ increases. Replacing $\sigma _{1}$ with ${\sigma}_{1}^{(m)}$ may not be meaningful
when $m$ is very small, but this is not the case when performing the first
cycle of iterations of the Arnoldi algorithm for the TF-CG methods. Note
that, to rewrite the second term of the last equality in the above equation,
we have also exploited (\ref{ArnoldiDec}) . While some numerical experiments
available in the literature (see \cite{Survey}) suggest that the quantity $%
\Vert A-W_{m+1}H_{m}W_{m}^{T}\Vert $ decays similarly to the singular
values of $A$, no theoretical results have been established, yet.
Similarly to what happens in the TSVD case, one can consider
\begin{equation*}
\Vert A-W_{m+1}H_{m}W_{m}^{T}\Vert \simeq {\sigma}_{m+1}^{(m+1)}\,.
\end{equation*}%
However, the above estimate can be quite optimistic, for various reasons. First of all, it
would be quite sharp if the matrices $\widehat{U}_{m}$ and $\widehat{V}_{m}$ in (\ref{svdhfull}) coincide with the TSVD matrices $U_{m}^A$ and $V_{m}^A$ in (\ref{TSVD}), respectively: if this
is not the case, $\Vert A-W_{m+1}H_{m}W_{m}^{T}\Vert \gg {\sigma}%
_{m+1}^{(m+1)}$. Secondly, contrarily to what happens to the extremal
singular values, and because of numerical inaccuracies, one cannot guarantee
that ${\sigma}_{m+1}^{(m+1)}\geq \sigma _{m+1}$. Nevertheless, experimentally
it appears reliable to stop the first set of Arnoldi iterations when
${\sigma}_{1}^{(m)}{\sigma}_{m+1}^{(m+1)}$ is sufficiently small, i.e.,
one should stop as soon as
\begin{equation}
{\sigma}_{1}^{(m)}{\sigma}_{m+1}^{(m+1)}<\tau^\prime\,,  \label{stopOvar3}
\end{equation}%
where $\tau ^{\prime }>0$ is a specified threshold.



To choose the number $k$ of additional iterations for the TF-CG methods,
some standard parameter choice strategies can be used. For instance, if one
has a good estimate of the noise level $\widehat{\varepsilon }$,
the discrepancy principle can be applied and the
iterations can be stopped as soon as
\begin{equation}
\Vert b-Ax_{m,k}\Vert = \Vert b-C_{m}C_{m}^{T}y_{m,k}\Vert =\left\Vert \Vert b\Vert e_{1}-H%
_{m}H_{m}^{T}z\right\Vert <\eta \widehat{%
\varepsilon }\Vert b\Vert\,,  \label{stopI}
\end{equation}%
where $\eta >1$ is a safety factor. If $\widehat{\varepsilon}$ is not known,
one can resort to other classical parameter choice methods such as GCV and
the L-curve (see \cite[Chapter 7]{PCH98}). The TF-CG methods are summarized in
Algorithm \ref{alg:tfCG}.

\begin{algorithm}
\caption{TF-CG methods}\label{alg:tfCG}
\begin{algorithmic}
\State \textbf{input} $A$, $b$, $\tau$ or $\tau^\prime$, \texttt{solver}, $\eta$, $\weps$
\For {$m=1,2,\dots,$until the stopping criterion (\ref{stopO}) or (\ref{stopOvar3}) is satisfied}{}
\State update the Arnoldi decomposition: $AW_m=W_{m+1}H_m$ 
\EndFor
\For {$k=1,2,\dots,$until (\ref{stopI}) is satisfied}
\If {\texttt{solver} is TF-CGLS}
    \State apply MINRES to the system $H_mH_m^Tt=\|b\|e_1$, to get $t_k$
\ElsIf {\texttt{solver} is TF-CGNE}
    \State apply CG to the system $H_mH_m^Tt=\|b\|e_1$, to get $t_k$
\EndIf
\EndFor
\State take $\xmk=W_mH_m^Tt_k$
\end{algorithmic}
\end{algorithm}

\section{Numerical experiments}\label{sect:numexp}

This section shows the performance of the methods summarized in
Algorithm \ref{alg:tfCG} on a variety of test problems: comparisons with
GMRES and, whenever possible, CGLS and CGNE, will be displayed, and the behavior of the class of the
TF-CG-like methods with respect to different choices of the number of iterations $%
m$ and $k$ will be assessed. A first set of experiments considers moderate-scale problems form \cite{RegT}, while a second set of experiments considers realistic large-scale problems arising in the framework of 2D image deblurring. All the tests are performed running MATLAB
R2013a on a single processor 2.2 GHz Intel Core i7.
%
%
%
\paragraph{First set of experiments.}
We consider problems with a nonsymmetric coefficient matrix and a righ-hand-side vector that is affected by Gaussian white noise, whose level is $\widehat{%
\varepsilon} = 10^{-2}$. For all the tests, the maximum allowed number of Arnoldi
iterations (in the first cycle of iterations in Algorithm \ref{alg:tfCG}) is
$m_{\max}=40$, and $\eta=1.01$. Since $A^T$, as well as the SVD of $A$, are easily available for these problems, the use of TF-CG-like methods may appear meaningless in this setting: these experiments are nonetheless included to compare the behavior of the TF-CG-like methods and the CGLS, and CGNE methods, and to test some theoretical estimates (such as (\ref{stopOvar1}) -- (\ref{stopOvar3})).
\begin{enumerate}
\item \texttt{\textbf{i\_laplace}}. Let us consider the inverse Laplace transform of the function \linebreak[4]$%
f(t) = \exp(-t/2)$, i.e., we wish to solve the integral equation
\begin{equation}  \label{ilaplace1}
\int_{0}^{\infty}\exp(-st)f(t)dt = g(s)+e\,,
\end{equation}
where $g(s)=1/(s+1/2)$ and $e$ is some unknown (continuous) noise. The discretization of (\ref%
{ilaplace1}) is available within \cite{RegT}: we choose $N=100$, so that
\begin{equation}  \label{dist_ilaplace}
\|A-A^T\|/\|A\|=0.7456\,.
\end{equation}
The values $\tau = 10^{-10}$ and $\tau^{\prime}=10^{-15}$ are
chosen for the stopping criteria in (\ref{stopO}) and (\ref{stopOvar3}),
respectively.
\begin{figure}[tbp]
\centering
\begin{tabular}{c}
\textbf{{\small {Relative Error History}}} \\
\includegraphics[width=11cm]{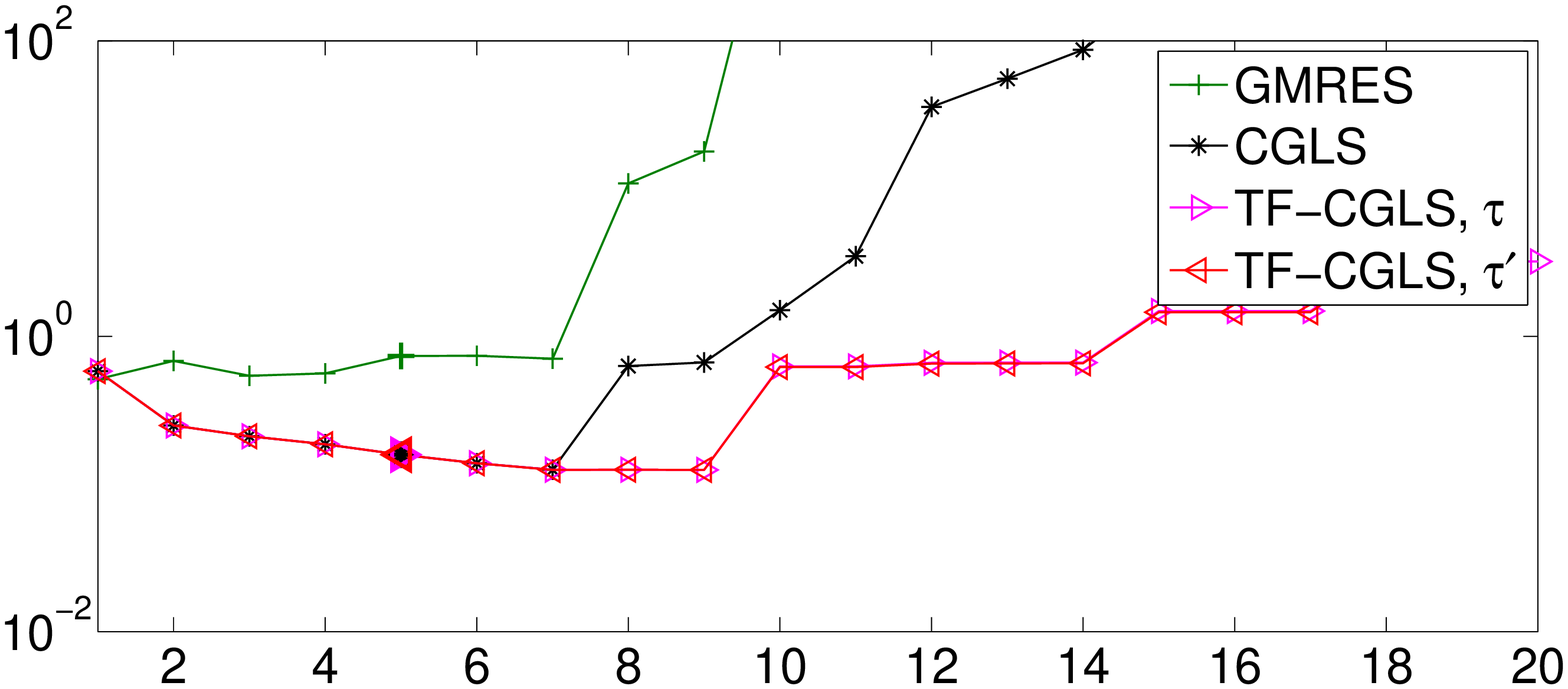} \\
\textbf{{\small {Relative Residual History}}} \\
\includegraphics[width=11cm]{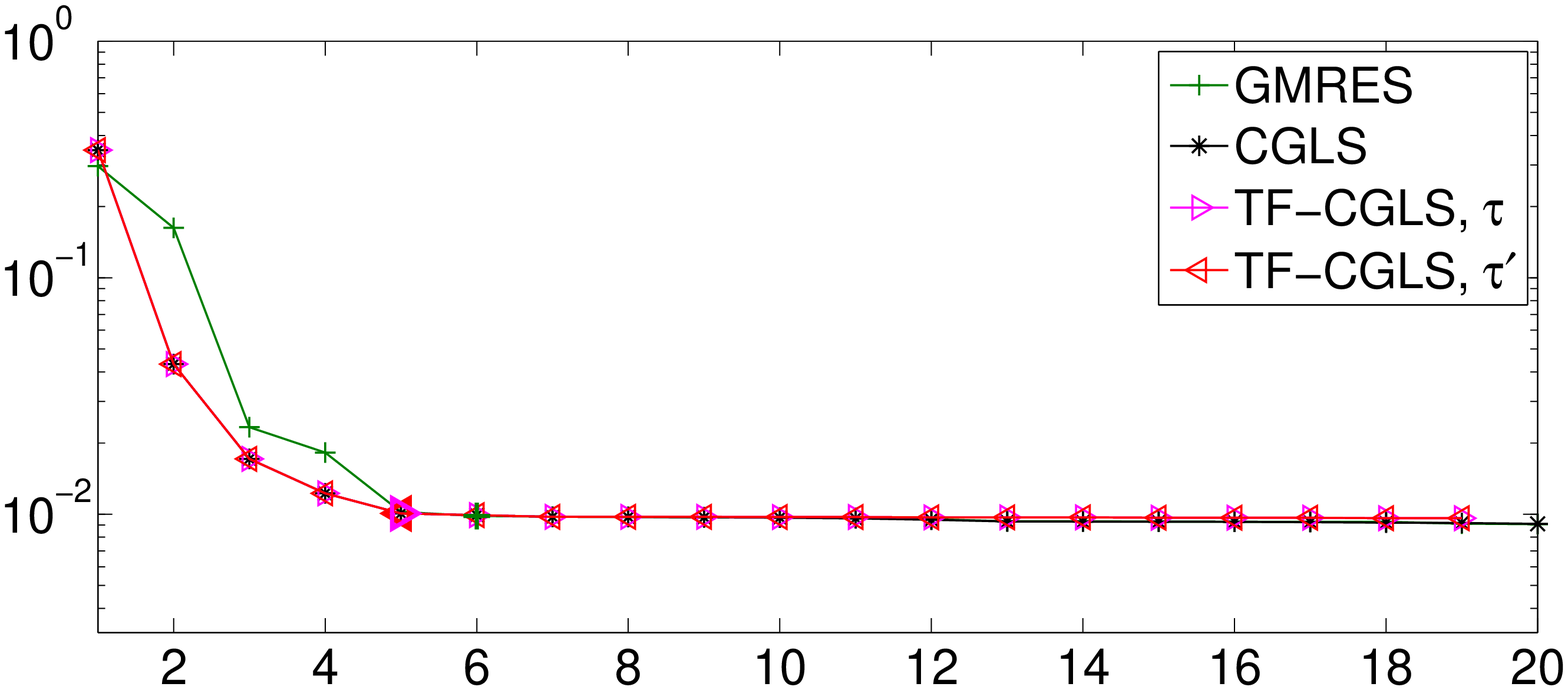}%
\end{tabular}%
\caption{Test problem \texttt{i\_laplace} (\protect\ref{ilaplace1}), with $f(t)=\exp(-t/2)$, $N=100$, and $\weps = 10^{-2}$. Upper
frame: relative errors versus number of iterations. Lower frame: relative
residuals versus number of iterations.}
\label{fig:ilaplace1_1}
\end{figure}
Figure \ref{fig:ilaplace1_1} compares the GMRES, CGLS, and TF-CGLS methods
(for different choices of the stopping criterion for the first set of iterations).
Both stopping criteria (\ref{stopO}) and (\ref{stopOvar3}) are satisfied after 20 Arnoldi iterations. Enlarged markers are used to highlight the
iterations satisfying the discrepancy principle (\ref{stopI}) for all the methods (so that, in the GMRES and CGLS case, the quantities $\|b-Ax_m\|$ are monitored). In the upper frame of Figure \ref%
{fig:ilaplace1_1}, one can clearly see the TF-CGLS methods to deliver a
huge improvement over the standard GMRES method, and the behavior of the TF-CGLS method is very similar to the CGLS one. TF-CGLS seems also very robust with respect to \tql semi-converegence\tqr. Numerical values of the relative error
\[
\|x_\ast - \xex\|/\|\xex\|\,,
\]
where $x_\ast = x_m$ (for GMRES and CGLS) or $x_\ast=\xmk$ (for TF-CGLS with (\ref{stopOvar3}) as first stopping criterion), are reported in Table \ref{tab:expo}: the average over 20 runs of each test problem, with different realizations of the random noise vector in the data, are taken. Table \ref{tab:expo} also reports the average number of iterations performed to satisfy the discrepancy principle (\ref{stopI}), and the average number of Arnoldi steps required to satisfy the stopping criteria (\ref{stopO}) and (\ref{stopOvar3}) during the first cycle of TF-CGLS iterations.
Regarding the stopping criteria, a word of caution is mandatory: although looking at Figure \ref{fig:ilaplace1_1} and Table \ref{tab:expo} it may seem that all the methods stop after roughly 5 iterations, we should recall that TF-CGLS actually stops after 5 iterations during the second cycle in Algorithm \ref{alg:tfCG}, and that this only happens after 20 Arnoldi iterations have been performed.
%
Therefore, for this test problem, the computational cost of GMRES, CGLS, and
TF-CGLS is roughly dominated by the cost of 5, 10, and 20 matrix-vector products with a matrix of size $N\times N$, 
respectively.
We think that the additional (but still small) number of matrix-vector
products required by TF-CGLS is tolerable if we consider the improved quality of the
solution (with respect to GMRES), and the transpose-free feature (with respect to
CGLS). In the lower frame of Figure \ref{fig:ilaplace1_1} one
can notice a slight increase in the TF-CGLS residuals (so
that they are not monotonic).
Moreover, inequality (\ref{optGMres}) also applies to the TF-CGLS case, i.e., \linebreak[4]when $x_m = x_{m,k}\in\mathcal{K}_m(A,b)$ (recall (\ref{xTFCGLS})): more precisely, once $m$ has been set, $\|Ax^{\text{GMR}}_m-b\|$ is smaller than any $%
\|Ax_{m,k}-b\|$, for $k\leq m$
(but this does not imply any other relation between $\|Ax^{\text{GMR}}_{\ell}-b\|$, $%
\ell< m$, and $\|Ax_{m,k}-b\|$).
\begin{figure}[tbp]
\centering
\begin{tabular}{cc}
\hspace{-0.7cm}\textbf{{\small {(a)}}} & \hspace{-0.7cm}\textbf{{\small {(b)}}} \\
\hspace{-0.7cm}\includegraphics[width=6.2cm]{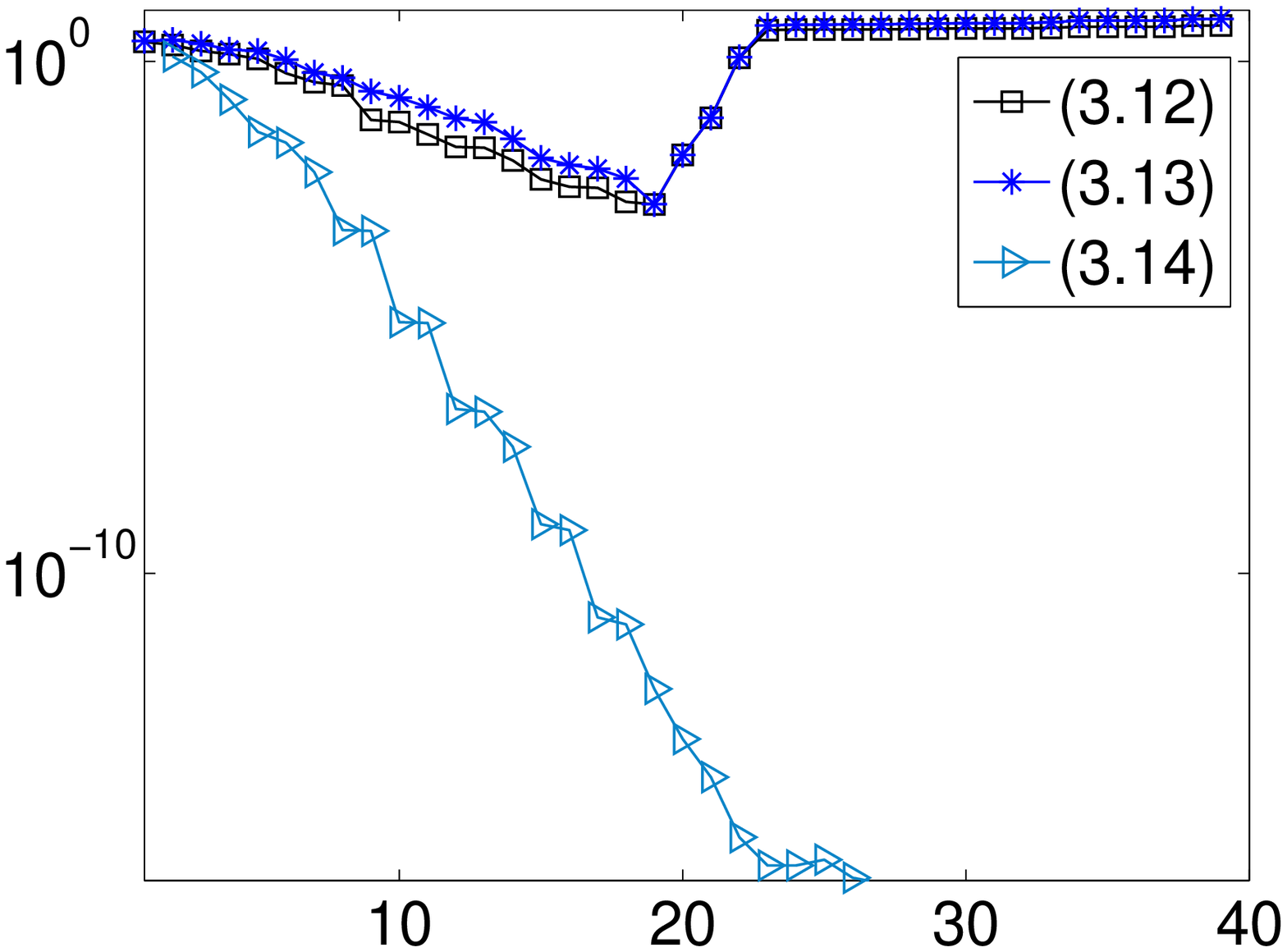} & %
\hspace{-0.7cm}\includegraphics[width=6.2cm]{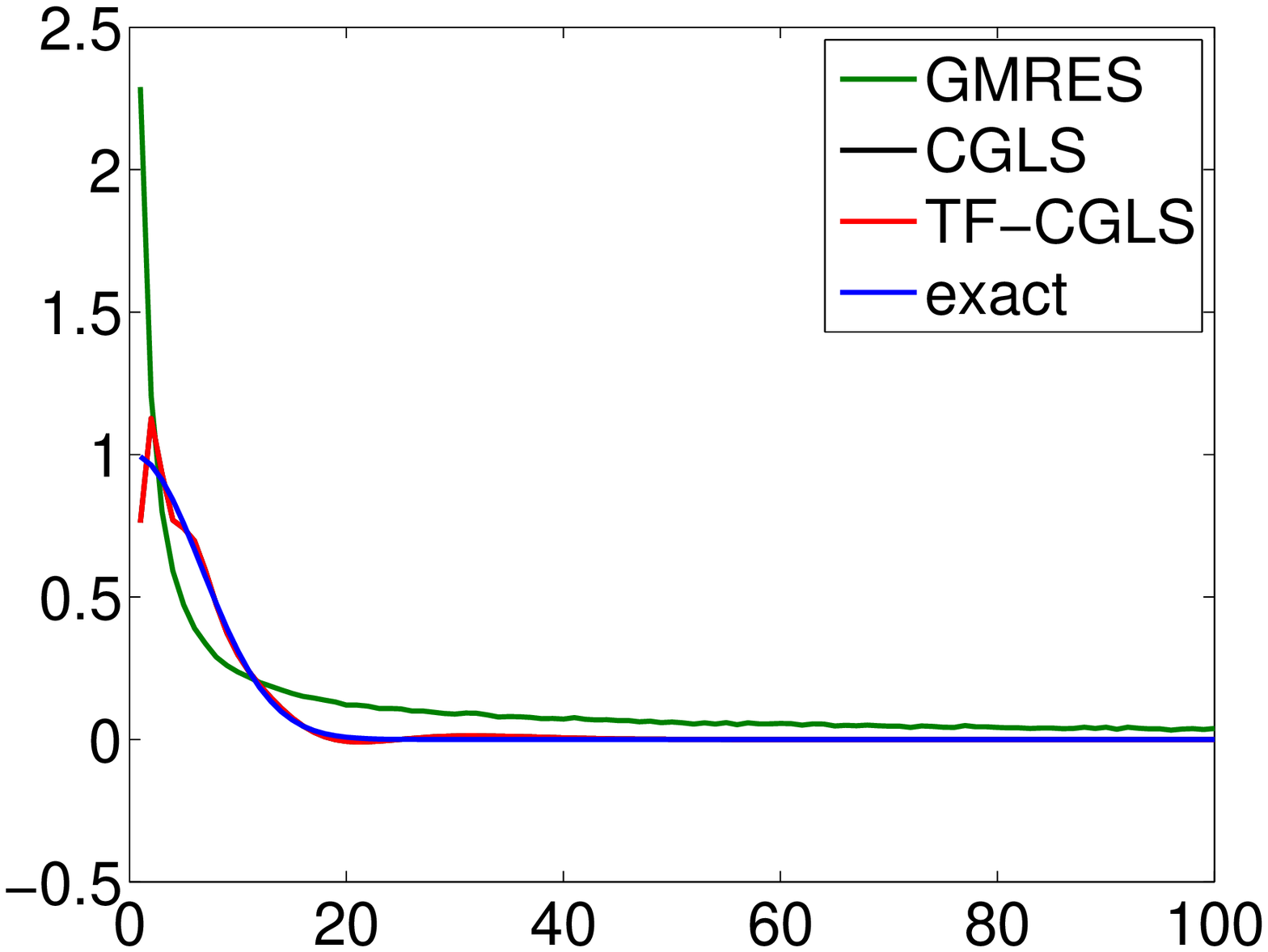}%
\end{tabular}%
\caption{Test problem \texttt{i\_laplace} (\protect\ref{ilaplace1}), with $f(t)=\exp(-t/2)$, $N=100$, and $\weps = 10^{-2}$.
\textbf{(a)} Values of the quantities (\protect\ref{stopOvar1}), (\ref{stopOvar2}), and (\protect\ref%
{stopOvar3}) versus the number of Arnoldi iterations $m$. \textbf{(b)} Best
approximations achieved by the GMRES, CGS, and TF-CGLS methods.}
\label{fig:ilaplace1_3}
\end{figure}

The left frame of Figure \ref{fig:ilaplace1_3} displays the
behavior of the quantities (\ref{stopOvar1}) -- (\ref{stopOvar3}) versus the
number of Arnoldi iterations $m$. One can clearly see that
(\ref{stopOvar2}) is a tight bound for the potentially unknown quantity (\ref%
{stopOvar1}).
One also realizes that estimate (\ref{stopOvar3}) is indeed very optimistic, as anticipated in Section \ref{sect:RegP}.
We also emphasize that the behavior of the sequence $(\zeta_m)_{m\geq 1}$ is not monotonic due to the loss of orthogonality in the columns of the matrix $W_m$ in (\ref{ArnoldiDec}). For this set of experiments, the modified Gram-Schmidt implementation of the Arnoldi algorithm was considered, and the TF-CGLS approximations are not very affected by the loss of orthogonality (as the stopping criteria for the first cycle of Arnoldi iterations prescribe to stop after 20 iterations, i.e., before a severe loss of orthogonality sets in). However, when a larger number of Arnoldi iterations is expected during the first cycle of Algorithm 1, one may consider the more numerically accurate (and more expensive) Householder-Arnoldi implementation, in order to reduce the effect of the loss of orthogonality (see \cite[\S 6.3]{Saad} for details). In the right frame of Figure \ref{fig:ilaplace1_3}, the best solutions achieved by GMRES, CGLS, and TF-CGLS
are plotted: while the CGLS and TF-CGLS solutions are basically aligned with
the exact one, this is not the case for the GMRES solution, which is heavily
mismatched on the left boundary, and presents some light spurious oscillations on the
right.

\begin{figure}[tbp]
\centering
\begin{tabular}{c}
\textbf{{\small {Relative Error History}}} \\
\includegraphics[width=11cm]{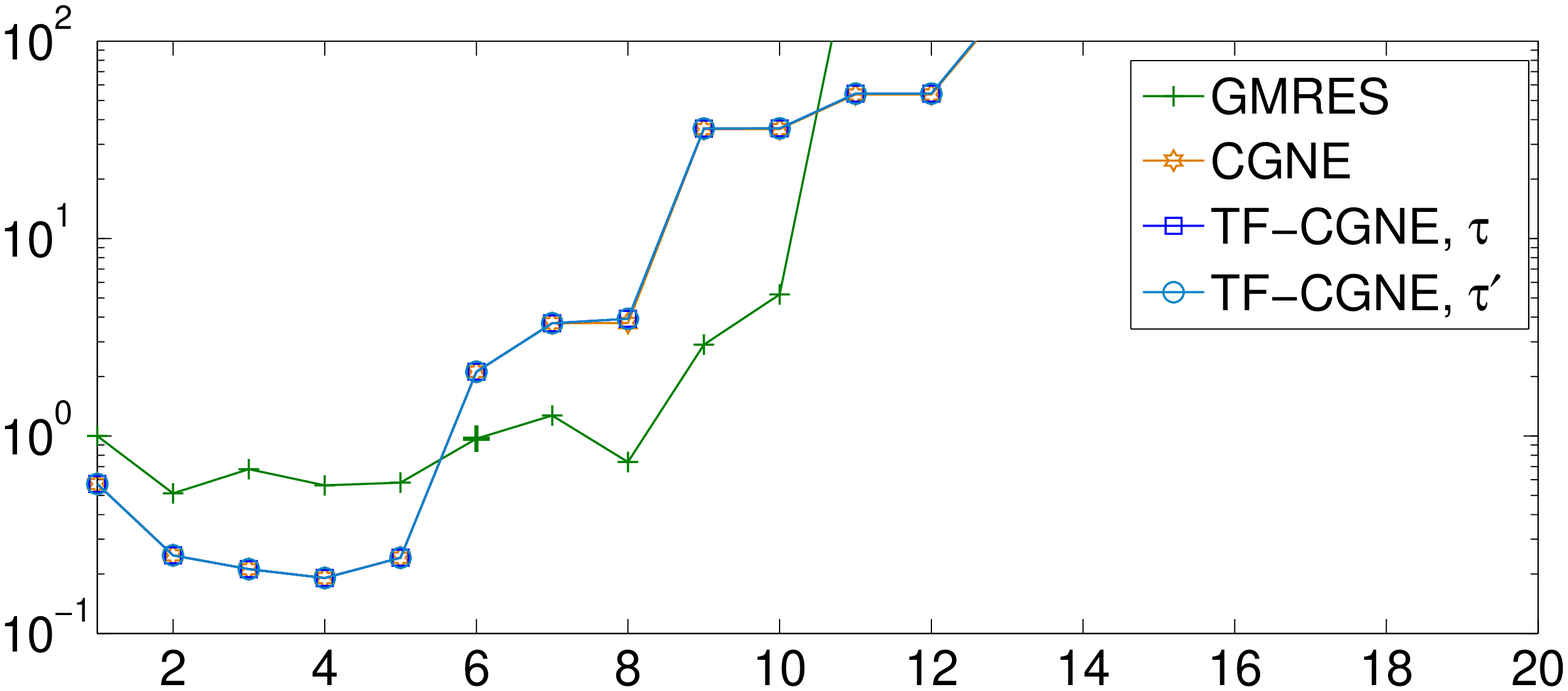} \\
\textbf{{\small {Relative Residual History}}} \\
\includegraphics[width=11cm]{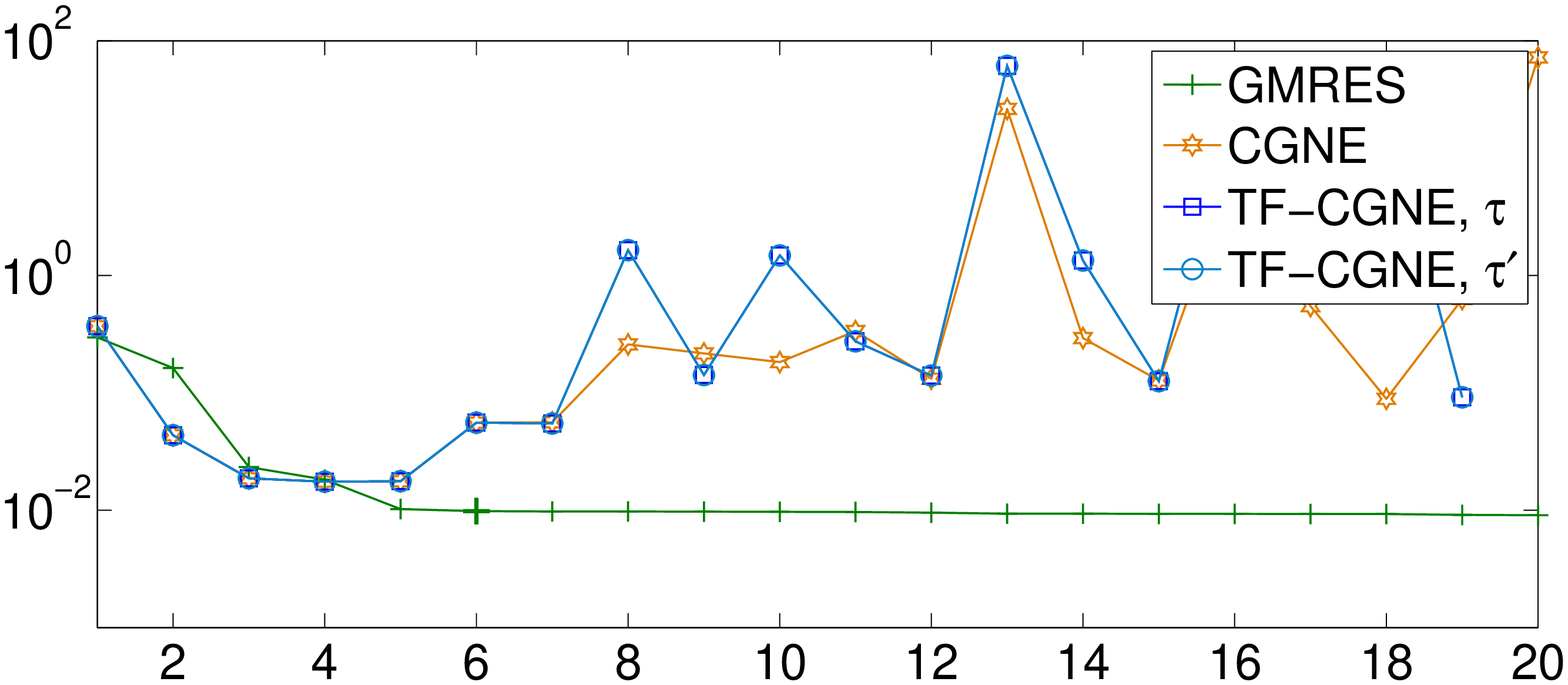}%
\end{tabular}%
\caption{Test problem \texttt{i\_laplace} (\protect\ref{ilaplace1}), with $f(t)=\exp(-t/2)$, $N=100$, and $\weps = 10^{-2}$. Upper
frame: relative errors versus number of iterations. Lower frame: relative
residuals versus number of iterations.}
\label{fig:ilaplace1_2}
\end{figure}
Figure \ref{fig:ilaplace1_2} compares the GMRES, CGNE, and TF-CGNE methods,
and has the same layout as Figure \ref{fig:ilaplace1_1}. One can clearly see
the performance of \textquotedblleft minimal error\textquotedblright\
methods to be much worse than the performance of \textquotedblleft minimal
residual\textquotedblright\ methods and, in particular, stopping criteria based on the discrepancy principle fail in this setting (this agrees with the analysis performed in \cite[Chapter 4]{Han95}). Despite this, the TF-CGNE method is
able to reproduce quite faithfully the behavior of CGNE (both in terms of
relative errors and relative residuals). Further tests with CGNE and TF-CGNE will not be performed in the following experiments.

\begin{figure}[tbp]
\centering
\begin{tabular}{cc}
\hspace{-0.7cm}\textbf{{\small {(a)}}} & \hspace{-0.7cm}\textbf{{\small {(b)}}} \\
\hspace{-0.7cm}\includegraphics[width=6.2cm]{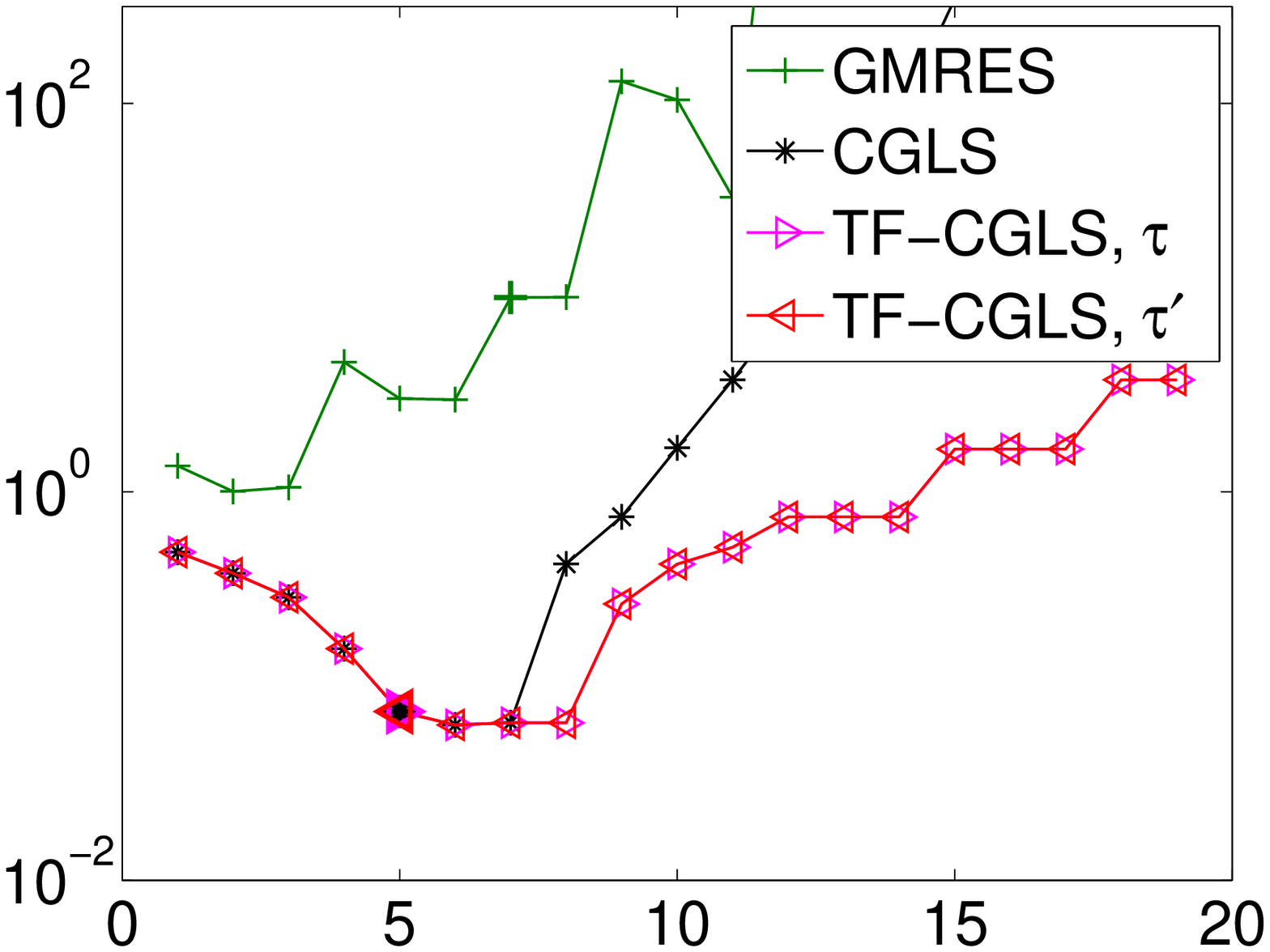} & %
\hspace{-0.7cm}\includegraphics[width=6.2cm]{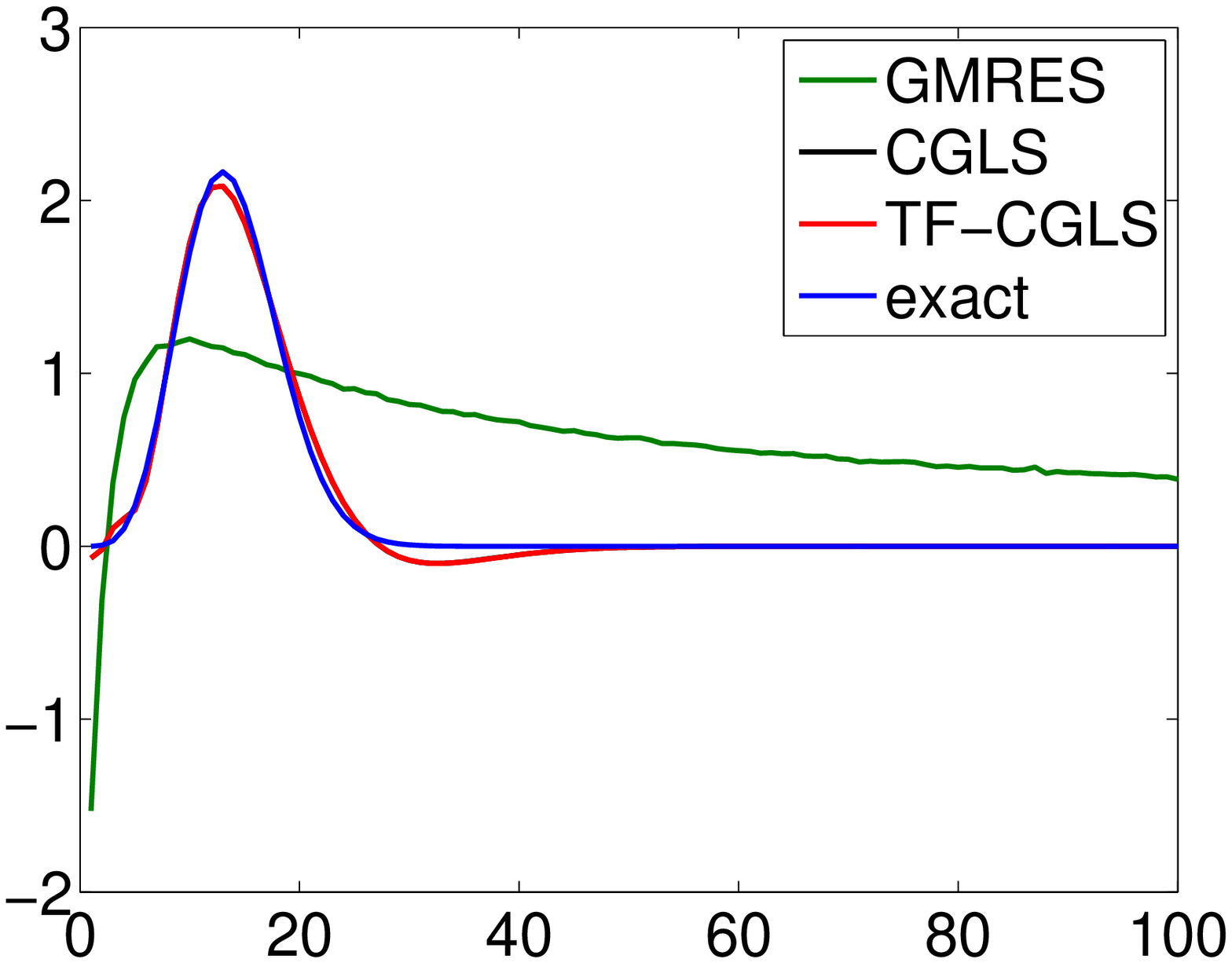}%
\end{tabular}%
\caption{Test problem \texttt{i\_laplace} (\protect\ref{ilaplace1}), with $f(t)=t^2\exp(-t/2)$, $N=100$, and $\weps = 10^{-2}$.
\textbf{(a)} Relative error history. \textbf{(b)} Best
approximations achieved by the GMRES, CGS, and TF-CGLS methods.}
\label{fig:ilaplace2}
\end{figure}
Finally, in Figure \ref{fig:ilaplace2} we consider the inverse Laplace transform (\ref{ilaplace1}) of the function $f(t)=t^2\exp(-t/2)$. We set again $N=100$, so that (\ref{dist_ilaplace}) still holds. The left frame of Figure \ref{fig:ilaplace2} shows the history of the relative errors, and enraged markers are used to highlight the iterations satisfying the discrepancy principle; for the TF-CGLS method, (\ref{stopO}) and (\ref{stopOvar3}) are satisfied after 20 and 21 Arnoldi iterations, respectively. The right frame of Figure \ref{fig:ilaplace2} displays the best reconstructions obtained by different methods and, also for this example, the CGLS and TF-CGLS solutions almost coincide, and they are very close to the exact one: the improvement of TF-CGLS over GMRES is clearly visible. Table \ref{tab:expo} reports numerical values for this experiment.
\item \texttt{\textbf{baart}}. This is an artificial Fredholm integral equation of the first kind, whose discretization is available within \cite{RegT}. We set $N=200$, so that $\|A-A^T\|/\|A\|=6.0345\cdot 10^{-1}$. The value \linebreak[4]$\tau = 10^{-10}$ is chosen for the stopping criterion in (\ref{stopO}), which is satisfied after 9 iterations; $\tau^{\prime}=10^{-14}$ is chosen for (\ref{stopOvar3}), which holds after 15 iterations. The modified Gram-Schmidt implementation of the Arnoldi algorithm is considered.
\begin{figure}[tbp]
\centering
\begin{tabular}{cc}
\hspace{-0.7cm}\textbf{{\small {(a)}}} & \hspace{-0.7cm}\textbf{{\small {(b)}}} \\
\hspace{-0.7cm}\includegraphics[width=6.2cm]{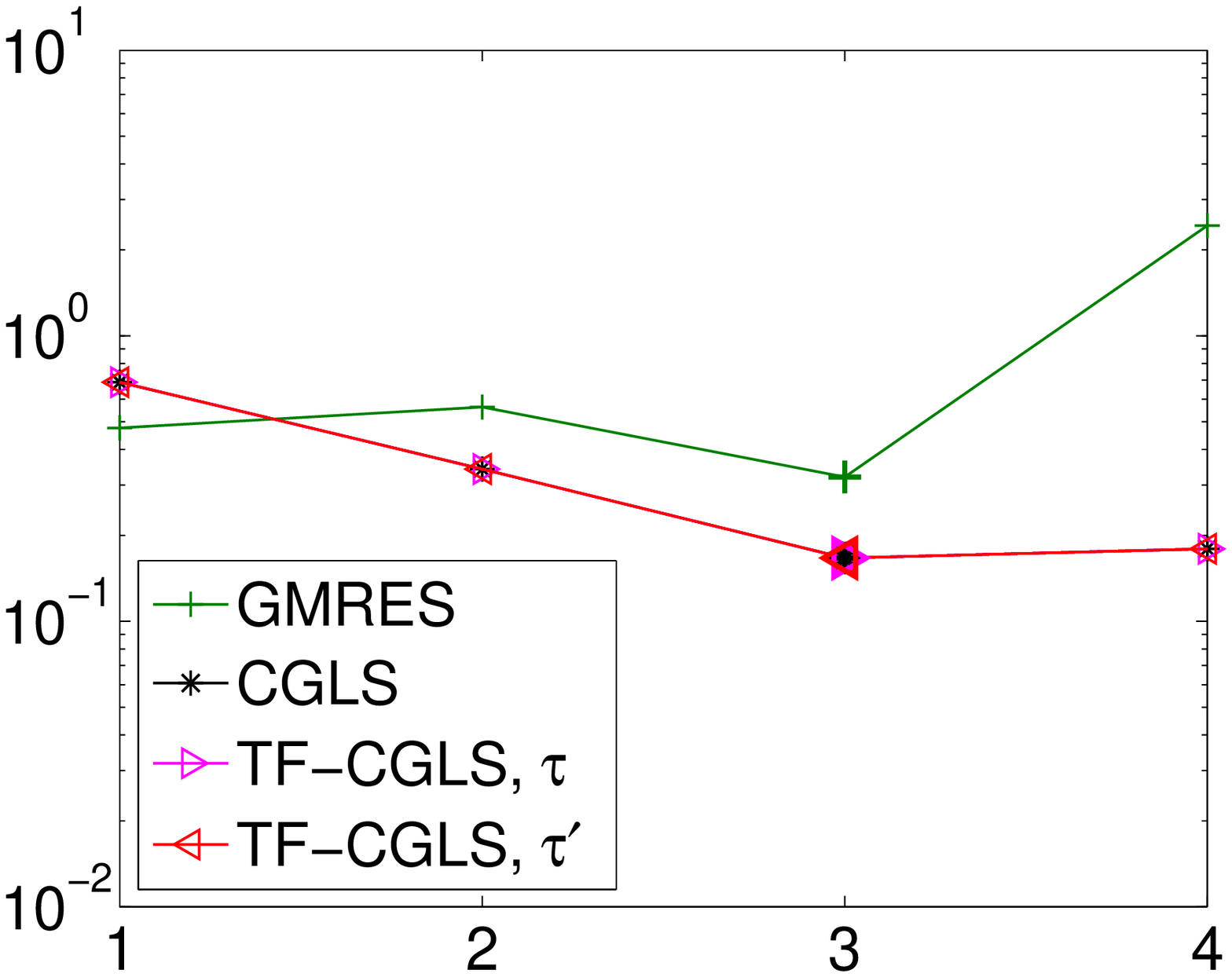} & %
\hspace{-0.7cm}\includegraphics[width=6.2cm]{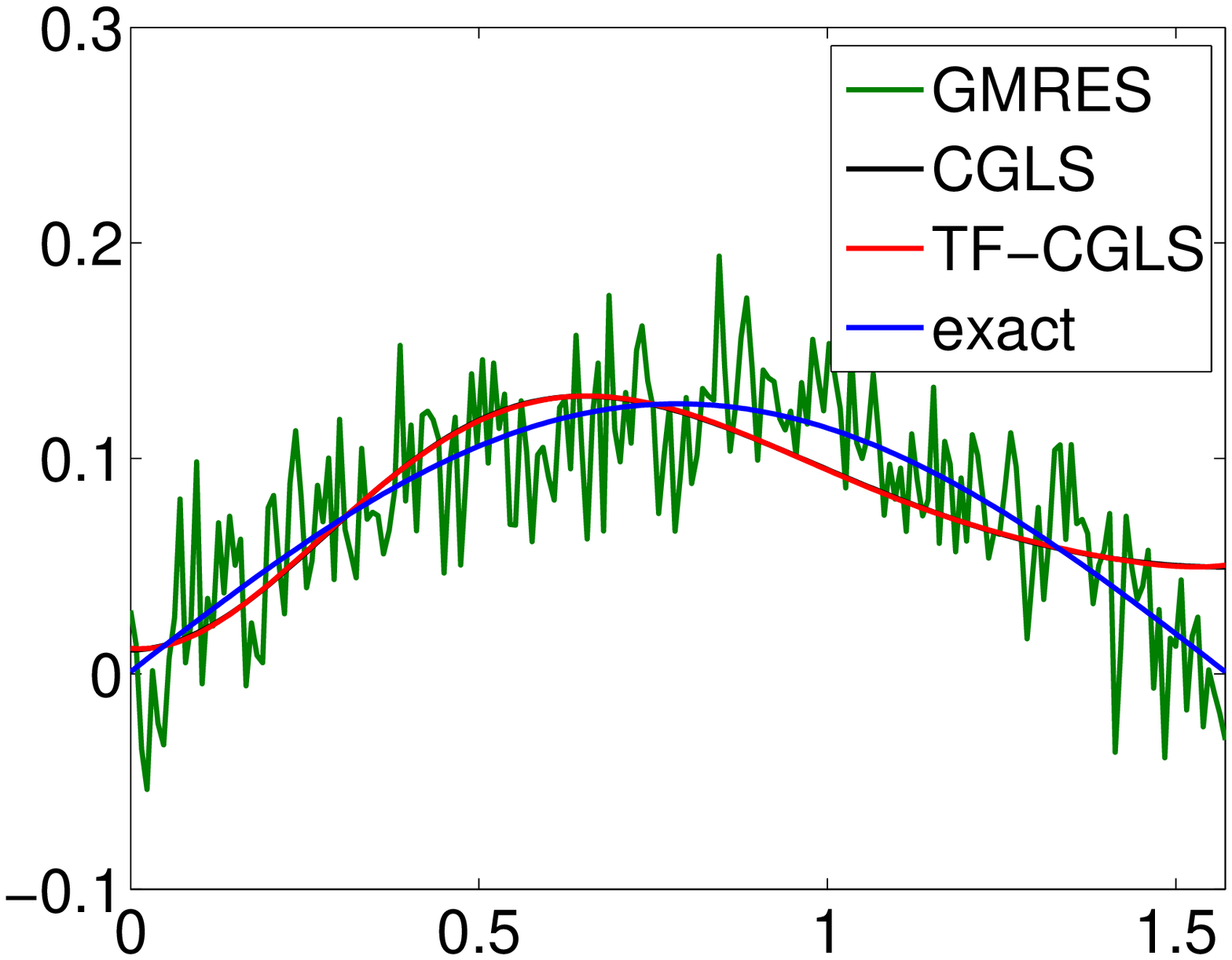}%
\end{tabular}%
\caption{Test problem \texttt{baart}, with $N=200$ and $\weps = 10^{-2}$.
\textbf{(a)} Relative error history. \textbf{(b)} Best
approximations achieved by the GMRES, CGS, and TF-CGLS methods.}
\label{fig:baart}
\end{figure}
The left frame of Figure \ref{fig:baart} shows the history of the relative errors for the GMRES, CGLS, and TF-CGLS methods, while its right frame displays the best reconstructions obtained by each method. The GMRES performance on the \texttt{baart} test problem is usually good, thanks to its favorable approximation subspace (see the analysis in \cite{JH07}). Indeed, the GMRES relative error is fairly low, and the behavior of the solution is somewhat recovered. However, the GMRES approximate solution has an evident oscillating behavior: this is due to a heavy presence of noise in the approximation subspace, and this drawback can be partially fixed by considering the range-restricted GMRES method \cite{RRGMRESfirst}. It is also clear that, for this test problem, the performance of the TF-CGLS and CGLS method is almost identical (the corresponding lines coincide in the graphs of Figure \ref{fig:baart}). Numerical values for this test problem are reported in Table \ref{tab:expo}.
\item \texttt{\textbf{heat}}. We consider a discretization of the inverse heat equation formulated as a Volterra integral equation of the first kind, as provided within \cite{RegT}. We choose $N=200$, so that $\|A-A^T\|/\|A\|=1.1244$; this problem can be regarded as numerically rank-deficient, with numerical rank equal to 195. According to the analysis in \cite{JH07}, GMRES does not converge to $A^\dagger b$ for this problem, as the null space of $A$ and $A^T$ are different. The stopping criteria (\ref{stopO}) with $\tau = 10^{-10}$, and (\ref{stopOvar3}) with $\tau^{\prime}=10^{-14}$, are both satisfied after 40 Arnoldi iterations (i.e., the maximum allowed number of iterations for this set of experiments). Because of this, the results can be affected by the loss of orthogonality in the Arnoldi vectors and, in order to assess the impact of this phenomenon, we consider the performance of both the modified Gram-Schmidt (MGS) and the Householder (HH) implementations of the Arnoldi algorithm.
\begin{figure}[tbp]
\centering
\begin{tabular}{cc}
\hspace{-0.7cm}\textbf{{\small {(a)}}} & \hspace{-0.7cm}\textbf{{\small {(b)}}} \\
\hspace{-0.7cm}\includegraphics[width=6.2cm]{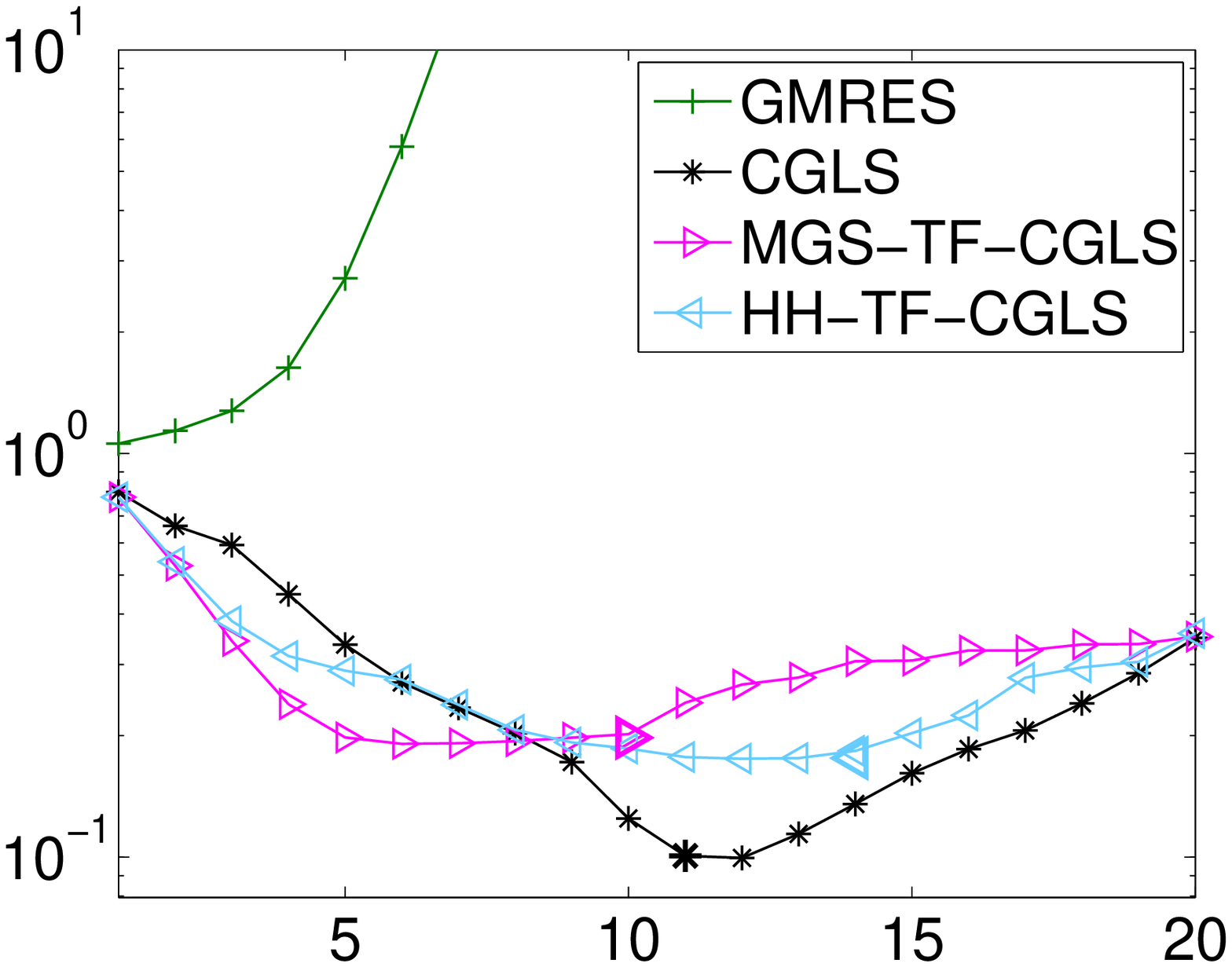} & %
\hspace{-0.7cm}\includegraphics[width=6.2cm]{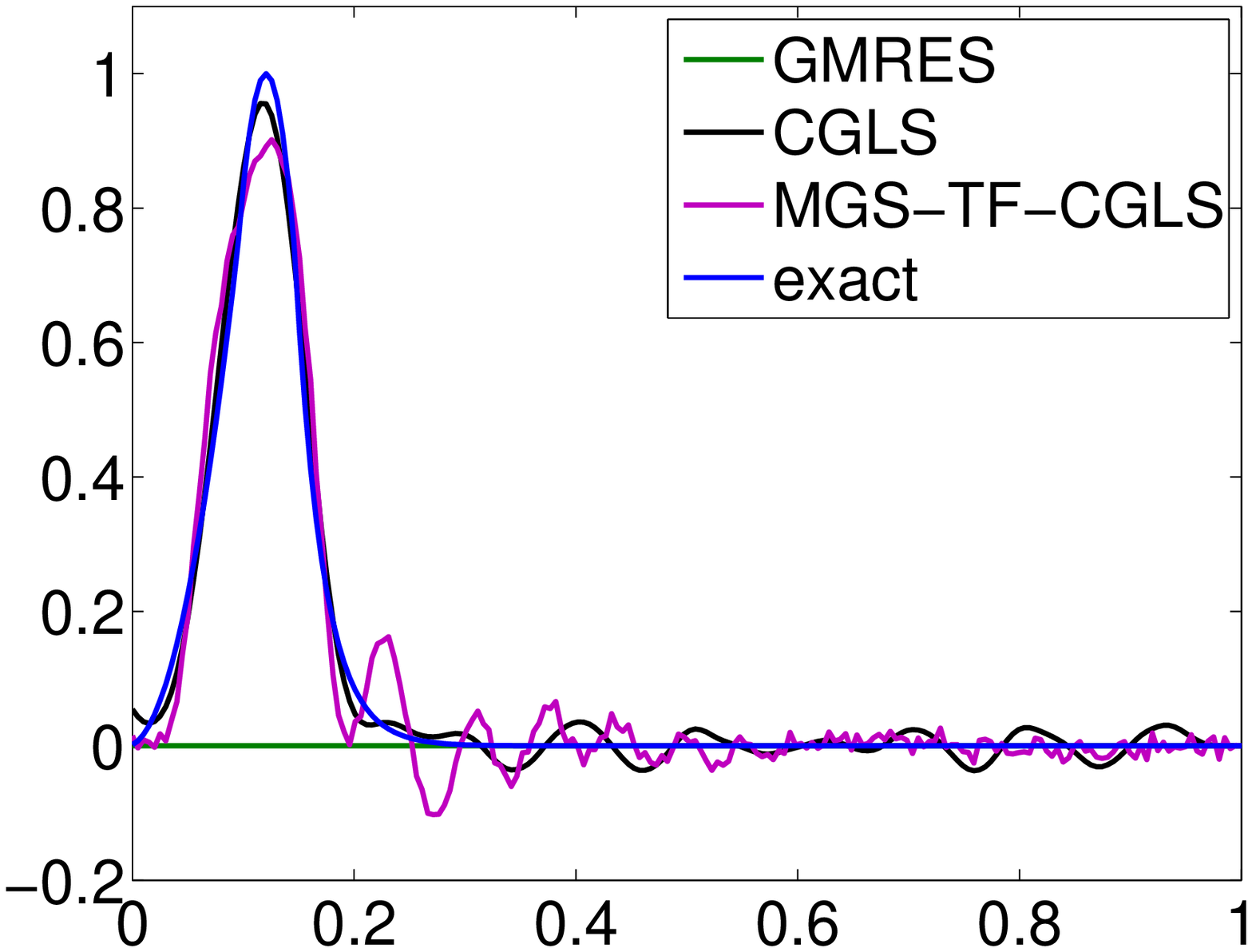}%
\end{tabular}%
\caption{Test problem \texttt{heat}, with $N=200$, and $\weps = 10^{-2}$.
\textbf{(a)} Relative error history, comparing GMRES, CGLS, TF-CGLS with modified Gram-Schmidt implementation (MGS-TF-CGLS), and TF-CGLS with Householder implementation (HH-TF-CGLS). \textbf{(b)} Best
approximations achieved by the GMRES, CGLS, and HH-TF-CGLS methods.}
\label{fig:heat}
\end{figure}
The left frame of Figure \ref{fig:heat} shows the history of the relative errors for GMRES, CGLS, and TF-CGLS (with both the MGS and HH implementations). While CGLS delivers the best approximations, the quality of the TF-CGLS solutions is much better than the GMRES ones (which diverge). Moreover, the TF-CGLS solutions obtained by the MGC and the HH implementations are comparable. The right frame of Figure \ref{fig:heat} displays the most accurate approximations obtained by each method: in the GMRES case, this is the zero solution (i.e., the initial guess); the TF-CGLS solution displays slightly more oscillations than the CGLS one, and this shortcoming might be partially remedied by including additional (standard form) Tikhonov regularization within the TF-CGLS iteration (in an hybrid-like fashion, see Remark \ref{rem:hybrid}).
%
%
\begin{table}
\caption{Average results over 20 runs of some of the test problems in the first set of experiments, with $\weps=10^{-2}$. The TF-CGLS relative error is the one atteined when stopping criteria (\ref{stopOvar3}) and (\ref{stopI}) are satisfied; the GMRES and CGLS ones are computed at the iterations (\ref{stopI}).}
\label{tab:expo}       
\begin{tabular}{lllll}
\hline\noalign{\smallskip}
 & Relative Error & (\ref{stopI}) & (\ref{stopO}) & (\ref{stopOvar3})\\
\noalign{\smallskip}\hline\noalign{\smallskip}
& \multicolumn{4}{c}{first \texttt{i\_laplace} test problem, $N=100$}\\
\noalign{\smallskip}\hline\noalign{\smallskip}
GMRES & $6.1034\cdot 10^{-1}$ & $5.1$ & - & - \\
CGLS & $1.5342\cdot 10^{-1}$ & $5.3$ & - & - \\
TF-CGLS & $1.5358\cdot 10^{-1}$ & $5.3$ & $20.5$ & $19.4$\\
\noalign{\smallskip}\hline\noalign{\smallskip}
& \multicolumn{4}{c}{second \texttt{i\_laplace} test problem, $N=100$}\\
\noalign{\smallskip}\hline\noalign{\smallskip}
GMRES & $3.0486\cdot 10^0$  & 7.1  & - & -  \\
CGLS & $7.5968\cdot 10^{-2}$ & 5 & - & -  \\
TF-CGLS & $7.6011\cdot 10^{-2}$ & 5 & 20.2 & 19.5\\
\noalign{\smallskip}\hline\noalign{\smallskip}
& \multicolumn{4}{c}{\texttt{baart}, $N=200$}\\
\noalign{\smallskip}\hline\noalign{\smallskip}
GMRES & $5.6460\cdot 10^{-1}$ & $3$  & - & - \\
CGLS & $1.6704\cdot 10^{-1}$ & $3$ & - & - \\
TF-CGLS & $1.6719\cdot 10^{-1}$ & $3$ & $8.7$ & $16.5$\\
\noalign{\smallskip}\hline
\end{tabular}
\end{table}
\end{enumerate}
\paragraph{Second set of experiments.}
We consider 2D image restoration problems, where the available images are affected by a spatially invariant blur and Gaussian white noise. In this setting, given a point-spread function (PSF) that describes how a single 
pixel is deformed, a blurring process is modeled as a 2D convolution of the PSF and an exact discrete finite image $\Xex\in\R^{n\times n}$. Here and in the following, a PSF is represented as a 2D image $P\in\R^{q\times q}$, with $q\ll n$, typically. One can immediately see that, if $P_{i,j}\neq 0$, $i,j=1,\dots,q$, the deblurring problem is underdetermined since, when convolving $P$ with $\Xex$,
additional (and unavailable) values of the exact image outside $\Xex$ should be considered.
A popular approach to overcome this phenomenon is to impose boundary conditions within the blurring process, i.e., to prescribe the behavior of the exact image outside $\Xex$ (see \cite{BerNagy13} and the references therein). A 2D image restoration problem can be rewritten as a linear system (\ref{sys}), where the 1D array $b$ is obtained by stacking the columns of the 2D blurred and noisy image (so that $N=n^2$), and the square matrix $A$ incorporates the convolution process together with the boundary conditions. Although popular choices such as zero or periodic boundary conditions are particularly simple, they often give rise to unwanted artifacts during the restoration process. The use of reflective or anti-reflective boundary conditions usually gives better results, as a sort of continuity of the image outside $\Xex$ is imposed (and, in the anti-reflective case, also a sort of continuity of the normal derivative). Antireflective boundary conditions (ARBC) were originally introduced in \cite{ARBCfirst}, and further analyzed in several papers (see \cite{DMR15} and the references therein). When dealing with a nonsymmetric PSF and ARBC, matrix-vector products with $A$ can be implemented by fast algorithms, but the same is not true for matrix-vector products with $A^T$ (as, to the best of our knowledge, there is no known algorithm that can efficiently exploit the structure of $A^T$). Therefore, in practice, $A^T$ is often approximated by a matrix $A^\prime$ defined by first rotating of $180\degree$ the PSF $P$ used to build $A$ (so to obtain the PSF $P^\prime$) and then modeling the 2D convolution process with $P^\prime$ and ARBC. In other words, image deblurring problems with a nonsymmetric PSF and ARBC can be only handled by transpose-free solvers. As addressed in Section \ref{sect:intro}, the authors of \cite{DMR15} propose to solve the equivalent, and somewhat symmetrized, linear system $AA^\prime y=b$ (with $x=A^\prime y$) by GMRES: in the following, this method is referred to as \tql RP-GMRES\tqr. In this set of experiments we compare the GMRES and RP-GMRES methods with the TF-CGLS method, in order to assess if approximating $A^T$ by $\Amp$  (\ref{Aprime}) guarantees restored images of improved quality.
Our experiments are created by considering three different grayscale test images of size $256\times 256$ pixels, together with three different PSFs, and antireflective or reflective boundary conditions; the sharp images are artificially blurred, and noise of variable levels is added. Matrix-vector products are computed efficiently by using the routines in \emph{Restore Tools} \cite{RestTools}\footnote{An extension to handle ARBC within \emph{Restore Tools} is available at:\\\small\url{http://scienze-como.uninsubria.it/mdonatelli/Software/software.html}.}. In the first and second experiment, the blurred image is cropped in order to reduce the effect of the chosen boundary conditions (and not to commit \tql inverse crime\tqr, see \cite[Chapter 7]{PCH10}). The maximum number $\mmax$ of Arnoldi iterations for Algorithm \ref{alg:tfCG} is set to 50, and only the stopping criterion (\ref{stopOvar3}) is considered. The Gram-Schmidt implementation of the Arnoldi algorithm is tested.
\begin{enumerate}
\item \textbf{Anisotropic Gaussian blur.} For this experiment, the elements $P_{i,j}$ of the PSF $P$ are analytically given by the following expression
\[
p_{i,j} = \exp\left(-\frac{1}{2(s_1^2s_2^2-\rho^4)}\left(s_2^2(i-k)^2-2\rho^2(i-k)(j-\ell)+s_1^2(j-\ell)^2\right)\right)\,,
\]
where $i,j=1,\dots,d$, and $[k,\ell]$ is the center of the PSF. The values $s_1=4$, $s_2=1.3$, $\rho=2$, and $d=21$ are considered, and the noise level is $\weps=2\cdot10^{-2}$. ARBC are imposed. The test data are displayed in Figure \ref{fig:camerset}. Figure \ref{fig:camerrec} shows the best restorations achieved by each method; relative errors and the corresponding number of iterations are displayed in the caption.
\begin{figure}[tbp]
\centering
\begin{tabular}{ccc}
\hspace{-1.2cm}\textbf{{\small {exact}}} &
\hspace{-1.2cm}\textbf{{\small {PSF}}} &
\hspace{-1.2cm}\textbf{{\small {corrupted}}} \\
\hspace{-1.2cm}\includegraphics[width=5.2cm]{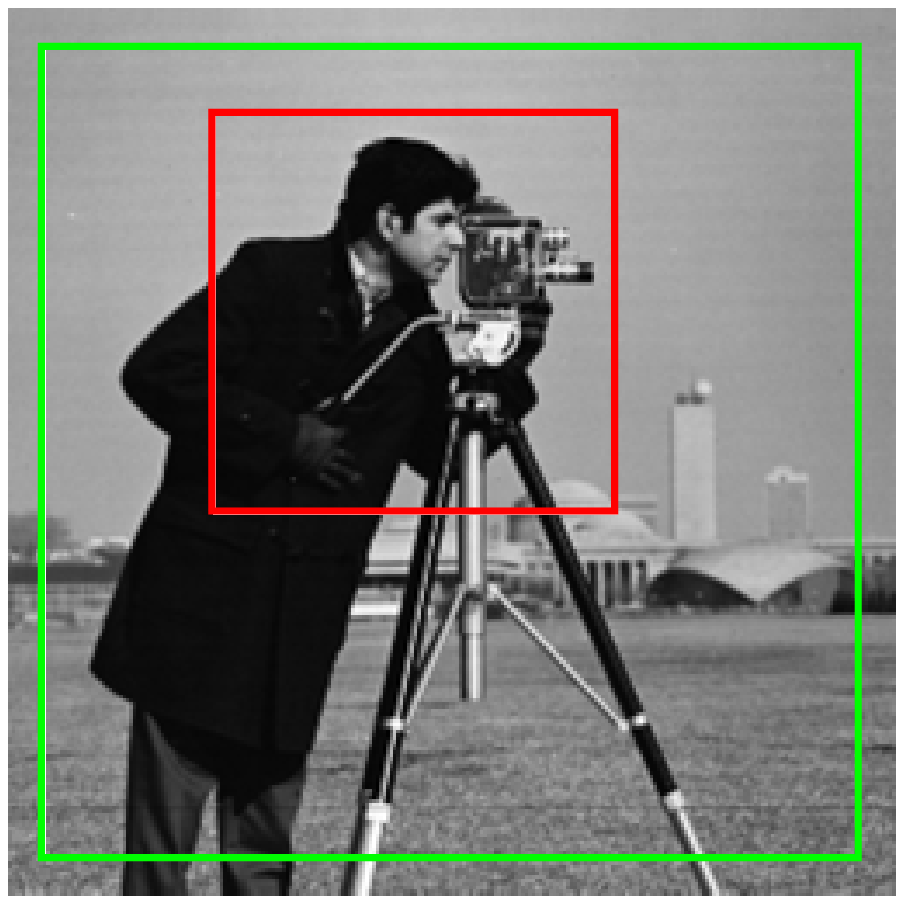} &
\hspace{-1.2cm}\includegraphics[width=5.2cm]{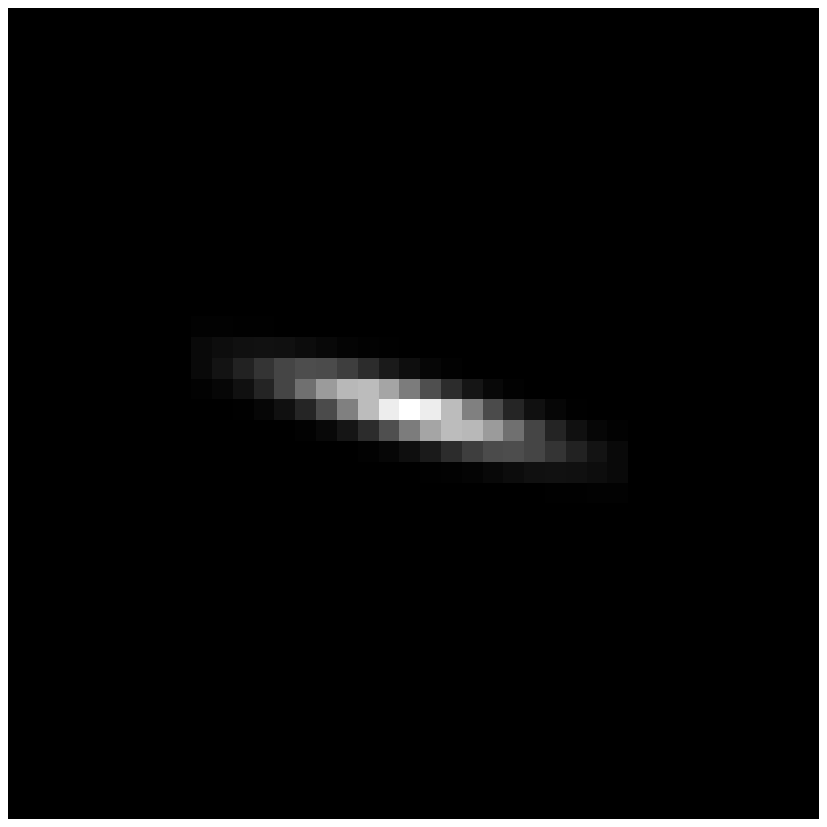} &
\hspace{-1.2cm}\includegraphics[width=5.2cm]{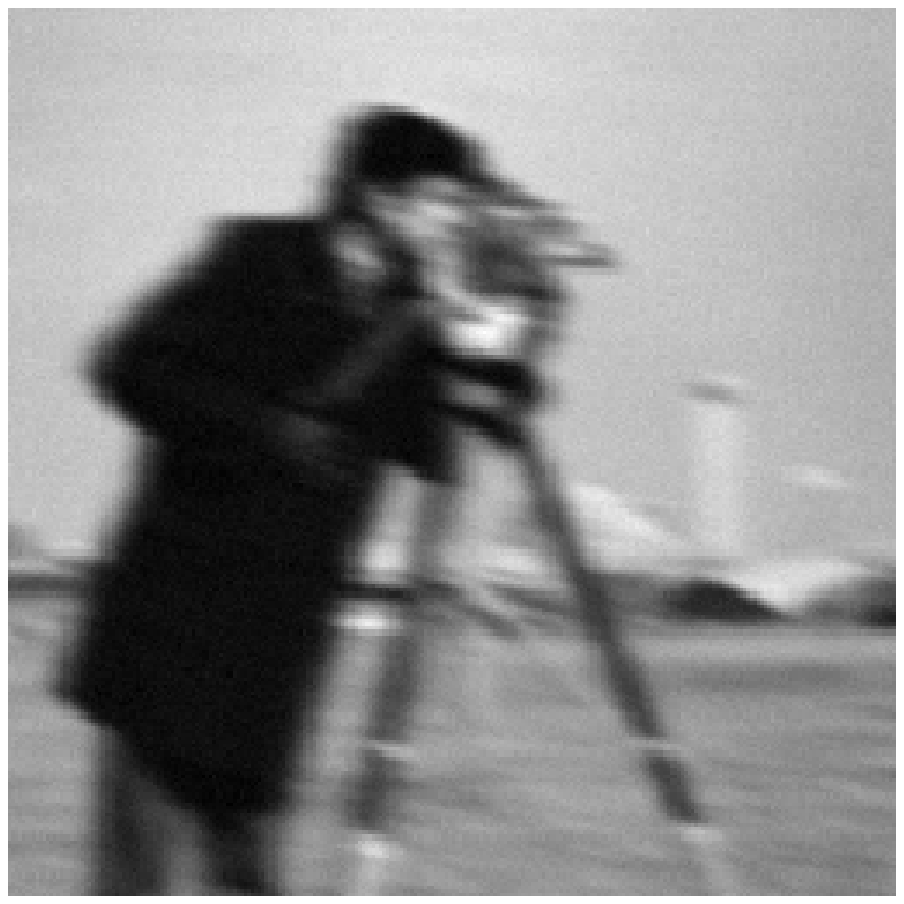}
\end{tabular}%
\caption{From left to right: exact image, where the cropped portion is highlighted; blow-up ($600\%$) of the anisotropic Gaussian PSF; blurred and noisy available image, with $\weps=2\cdot10^{-2}$.}
\label{fig:camerset}
\end{figure}
\begin{figure}[tbp]
\centering
\begin{tabular}{ccc}
\hspace{-1.2cm}\textbf{{\small {GMRES}}} &
\hspace{-1.2cm}\textbf{{\small {TF-CGLS}}} &
\hspace{-1.2cm}\textbf{{\small {RP-GMRES}}} \\
\hspace{-1.2cm}\includegraphics[width=5.2cm]{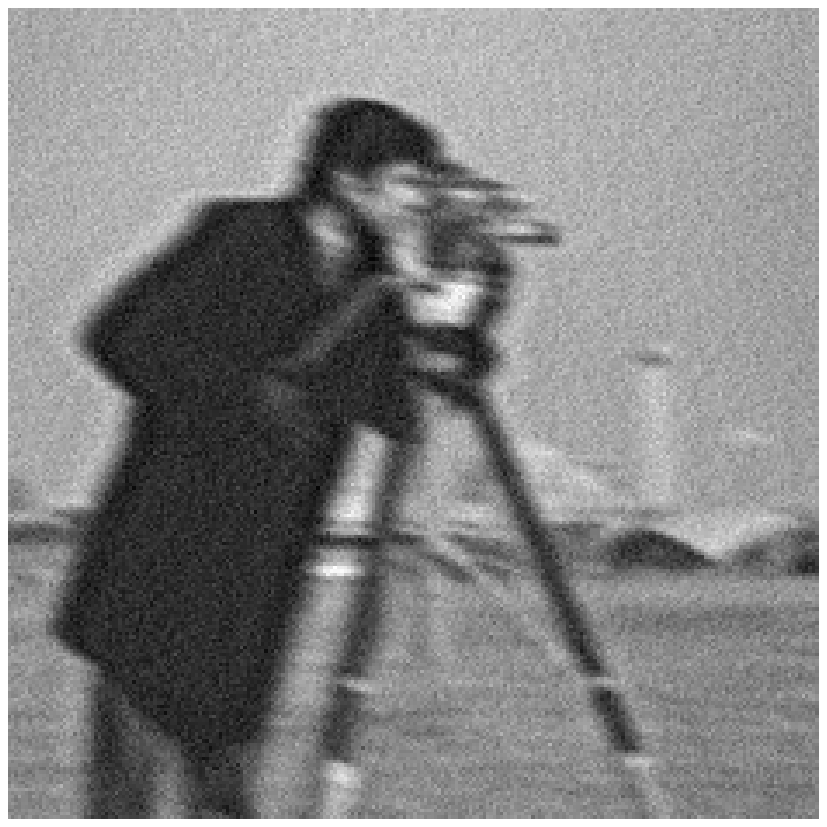} &
\hspace{-1.2cm}\includegraphics[width=5.2cm]{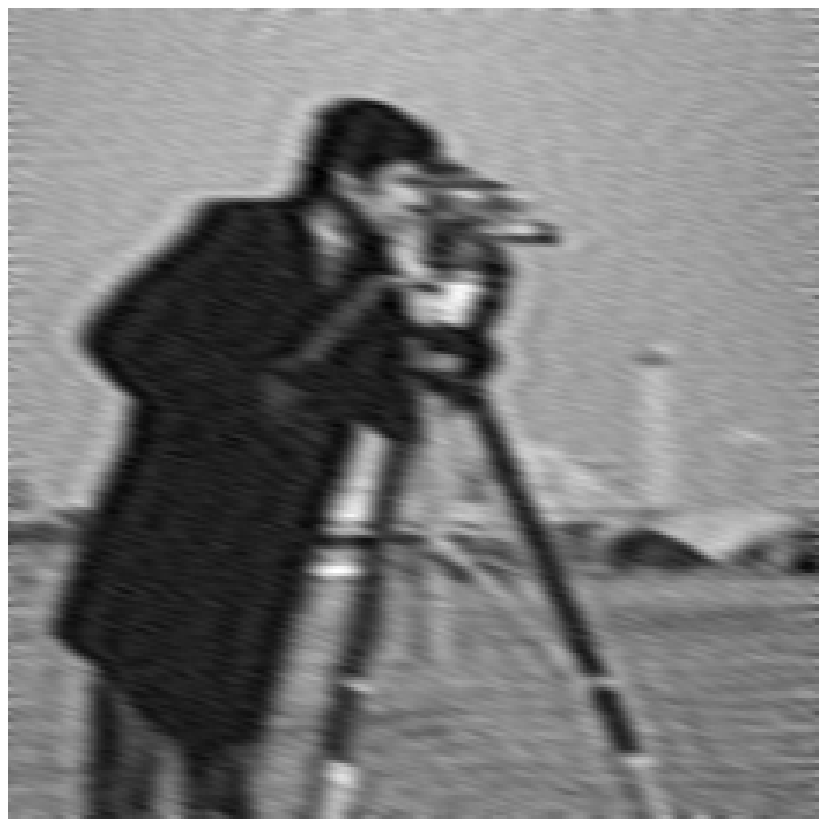} &
\hspace{-1.2cm}\includegraphics[width=5.2cm]{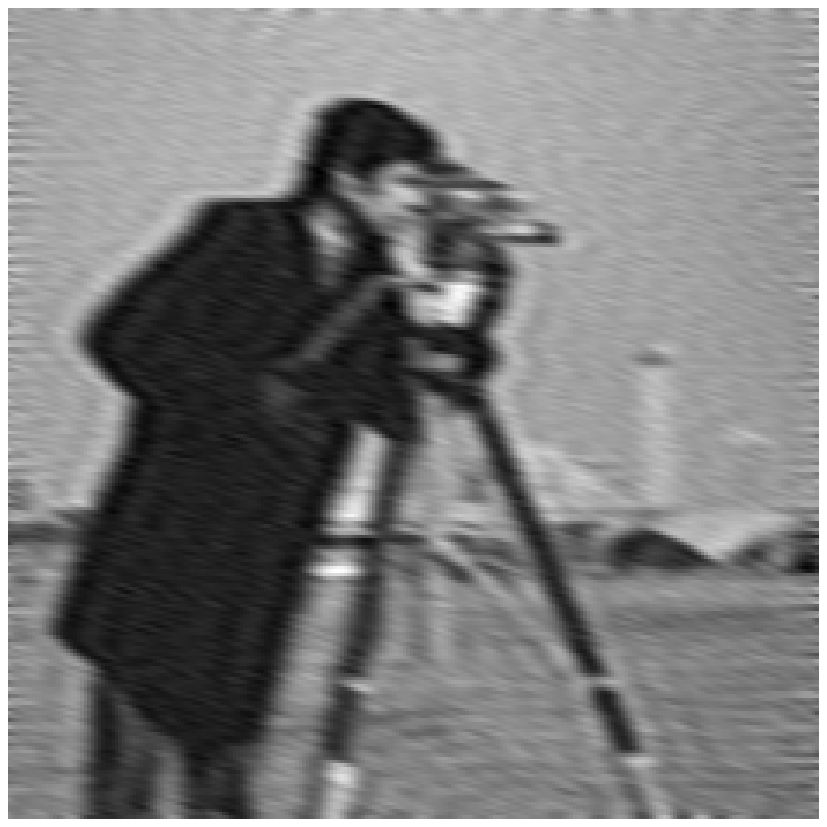}\vspace{-0.5cm}\\
\hspace{-1.2cm}\includegraphics[width=5.2cm]{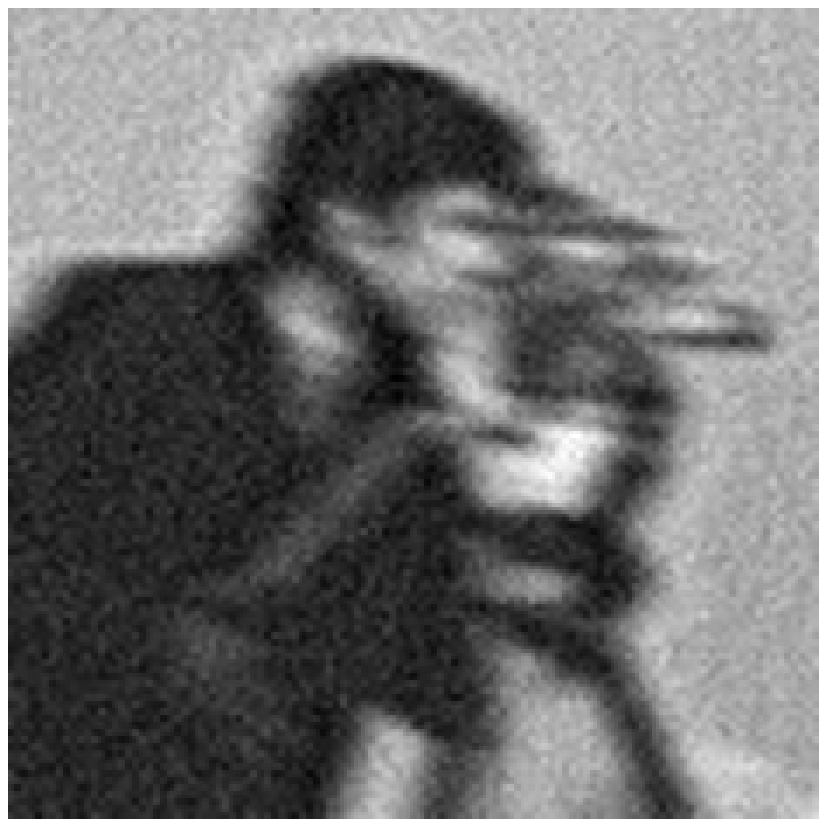} &
\hspace{-1.2cm}\includegraphics[width=5.2cm]{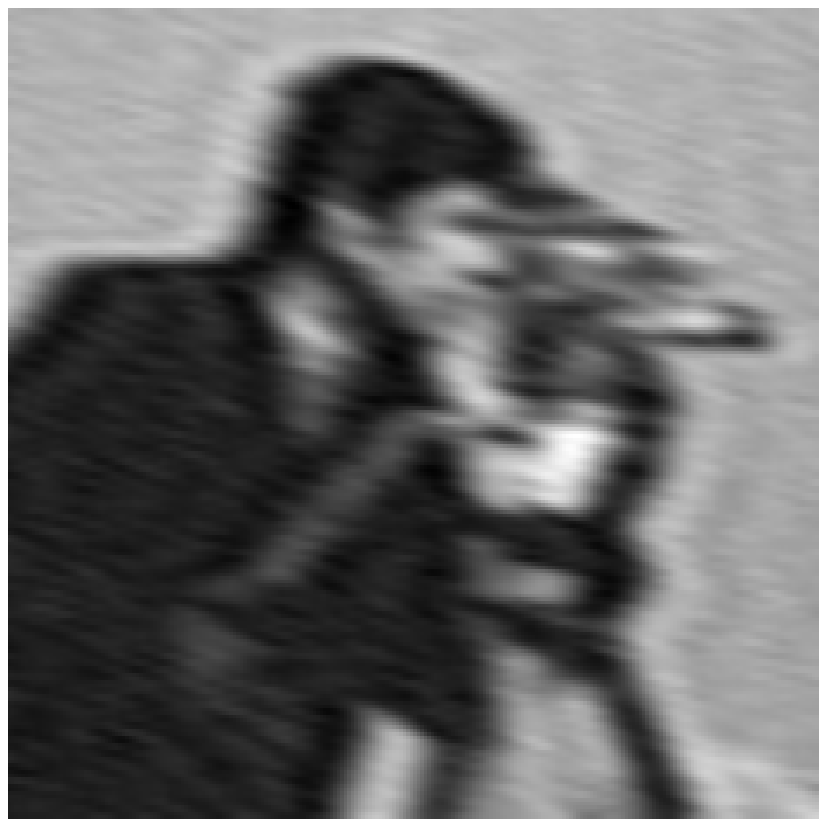} &
\hspace{-1.2cm}\includegraphics[width=5.2cm]{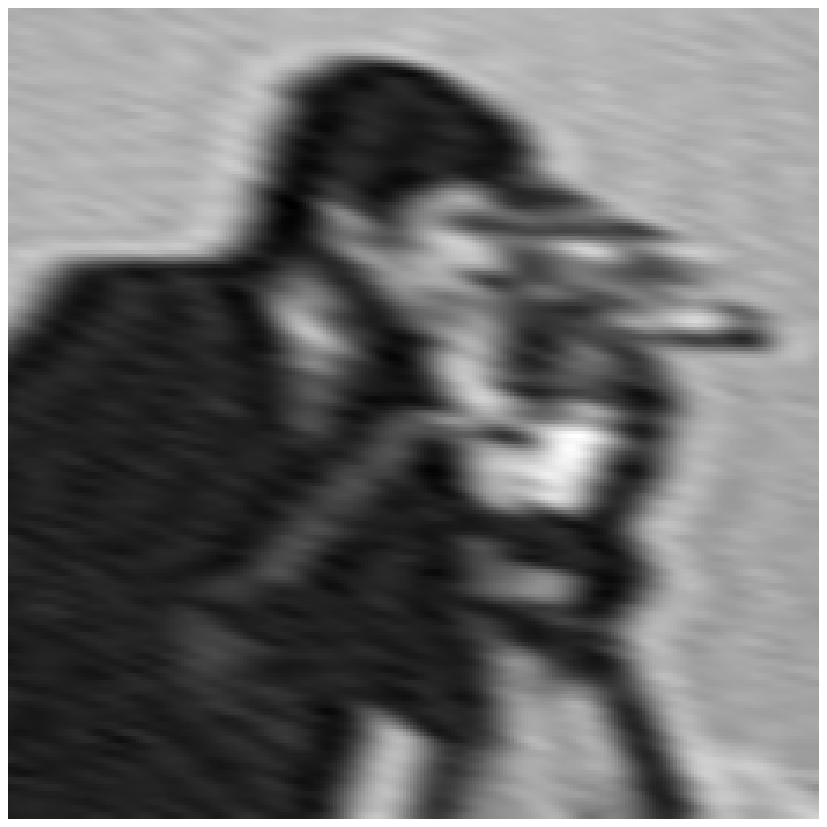}
\end{tabular}
\caption{The lower row displays blow-ups ($200\%$) of the restored images in the upper row. From left to right: standard GMRES method ($0.1483$, $m=4$); TF-CGLS method ($0.1344$, $m=14$, $k=12$); right-preconditioned GMRES ($0.1354$, $m=19$).}
\label{fig:camerrec}
\end{figure}
Looking at the restored images, it is evident that the GMRES one still appears pretty noisy and blurred. Though the relative error for the RP-GMRES restoration is slightly lower than the TF-CGLS one, the two images are visually very similar, and there is great improvement over the standard GMRES restoration. It should also be emphasized that the cost of each RP-GMRES iteration is dominated by two matrix-vector products (one with $A$, and one with $A^\prime$). Therefore, the cost of computing the GMRES, TF-CGLS, and RP-GMRES restorations is dominated by 4, 14, and 38 matrix-vector products, respectively. For this experiment, TF-CGLS can deliver a solution whose quality is almost identical to the RP-GMRES one, with great computational savings.
\item \textbf{Motion blur.} The test data for this experiment are displayed in Figure \ref{fig:boatset}. We consider a $17\times 17$ PSF modeling diagonal motion blur, and ARBC are imposed. The noise level is $\weps = 5\cdot 10^{-3}$. Figure \ref{fig:boatrec} shows the best restorations achieved by each method; relative errors and the corresponding number of iterations are displayed in the caption.
\begin{figure}[tbp]
\centering
\begin{tabular}{ccc}
\hspace{-1.2cm}\textbf{{\small {exact}}} &
\hspace{-1.2cm}\textbf{{\small {PSF}}} &
\hspace{-1.2cm}\textbf{{\small {corrupted}}} \\
\hspace{-1.2cm}\includegraphics[width=5.2cm]{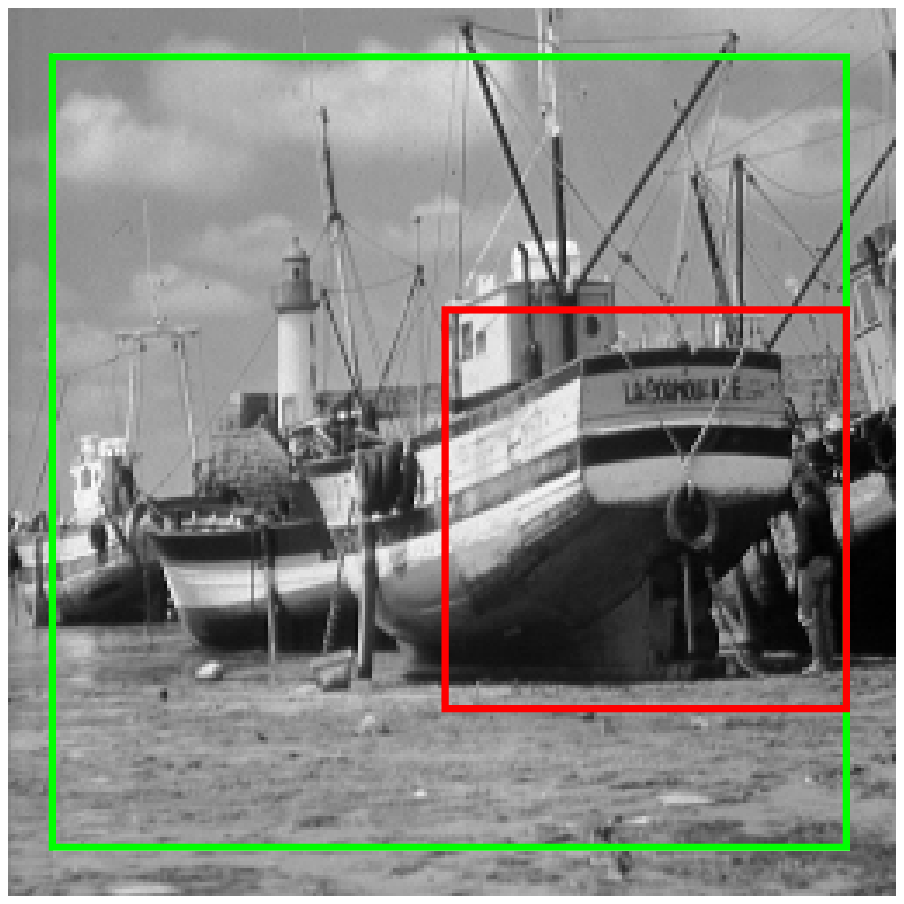} &
\hspace{-1.2cm}\includegraphics[width=5.2cm]{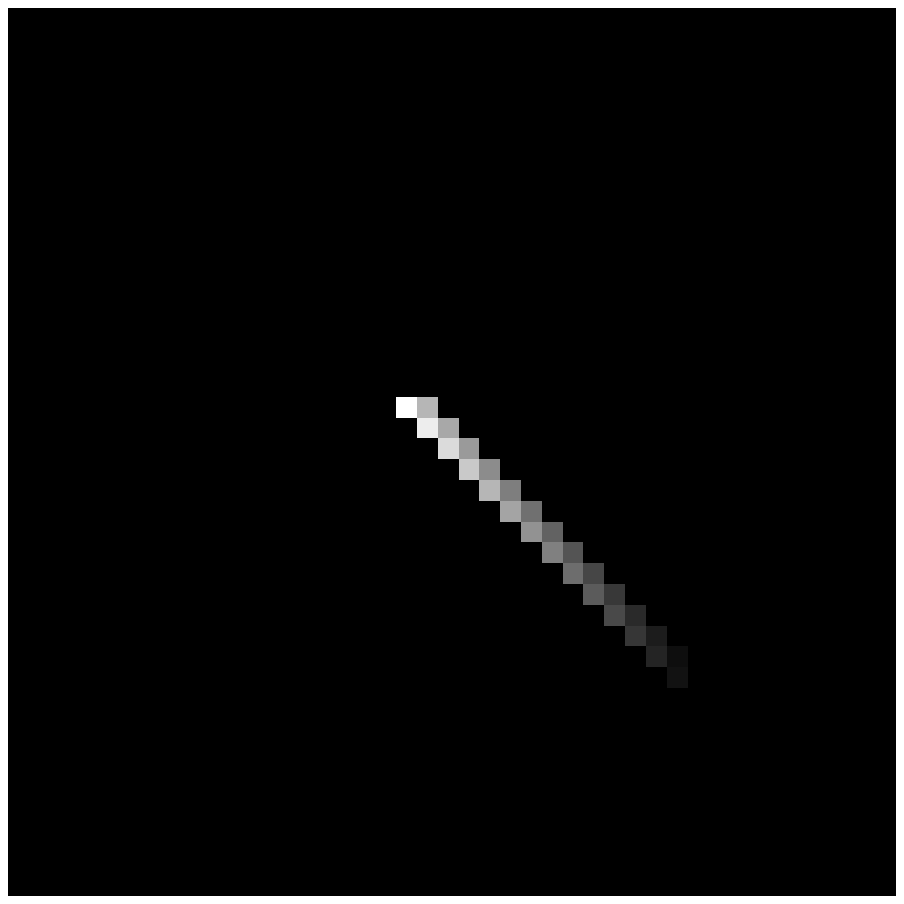} &
\hspace{-1.2cm}\includegraphics[width=5.2cm]{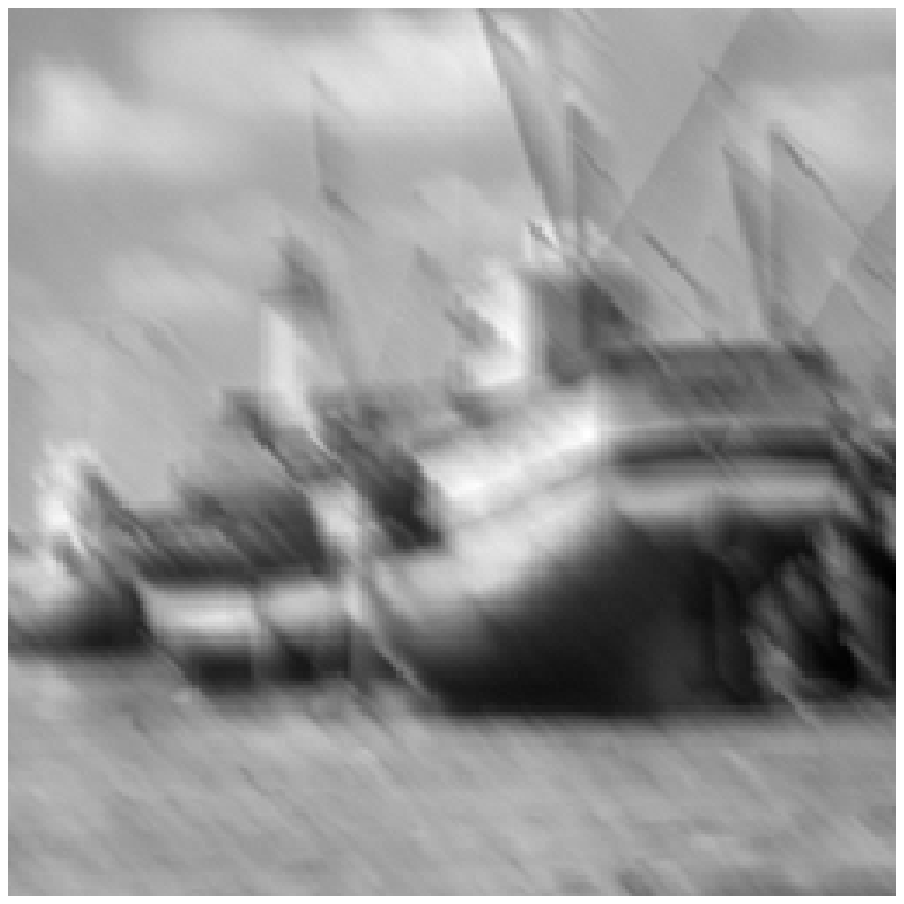}
\end{tabular}%
\caption{From left to right: exact image, where the cropped portion is highlighted; blow-up ($600\%$) of the diagonal motion PSF; blurred and noisy available image, with $\weps=5\cdot 10^{-3}$.}
\label{fig:boatset}
\end{figure}
\begin{figure}[tbp]
\centering
\begin{tabular}{ccc}
\hspace{-1.2cm}\textbf{{\small {GMRES}}} &
\hspace{-1.2cm}\textbf{{\small {TF-CGLS}}} &
\hspace{-1.2cm}\textbf{{\small {RP-GMRES}}} \\
\hspace{-1.2cm}\includegraphics[width=5.2cm]{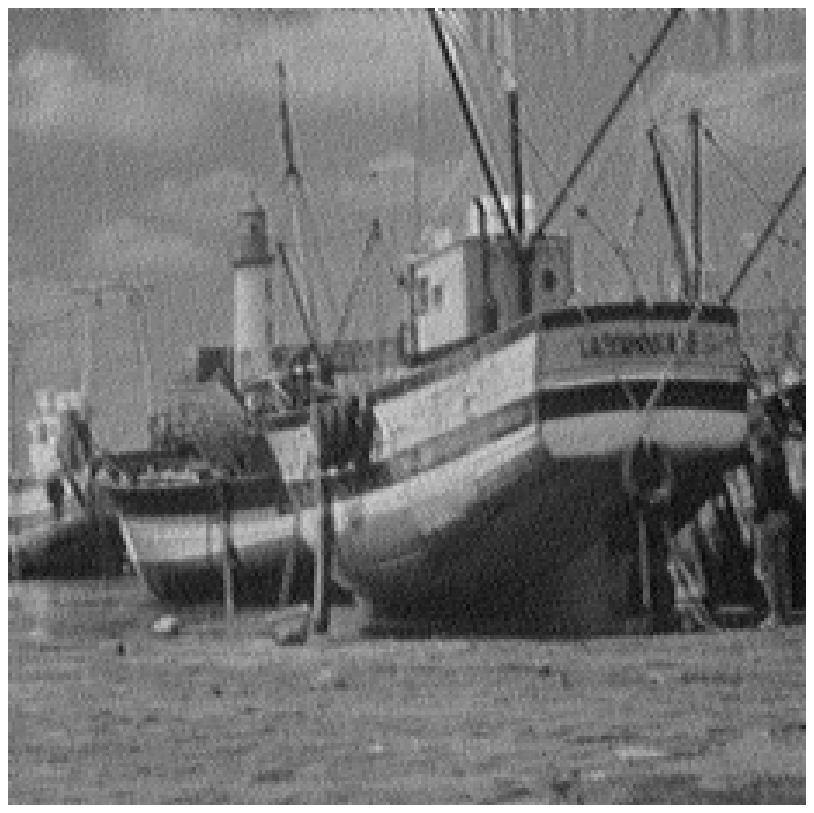} &
\hspace{-1.2cm}\includegraphics[width=5.2cm]{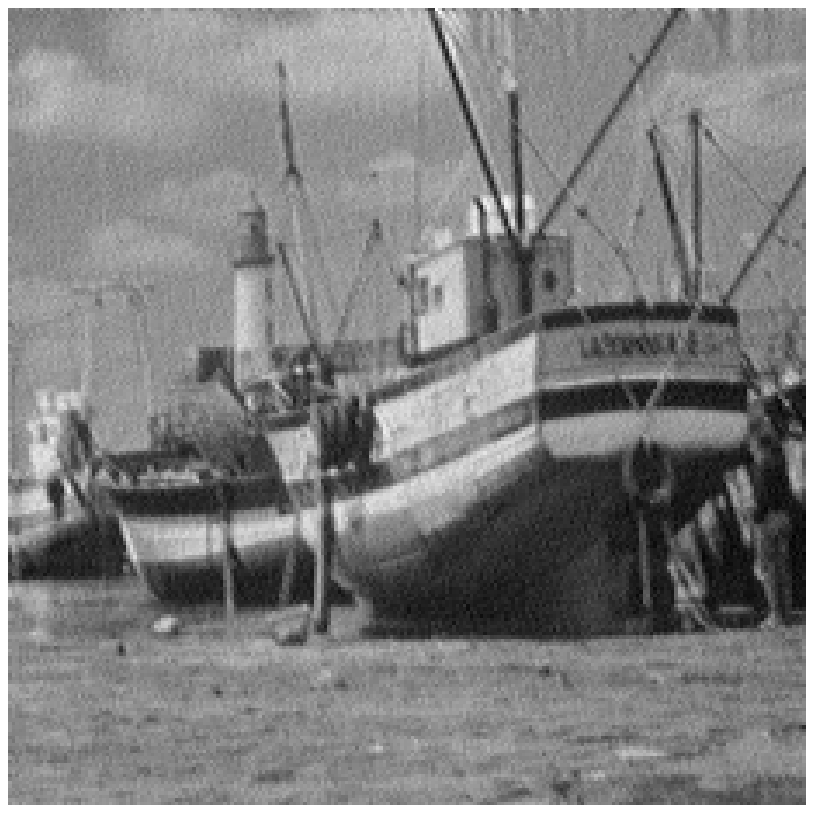} &
\hspace{-1.2cm}\includegraphics[width=5.2cm]{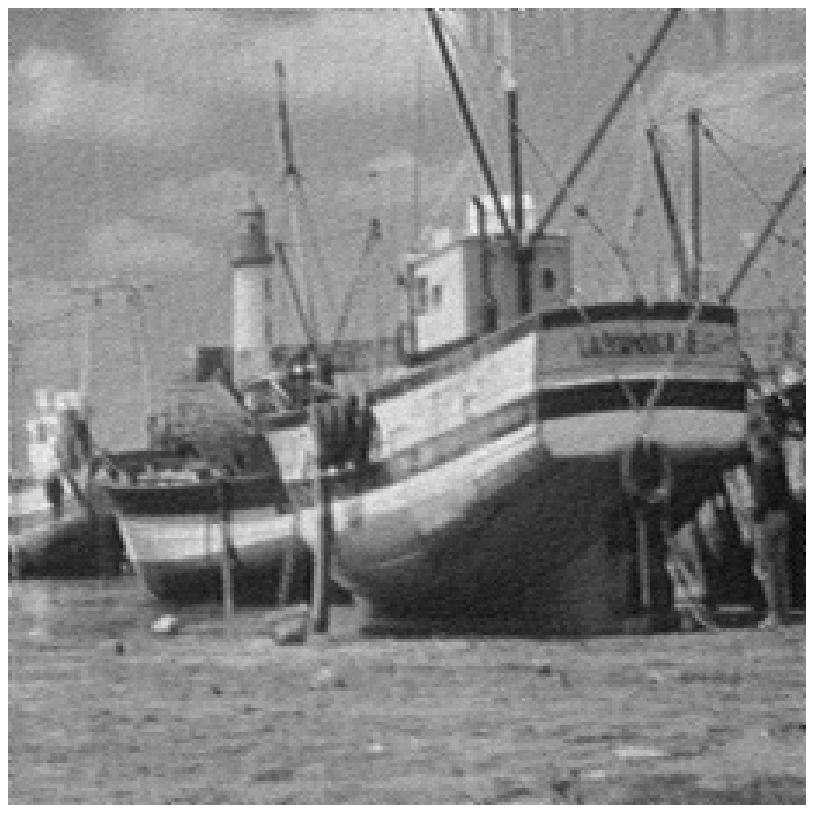}\vspace{-0.5cm}\\
\hspace{-1.2cm}\includegraphics[width=5.2cm]{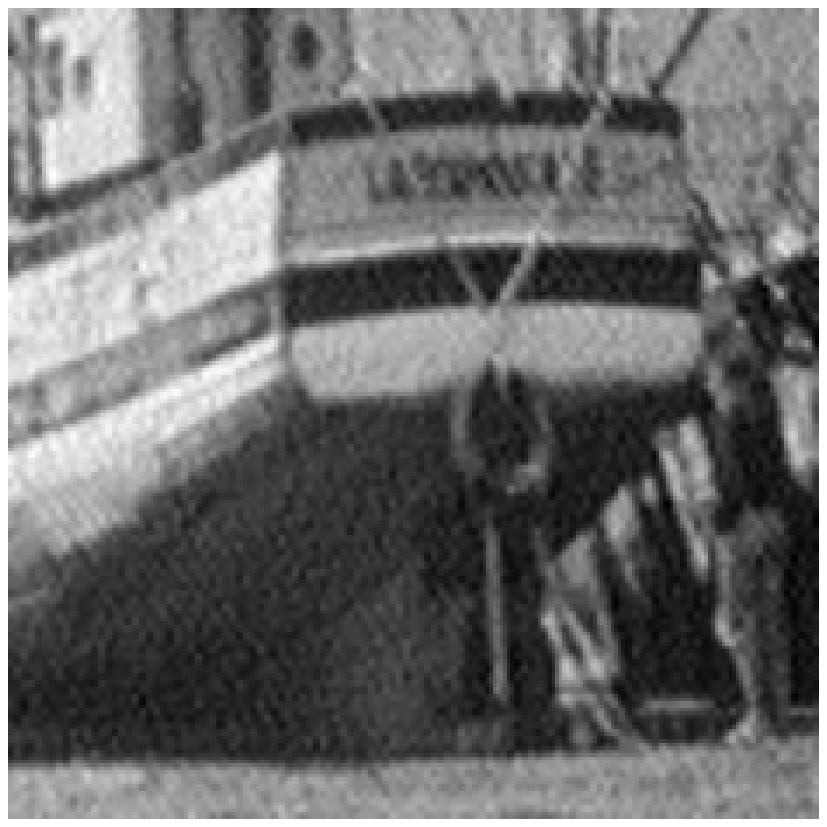} &
\hspace{-1.2cm}\includegraphics[width=5.2cm]{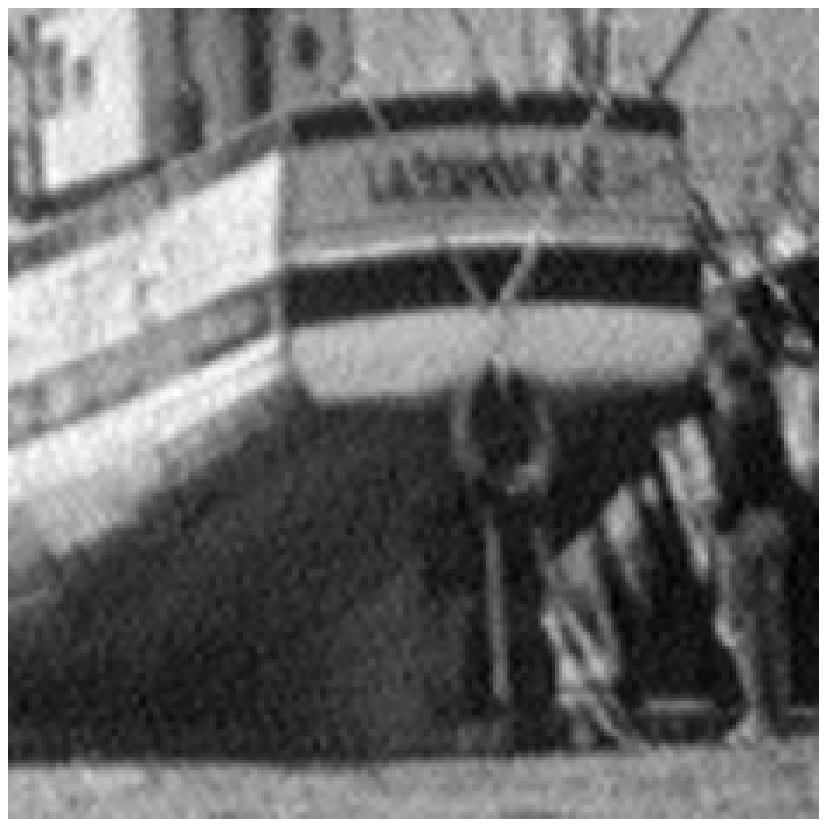} &
\hspace{-1.2cm}\includegraphics[width=5.2cm]{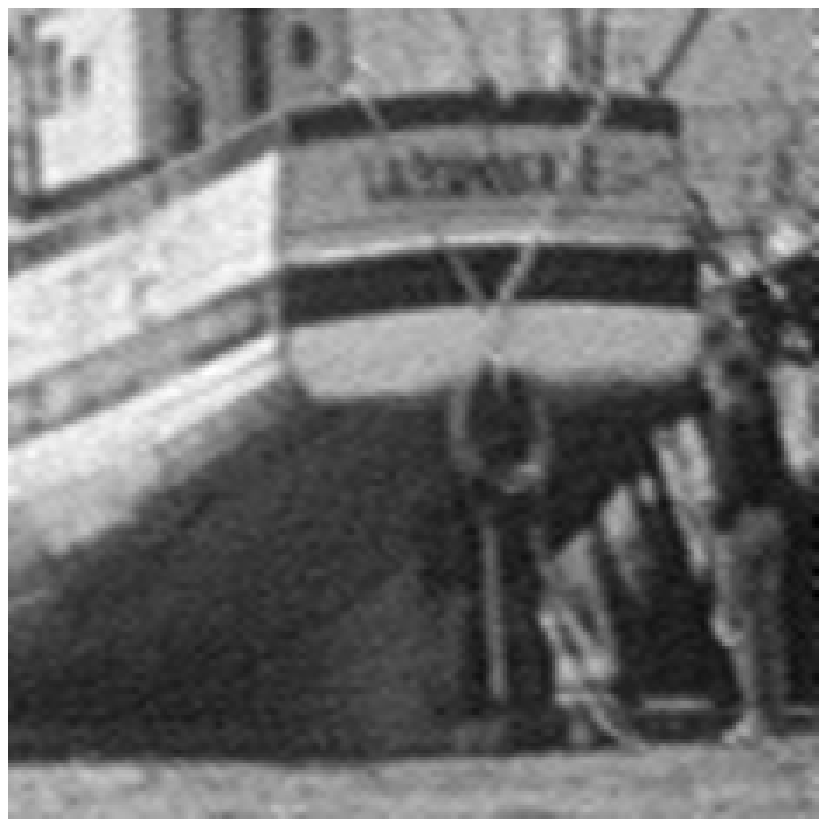}
\end{tabular}
\caption{The lower row displays blow-ups ($200\%$) of the restored images in the upper row. From left to right: standard GMRES method ($8.5215\cdot 10^{-2}$, $m=40$); TF-CGLS method ($8.2764\cdot 10^{-2}$, $m=40$, $k=9$); right-preconditioned GMRES ($6.0439\cdot 10^{-2}$, $m=25$).}
\label{fig:boatrec}
\end{figure}
With respect to the previous experiment, all the methods perform more iterations, due to the lower amount of noise chosen for this problem. The computational  cost of the GMRES, TF-CGLS, and RP-GMRES methods is dominated by 40, 44, and 50 matrix-vector products, respectively. By visually inspecting the images in Figure \ref{fig:boatrec}, we can see the GMRES solution to bear some motion artifacts, as the restored image displays some shifts in the diagonal directions, i.e., in the direction of the motion blur (this agrees with the arguments presented in \cite{DMR15}). These spurious effects are not so pronounced in the TF-CGLS restoration, as multiplication by $\Amp$ enforces some symmetry in the original problem; it should be noted that, for this experiment, both GMRES and TF-CGLS have approximately the same computational cost. The best reconstruction is delivered by RP-GMRES, which has however a higher computational cost.
\item \textbf{Atmospheric blur.} The test data for this experiment are displayed in Figure \ref{fig:satset}. The PSF, of size $256\times 256$ pixels and available within \cite{RestTools}, models a realistic atmospheric blur. Reflective boundary conditions are imposed, so that multiplications with $A^T$ can be easily computed. The noise level is $\weps = 5\cdot 10^{-2}$. Figure \ref{fig:satrec} shows the best restorations achieved by the GMRES, the TF-CGLS, and the CGLS methods; relative errors and the corresponding number of iterations are displayed in the caption. Also for this test problem, the (TF-)CGLS methods deliver much better solutions than the GMRES methods. As in the previous examples, TF-CGLS proves to be much more efficient than CGLS, as its computational cost is dominated by 18 matrix-vector products (versus the 80 matrix-vector products required by CGLS).
\begin{figure}[tbp]
\centering
\begin{tabular}{ccc}
\hspace{-1.2cm}\textbf{{\small {exact}}} &
\hspace{-1.2cm}\textbf{{\small {PSF}}} &
\hspace{-1.2cm}\textbf{{\small {corrupted}}} \\
\hspace{-1.2cm}\includegraphics[width=5.2cm]{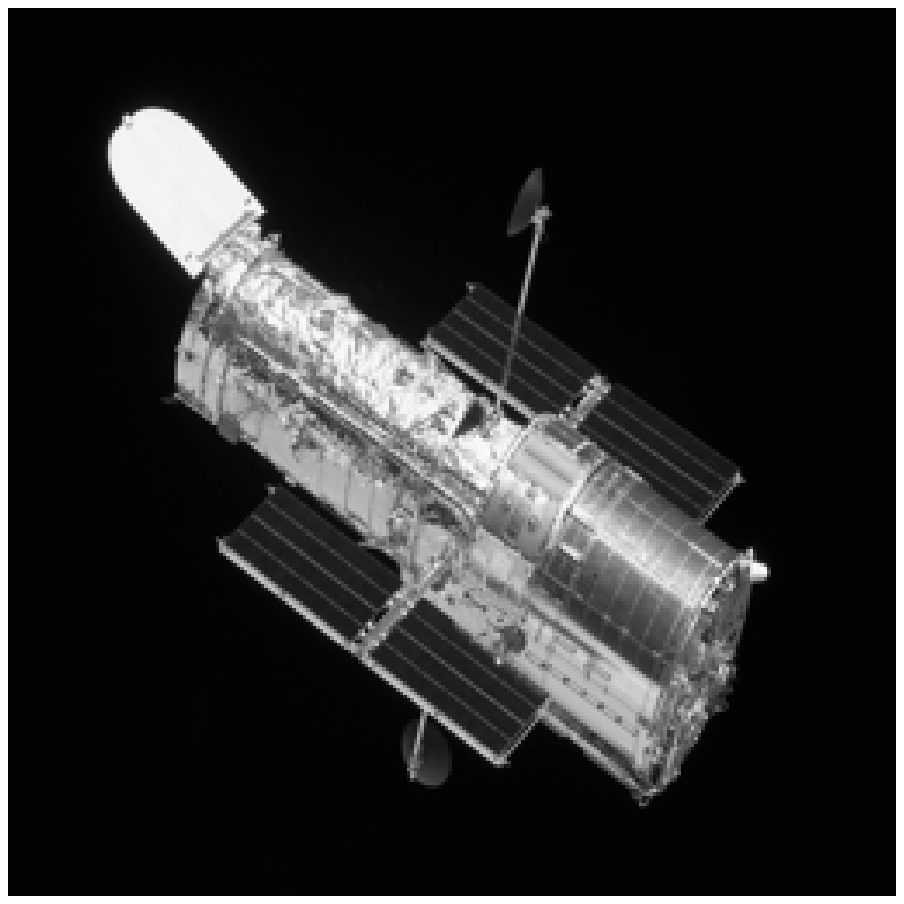} &
\hspace{-1.2cm}\includegraphics[width=5.2cm]{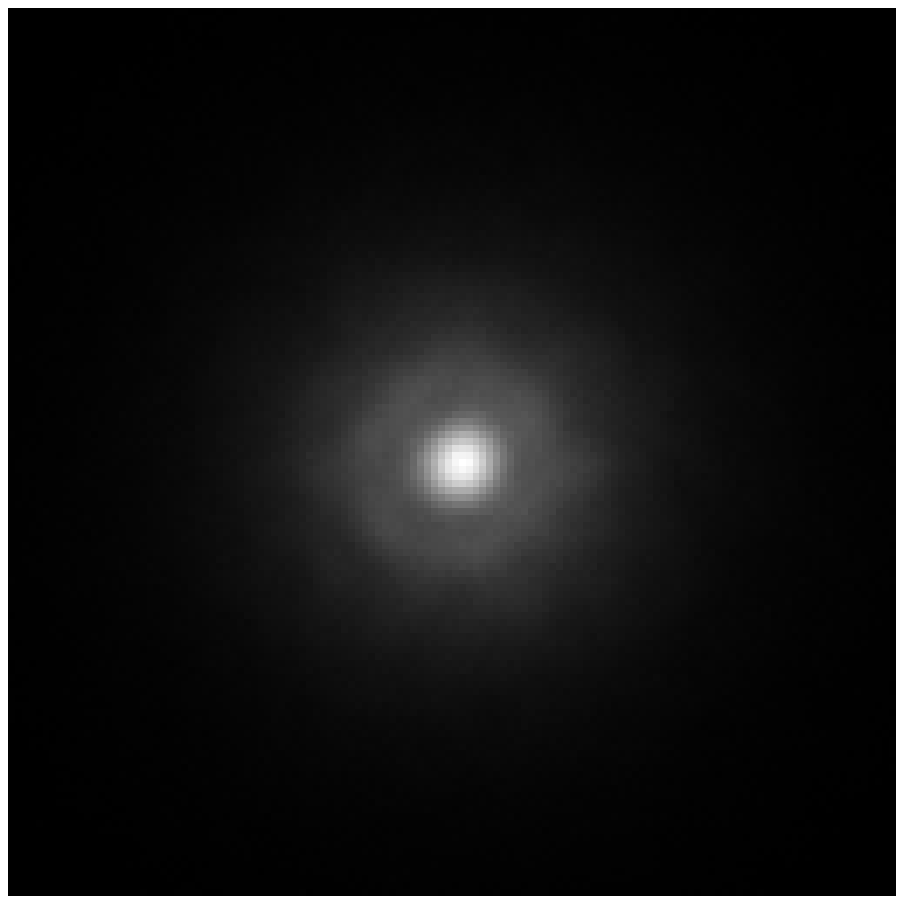} &
\hspace{-1.2cm}\includegraphics[width=5.2cm]{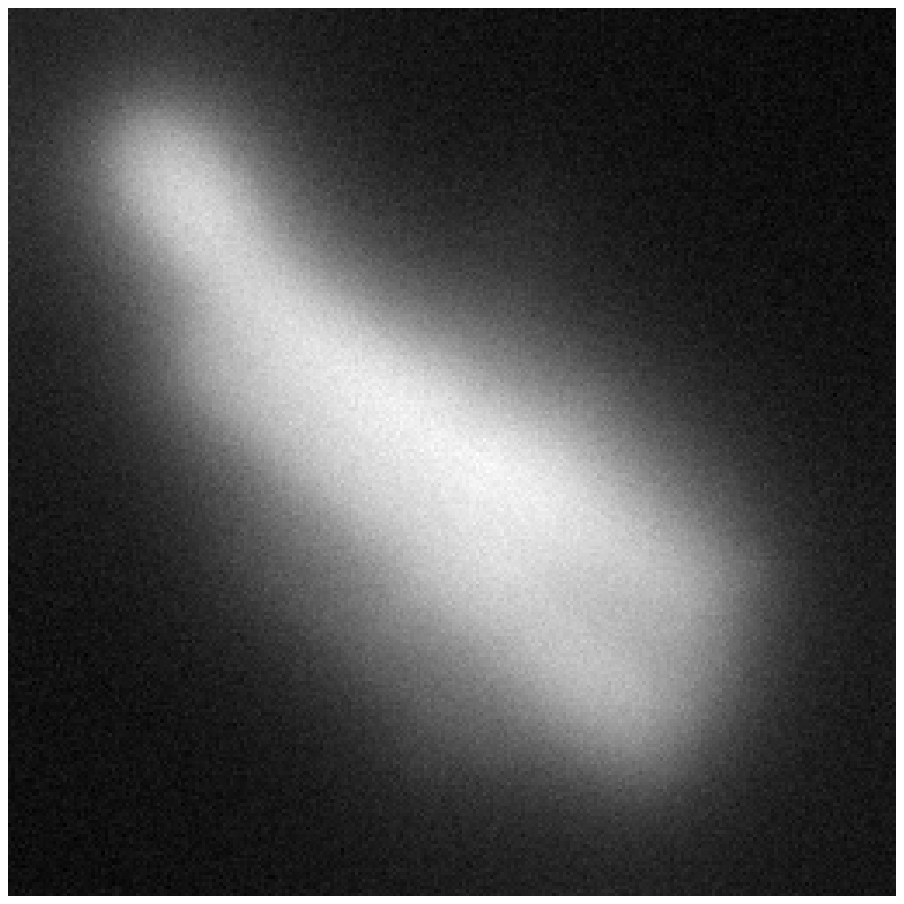}
\end{tabular}%
\caption{From left to right: exact image; blow-up ($200\%$) of the anisotropic Gaussian PSF; blurred and noisy available image, with $\weps=5\cdot 10^{-2}$.}
\label{fig:satset}
\end{figure}
\begin{figure}[tbp]
\centering
\begin{tabular}{ccc}
\hspace{-1.2cm}\textbf{{\small {GMRES}}} &
\hspace{-1.2cm}\textbf{{\small {TF-CGLS}}} &
\hspace{-1.2cm}\textbf{{\small {CGLS}}} \\
\hspace{-1.2cm}\includegraphics[width=5.2cm]{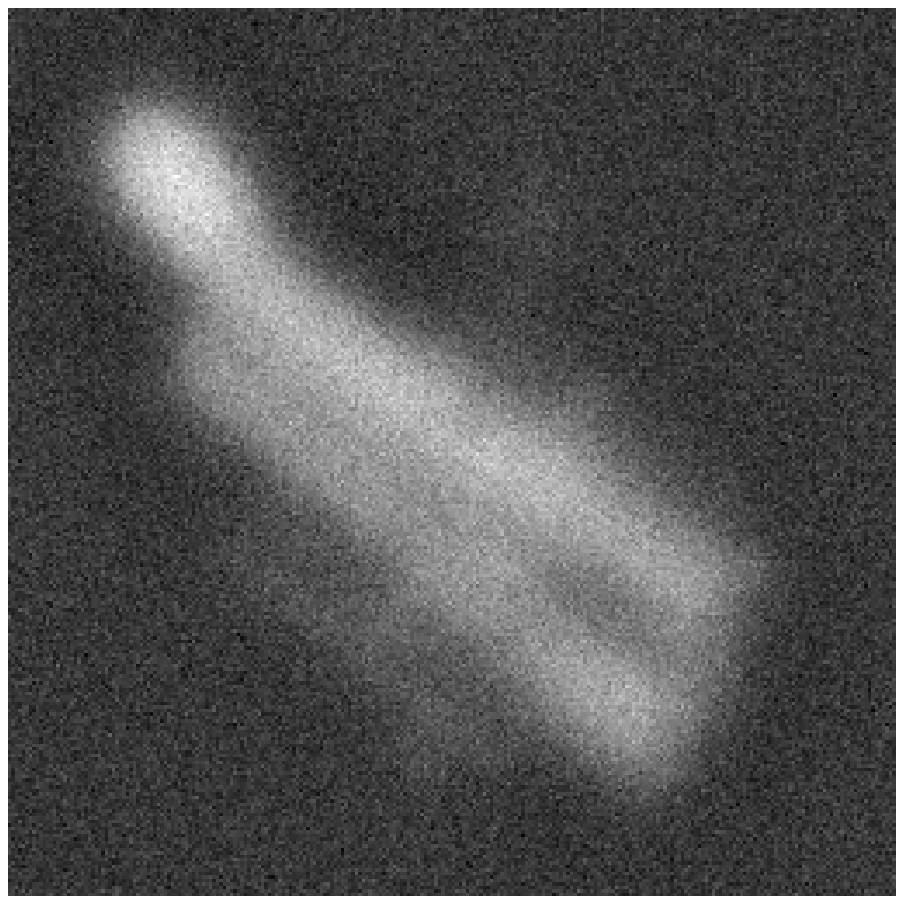} &
\hspace{-1.2cm}\includegraphics[width=5.2cm]{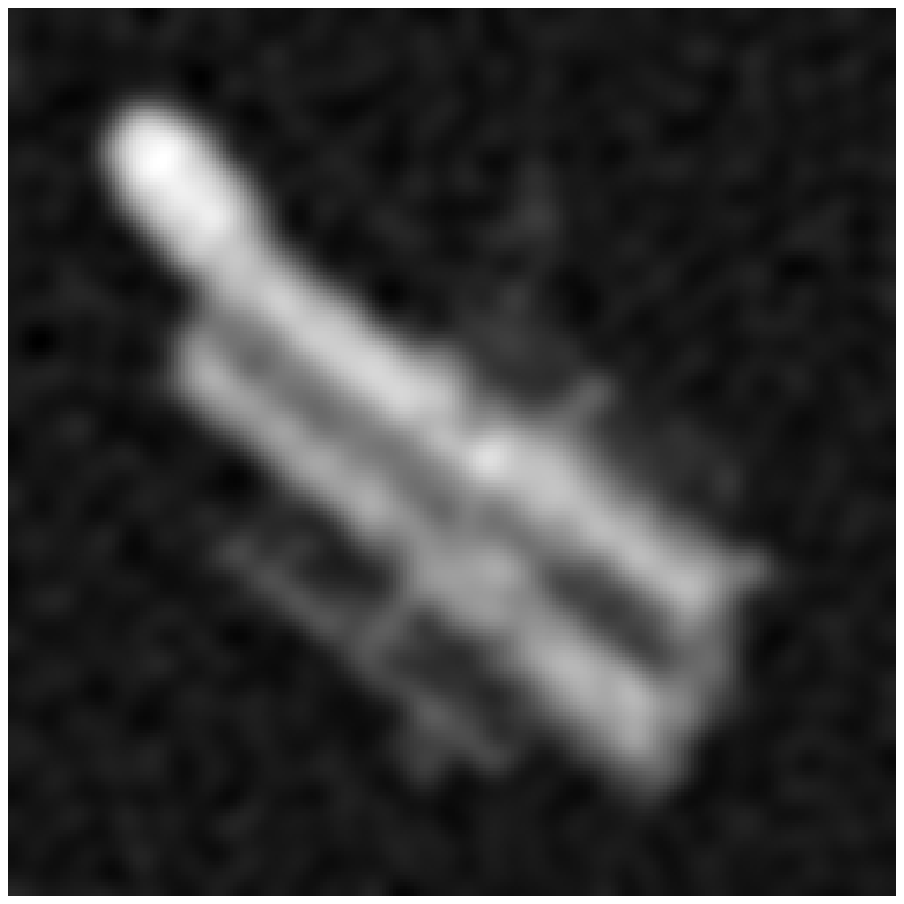} &
\hspace{-1.2cm}\includegraphics[width=5.2cm]{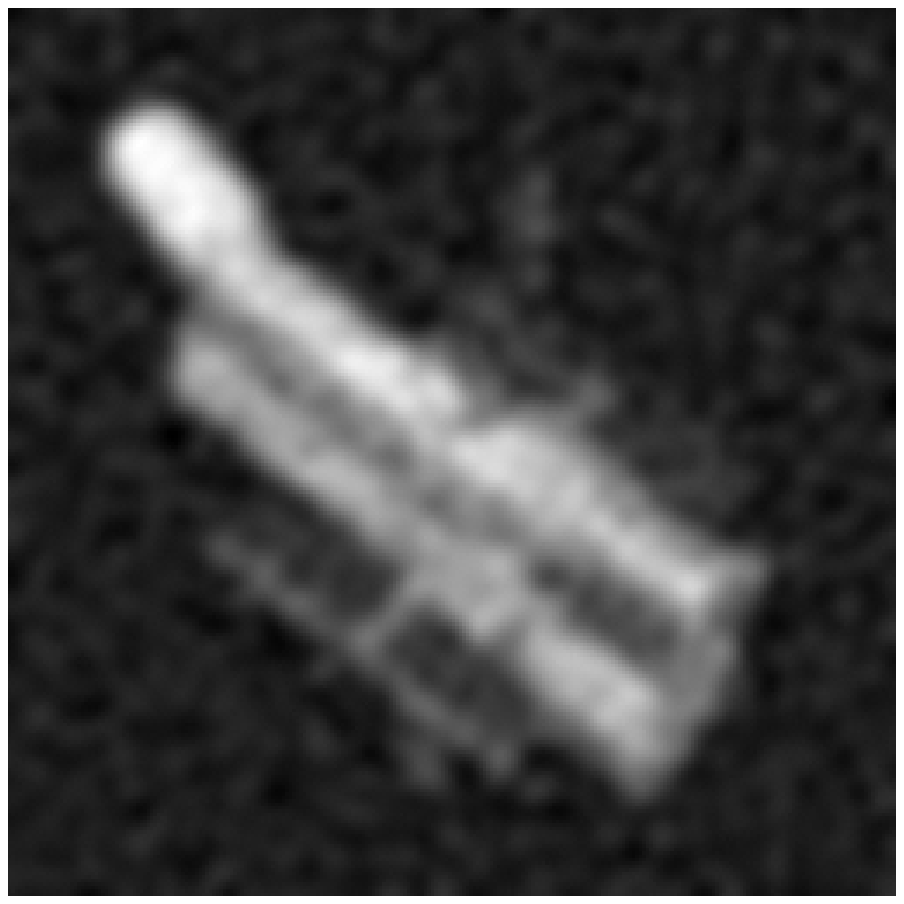}\vspace{-0.5cm}\\
\hspace{-1.2cm}\includegraphics[width=5.2cm]{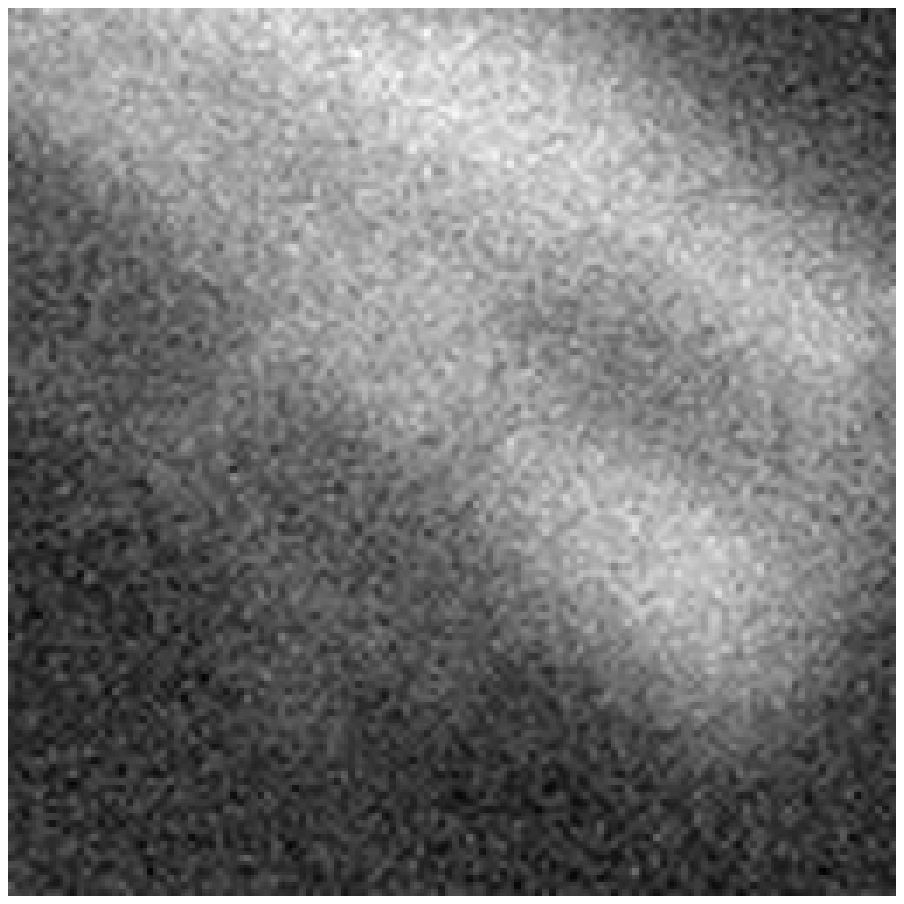} &
\hspace{-1.2cm}\includegraphics[width=5.2cm]{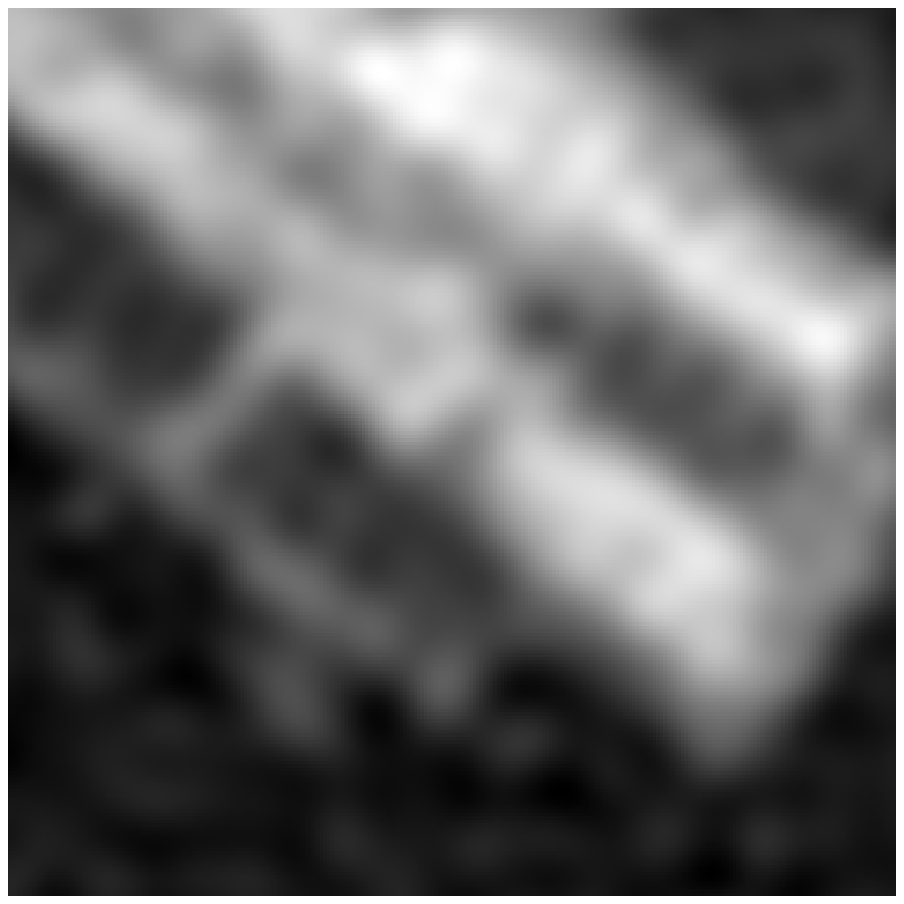} &
\hspace{-1.2cm}\includegraphics[width=5.2cm]{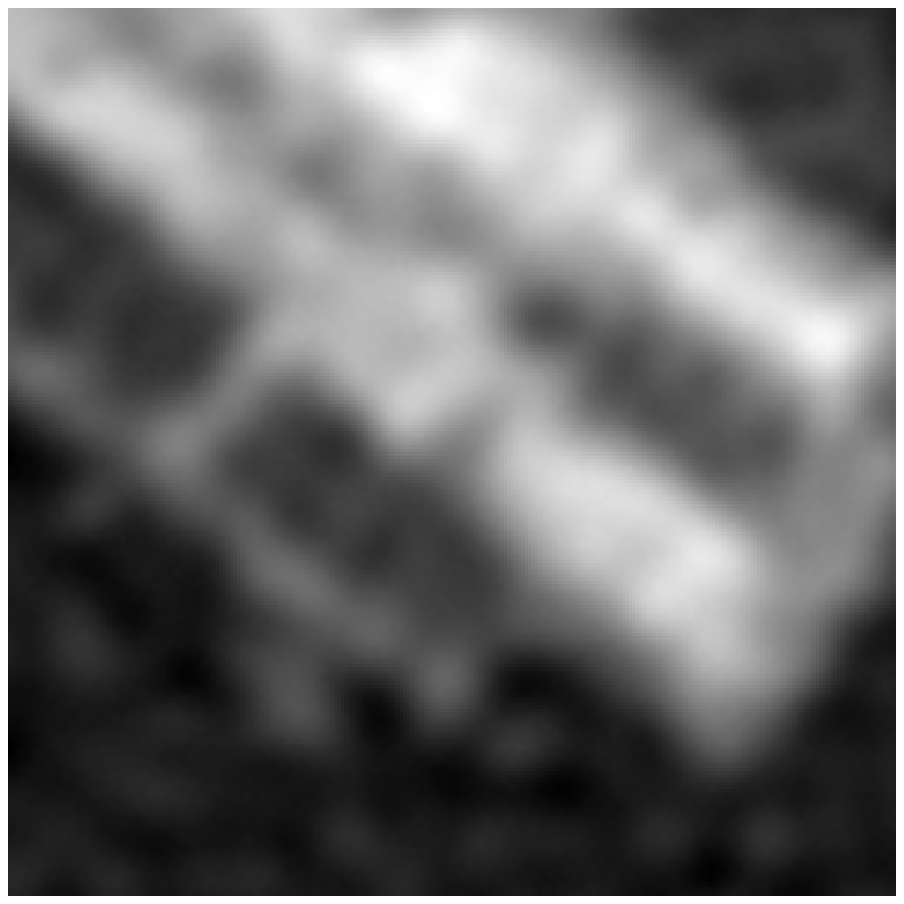}
\end{tabular}
\caption{The lower row displays blow-ups ($200\%$) of the restored images in the upper row. From left to right: standard GMRES method ($4.0018\cdot 10^{-1}$, $m=4$); TF-CGLS method ($2.7855\cdot 10^{-1}$, $m=18$, $k=5$); CGLS method ($2.7619\cdot 10^{-1}$, $m=40$).}
\label{fig:satrec}
\end{figure}

\end{enumerate}

\section{Conclusions}

\label{sect:final}
This paper presented a new class of transpose-free CG-like methods, which can be suitably and efficiently employed to regularize large-scale linear inverse problems. These methods are particularly meaningful when the transpose of the coefficient matrix is not easily available, and they represent a very valid alternative to the standard GMRES method, as they can successfully handle situations where the latter performs badly (e.g., when the SVD components of the original matrix $A$ are heavily mixed in the GMRES approximation subspace, or when the GMRES solutions diverge). When compared with CGLS or with other transpose-free solvers, the new TF-CGLS methods have a similar performance, with considerable computational savings.

Extensions of these transpose-free CG-like methods can be considered in order to incorporate some additional Tikhonov regularization at each iteration (so to further regularize and stabilize their behavior). This class of transpose-free CG-like methods can be probably employed to solve well-posed problems, as well, provided that some insight on the SVD behavior of the projected problems is available




\begin{thebibliography}{}
%
\bibitem{BerNagy13}
Berisha, S., Nagy, J.G.: Iterative image restoration.
\newblock In: R.~Chellappa, S.~Theodoridis (eds.) Academic Press Library in
  Signal Processing, vol.~4, chap.~7, pp. 193--243. Elsevier (2014)
%
\bibitem{RRGMRESfirst}
Calvetti, D., Lewis, B., Reichel, L.: {GMRES}-type methods for inconsistent systems.
\newblock {L}inear {A}lgebra {A}ppl. \textbf{316}, 157--169 (2000)
%
\bibitem{CLR02-2}
Calvetti, D., Lewis, B., Reichel, L.: On the regularizing properties of the
  {GMRES} method.
\newblock Numer. Math. \textbf{91}, 605--625 (2002)
%
\bibitem{ATfirst}
Calvetti, D., Morigi, S., Reichel, L., Sgallari, F.: Tikhonov regularization
  and the {L}-curve for large discrete ill-posed problems.
\newblock J. Comput. Appl. Math. \textbf{123}, 423--446 (2000)
%
\bibitem{DMR15}
Donatelli, M., Martin, D., Reichel, L.: Arnoldi methods for image deblurring with anti-reflective boundary conditions.
\newblock Appl. Math. Comput. \textbf{253}, 135--150 (2015)
%
\bibitem{iDPC}
Gazzola, S., Novati, P.: Inheritance of the discrete Picard condition in Krylov subspace methods.
\newblock BIT \textbf{56}(3), 893--918 (2016)
%
\bibitem{Survey}
Gazzola, S., Novati, P., Russo, M.R.: {O}n {K}rylov projection methods and
  {T}ikhonov regularization.
\newblock Electron.~Trans.~Numer.~Anal. \textbf{44}, 83--123 (2015).
%
\bibitem{Han95}
Hanke, M.: {C}onjugate {G}radient {T}ype {M}ethods for {I}ll-{P}osed
  {P}roblems.
\newblock Longman, Essex, UK (1995)
%
\bibitem{LancReg}
Hanke, M.: On Lanczos based methods for the regularization of discrete ill-posed problems.
\newblock BIT \textbf{41}, 1008--1018 (2001).
%
\bibitem{RegT}
Hansen, P.C.: {R}egularization {T}ools: {A} {M}atlab package for analysis and
  solution of discrete ill-posed problems.
\newblock {N}umerical {A}lgorithms \textbf{6}, 1--35 (1994)
%
\bibitem{PCH98}
Hansen, P.C.: Rank-deficient and discrete ill-posed problems.
\newblock SIAM, Philadelphia, PA (1998)
%
\bibitem{PCH10}
Hansen, P.C.: Discrete inverse problems.
\newblock SIAM, Philadelphia, PA (2010)
%
\bibitem{SN}
Hansen, P.C., Jensen, T.K.: Noise propagation in regularizing iterations for image deblurring.
\newblock Electron. Trans. Numer. Anal. \textbf{31}, 204--220 (2008)
%
\bibitem{JH07}
Jensen, T.K., Hansen, P.C.: Iterative regularization with minimum-residual methods.
\newblock BIT \textbf{47}, 103--120 (2007)
%
\bibitem{Mo}
Moret, I.: A note on the superlinear convergence of {GMRES}.
\newblock {SIAM} {J}. {N}umer. {A}nal. \textbf{34}, 513--516 (1997)
%
\bibitem{RestTools}
Nagy J.G., Palmer, K.M., Perrone, L.: {I}terative methods for image deblurring: a {M}atlab object oriented Approach.
\newblock Numer. Algorithms \textbf{36}, 73--93 (2004)
%
\bibitem{N}
Novati, P.: {S}ome properties of the {A}rnoldi based methods for linear ill-posed problems.
\newblock To appear (2017)
%
\bibitem{NR14}
Novati, P., Russo, M.R.: A {GCV} based {A}rnoldi-{T}ikhonov regularization
  method.
\newblock BIT \textbf{54}(2), 501--521 (2014)
%
\bibitem{O'LS}
O'Leary, D.P., Simmons, J.A.: A bidiagonalization-regularization procedure for
  large scale discretizations of ill-posed problems.
\newblock SIAM J. Sci. Stat. Comp. \textbf{2}(4), 474--489 (1981)
%
\bibitem{Ri}
Ringrose, J.R.: Compact non-self-adjoint operators.
\newblock {V}an {N}ostrand {R}einhold {C}ompany, London (1971)
%
\bibitem{Saad}
Saad, Y.: {I}terative {M}ethods for {S}parse {L}inear {S}ystems, 2nd {E}d.
\newblock SIAM, Philadelphia, PA (2003)
%
\bibitem{ARBCfirst}
Serra-Capizzano, S.: {A note on anti-reflective boundary conditions and fast deblurring models}.
\newblock {SIAM J. Sci. Comput.} \textbf{25}, 1307--1325 (2003)
%
\end{thebibliography}


\end{document}